\documentclass[a4paper]{article}
\usepackage{amsmath,amsthm,amssymb,bm}
\usepackage{mathscinet}

\makeatletter
    
    \@addtoreset{equation}{section}
\makeatother

\newtheorem{Thm}{Theorem}[section]
\newtheorem{Lem}[Thm]{Lemma}
\newtheorem{Cor}[Thm]{Corollary}
\newtheorem{Prop}[Thm]{Proposition}
\newtheorem{Rem}[Thm]{Remark}
\newtheorem{Def}[Thm]{Definition}
\newtheorem*{Thm2}{Theorem A}
\newtheorem*{Thm3}{Theorem B}
\newtheorem*{Thm4}{Theorem C}
\newtheorem*{Cor2}{Corollary A}
\newtheorem*{Cor3}{Corollary B}
\newtheorem*{Cor4}{Corollary C}
\newtheorem*{Prop2}{Proposition}

\newcommand{\R}{\mathbb{R}}           
\newcommand{\C}{\mathbb{C}}           
\newcommand{\Z}{\mathbb{Z}}           
\newcommand{\mP}{\mathbb{P}}
\newcommand{\mB}{\mathbb{B}}
\newcommand{\mcB}{\mathcal{B}}
\newcommand{\mcF}{\mathcal{F}}
\newcommand{\height}{\mathrm{ht} \,}

\newcommand{\Img}{\text{Im} \,}
\newcommand{\id}{\mathrm{id}}
\newcommand{\sh}{\mathrm{sh}}
\newcommand{\ch}[1]{\mathrm{ch}_{#1}\,}
\newcommand{\cl}{\mathrm{cl}}

\newcommand{\fa}{{\mathfrak a}}             
\newcommand{\fb}{{\mathfrak b}}
\newcommand{\fg}{{\mathfrak g}}
\newcommand{\fh}{{\mathfrak h}}

\newcommand{\fn}{{\mathfrak n}}

\newcommand{\hfg}{\hat{\fg}}
\newcommand{\hfh}{\hat{\fh}}
\newcommand{\hfb}{\hat{\fb}}

\newcommand{\hfn}{\hat{\fn}}
\newcommand{\hP}{\hat{P}}

\newcommand{\hI}{\hat{I}}
\newcommand{\hW}{\hat{W}}

\newcommand{\ga}{\alpha}
\newcommand{\gb}{\beta}
\newcommand{\gk}{\kappa}

\newcommand{\gl}{\lambda}
\newcommand{\gL}{\Lambda}
\newcommand{\gd}{\delta}
\newcommand{\gD}{\Delta}

\newcommand{\gt}{\theta}

\renewcommand{\gg}{\gamma}
\newcommand{\gs}{\sigma}

\newcommand{\gee}{\varepsilon}
\newcommand{\gp}{\varpi}
\newcommand{\gph}{\varphi}
\newcommand{\gi}{\iota}

\renewcommand{\hat}{\widehat}
\newcommand{\ol}{\overline}

\newcommand{\mC}{\mathcal{C}}
\newcommand{\Cg}{\mathcal{C}{\fg}}
\newcommand{\Cgd}{\mathcal{C}{\fg}_d}

\newcommand{\fhd}{\fh_d}

\newcommand{\lpishr}{\gD^{\sh}}
\newcommand{\D}{\mathcal{D}}
\newcommand{\ul}[1]{\underline{#1}}
\newcommand{\mPint}{\mP_{\mathrm{int}}}
\newcommand{\te}{\tilde{e}}
\newcommand{\tf}{\tilde{f}}
\newcommand{\0}{\mathbf{0}}
\newcommand{\wt}{\mathrm{wt}}
\newcommand{\Deg}{\mathrm{Deg}}

\newcommand{\mugs}{(\ul{\mathstrut \mu}, \ul{\mathstrut \gs})}

\begin{document}
\title{Weyl modules, Demazure modules and finite crystals for non-simply laced type}
\author{Katsuyuki Naoi}
\date{}

\maketitle

\begin{abstract}
  We show that every Weyl module for a current algebra has a filtration whose successive quotients are 
  isomorphic to Demazure modules,
  and that the path model for a tensor product of level zero fundamental representations
  is isomorphic to a disjoint union of Demazure crystals.
  Moreover, we show that the Demazure modules appearing in these two objects coincide exactly.
  Though these results have been previously known in the simply laced case, they are new in the non-simply laced case.  
\end{abstract}

\section{Introduction}
  In this article, we study finite-dimensional representations of a current algebra,
  crystals
  for a quantum affine algebra, and connections between them. 
  First, we begin with an introduction of the results concerning finite-dimensional representations of a current algebra.
  Let $\fg$ be a complex simple Lie algebra. 
  The study of representations of its current algebra $ \Cg= \fg \otimes \C[t]$, 
  together with that of the loop algebra $\fg \otimes \C[t, t^{-1}]$,
  has been the subjects of many articles. 
  For example, see \cite{MR2290922,MR2351379,MR2271991,MR2238884,MR1850556,MR2428305,MR1729359,MR2323538}.

  Among finite-dimensional $\Cg$-modules, Weyl modules are especially important.
  The notion of a Weyl module was originally introduced in \cite{MR1850556} 
  for a loop algebra as a module having some universal property, 
  and defined similarly for a current algebra in \cite{MR2271991}.
  We denote by $W(\gl)$ the Weyl module for $\Cg$ associated with $\gl \in P_+$,
  where $P_+$ is the set of dominant integral weights of $\fg$.
  $W(\gl)$ has a natural $\Z$-grading.
  In order to study its $\Z$-graded structure,
  we consider $W(\gl)$ as a $\Cgd (= \Cg \oplus \C d)$-module in this article, where $d$ is the degree operator.
    
  There is another important class of finite-dimensional $\Cgd$-modules called Demazure modules.
  Let $\fb$ be a Borel subalgebra of $\fg$, and $\hfb = \fb + \C K + \C d + \fg \otimes t\C[t]$ the Borel subalgebra 
  of the affine Lie algebra $\hat{\fg}$, where $K$ is the canonical central element.
  A Demazure module is, by definition, a $\hat{\fb}$-submodule of an irreducible highest weight $\hfg$-module
  generated by an extremal weight vector.
  Among Demazure modules, we are mainly interested in $\Cgd$-stable ones, 
  which are denoted by $\D(\ell, \gl)[m]$ for $\ell \in \Z_{> 0}$, $\gl \in P_+$,
  and $m \in \Z$ (for precise definitions, see Subsection \ref{Demazure modules}).
   
  When $\fg$ is simply laced, the Weyl module $W(\gl)$ is known to be isomorphic to the Demazure module $\D(1, \gl)[0]$.
  This result was proved for $\fg = \mathfrak{sl}_n$ in \cite{MR2271991},
  and for general simply-laced $\fg$ in \cite{MR2323538}. 
  From this result, we can obtain a lot of informations about the structure of Weyl modules in the simply laced case.
  Then it is natural to ask what happens in the non-simply laced case.
  As stated in \cite{MR2323538}, a similar result does not hold any longer in 
  this case.
  
  One of the aims in this article is to generalize the above result in the non-simply laced case.
  To state our result, we prepare some notation.
  Assume that $\fg$ is non-simply laced, fix a Cartan subalgebra $\fh$ of $\fg$, 
  and let $\gD \in \fh^*$ be the root system of $\fg$.
  We denote by $\gD^{\sh}$ the root subsystem of $\gD$ generated by short simple roots,
  and by $\fg^{\sh}$ the simple Lie subalgebra corresponding to $\gD^{\sh}$.
  Put $\fh^{\sh} = \fh \cap \fg^{\sh}$, which is a Cartan subalgebra of $\fg^{\sh}$, 
  and $\Cgd^{\sh}= \fg^{\sh} \otimes \C[t] \oplus \C d$.
  Denote by $\ol{P}_+ \subseteq (\fh^{\sh})^*$ the set of dominant integral weights of $\fg^{\sh}$,
  and by $\D^{\sh}(\ell, \nu)[m]$ ($\ell \in \Z_{>0}$, $\nu \in \ol{P}_+$, $m \in \Z$) Demazure modules for $\Cgd$. 
  We define a positive integer $r$ by 
  \[ r = \begin{cases} 2 & \text{if $\fg$ is of type $B_n, C_n$ or $F_4$}, \\
                       3 & \text{if $\fg$ is of type $G_2$}. 
         \end{cases}
  \]
  
  For $\gl \in P_+$, denote by $\ol{\gl} \in \ol{P}_+$ the image of $\gl$ under 
  the canonical projection $\fh^* \to (\fh^{\sh})^*$.
  By the Joseph's result in \cite{MR2214249}, 
  $\D^{\sh}(1,\ol{\gl})[0]$ has a $\Cgd^{\sh}$-module filtration $0=D_0 \subseteq D_1 \subseteq 
  \dots \subseteq D_k = \D^{\sh}(1,\ol{\gl})[0]$
  such that each subquotient $D_i / D_{i-1}$ is isomorphic to $\D^{\sh}(r, \nu_i)[m_i]$ for some $\nu_i \in \ol{P}_+$
  and $m_i \in \Z_{\ge 0}$.
  Define a linear map $i_{\sh}\colon (\fh^{\sh})^* \to \fh^*$ by $\ga|_{\fh^{\sh}} \mapsto \ga$ for $\ga \in \gD^{\sh}$,
  and put $\gl' = \gl - i_{\sh}(\ol{\gl})$, $\mu_i = i_{\sh}(\nu_i) + \gl'$ for each $1 \le i \le k$.
  Then the following theorem is proved (Theorem \ref{Thm: module_main_theorem}):

  \begin{Thm2}
    The Weyl module $W(\gl)$ has a $\Cgd$-module filtration $0 = W_0 \subseteq W_1 \subseteq \cdots \subseteq W_k = W(\gl)$ 
    such that each subquotient $W_i /W_{i-1}$ is isomorphic to the Demazure module $\D(1,\mu_i)[m_i]$.
  \end{Thm2}
  
  When a Weyl module is isomorphic to a Demazure module, the successive quotient of a trivial filtration is 
  of course isomorphic to a Demazure module.
  Hence this is a kind of generalization of the result in the simply laced case.
  
  It should be remarked that, although the statement of Theorem A is purely that
  concerning modules of a current algebra, we need some results on crystals to prove this theorem.
  To be more precise, though we can show without using any crystal theory that the Weyl module has a filtration 
  such that each successive quotient is a homomorphic image of the Demazure module,
  we need to apply some results on crystals to verify that they are in fact isomorphic.
  
  Some corollaries concerning Weyl modules for non-simply laced $\fg$ are obtained from Theorem A,
  which have already been proved in the simply laced case (for $\fg = \mathfrak{sl}_n$ in \cite{MR2271991},
  and for general simply laced $\fg$ in \cite{MR2323538}).
  First, the following corollary is obtained (Corollary \ref{Cor: Main_Corollary1} (i)):

  \begin{Cor2}
    Let $\gl = \sum_{i} \gl_i \varpi_i \in P_+$, where $\varpi_i$ denote the fundamental weights of $\fg$.
    Then we have 
      \[ \dim W(\gl) = \prod_{i}\dim W(\varpi_i)^{\gl_i}.
      \]
  \end{Cor2}

  By the same way as \cite{MR2271991} or \cite{MR2323538}, 
  the dimension conjecture of Weyl modules for a loop algebra, 
  which was conjectured (and proved for $\fg = \mathfrak{sl}_2$) in \cite{MR1850556}, is deduced from Corollary A. 
  It should be remarked that
  the dimension conjecture for general $\fg$ also can be proved using the global basis theory. 
  (This was pointed out by Hiraku Nakajima.   
  The proof is not written in any literature, but a brief sketch of this proof can be found 
  in the introduction of \cite{MR2323538}.)
  Our approach is quite different from this proof.

  The second corollary is as follows (Corollary \ref{Cor: Main_Corollary1} (ii)):
  
  \begin{Cor3}
    Let $\gl \in P_+$, and $\gl_1, \dots, \gl_\ell \in P_+$ a sequence of elements satisfying $\gl = \gl_1 + \dots + \gl_\ell$.
    For arbitrary pairwise distinct complex numbers $c_1,\dots,c_\ell$, we have 
    \[ W(\gl) \cong W(\gl_1)_{c_1} *\dots * W(\gl_\ell)_{c_\ell},
    \]
    where $*$ denotes the fusion product introduced in {\normalfont\cite{MR1729359}}, 
    and $c_1,\dots,c_\ell$ are parameters used to define the fusion product.
  \end{Cor3}

  Note that this corollary implies that the fusion product of Weyl modules is associative and independent 
  of the parameters $c_1,\dots,c_\ell$.
  This statement was conjectured in \cite{MR1729359} for more general modules for $\Cg$.
  
  Next, we introduce our results on crystals.
  Let $U_q(\hfg)$ be the quantum affine algebra associated with $\hfg$, and $U_q'(\hfg)$ the one without the degree operator.
  Crystals we mainly study in this article are realized by path models, which were originally introduced by Littelmann
  \cite{MR1253196, MR1356780}.
  Let $\hP$ denote the weight lattice of $\hfg$.
  A path with weight in $\hP$ is, by definition, a piecewise linear, continuous map 
  $\pi\colon [0,1] \to \R \otimes_{\Z} \hP$ such that $\pi(0) = 0$
  and $\pi(1) \in \hP$. Let $\mP$ denote the set of paths with weight in $\hP$,
  which has a $U_q(\hfg)$-crystal structure \cite{MR1253196, MR1356780}. 
  Let $\gl \in \hP$ be a level zero weight that is dominant integral for $\fg$,
  and $\mB_0(\gl)$ the connected component of $\mP$ containing the straight line path $\pi_{\gl}\colon \pi_{\gl}(t) = t\gl$.
  Put $\hP_{\cl} =\hP / \Z \gd$ where $\gd$ is the indivisible null root.
  By projecting $\mB_0(\gl)$ to $\R \otimes_{\Z} \hP_{\cl}$, 
  a finite $U_q'(\hfg)$-crystal $\mB(\gl)_{\cl}$ is obtained, and
  Naito and Sagaki verified in \cite{MR2146858, MR2199630} that this $\mB(\gl)_{\cl}$ is isomorphic 
  to the tensor product of the crystal bases of 
  level zero fundamental $U_q'(\hfg)$-representations introduced by Kashiwara in \cite{MR1890649}.
  Note that elements of $\mB(\gl)_{\cl}$ have only $\hP_{\cl}$-weights by definition. 
  In this article, we define $\hP$-weights on these elements using the degree function introduced in \cite{MR2474320}. 
  These $\hP$-weights are important for stating our theorem stated below. 
  
  Demazure modules have counterparts in the crystal theory.
  Let $V_q(\gL)$ be the irreducible highest weight $U_q(\hfg)$-module of highest weight $\gL$ and $\mcB(\gL)$ 
  its crystal basis.
  Similarly as the classical case, a quantized version of a Demazure module is defined as a submodule of $V_q(\gL)$.
  For each Demazure module, Kashiwara defined in \cite{MR1240605} a subset of $\mcB(\gL)$ called a Demazure crystal,
  and proved that there exist strong connections between a Demazure module and the corresponding Demazure crystal
  (for example, the character of a Demazure module coincides with the weight sum of the Demazure crystal).
  We denote by $\mcB(\ell, \gl)[m]$ the Demazure crystal corresponding to (the quantized version of) $\D(\ell, \gl)[m]$.
  
  Now let us state our second main theorem (Theorem \ref{Thm: crystal_main_theorem}). 
  Assume that $\fg$ is non-simply laced, 
  and let $\mu_i\in P_+$ and $m_i \in \Z_{\ge 0}$ $(1 \le i \le k)$ be the elements defined just above Theorem A.
  We denote by $b_{\gL}$ the highest weight element of $\mcB(\gL)$, and by $\gL_0$ the fundamental weight of $\hfg$ 
  associated with the additional node of the extended Dynkin diagram of $\fg$:

  \begin{Thm3}
    $\mcB(\gL_0) \otimes \mB(\gl)_{\cl}$ is isomorphic as a $U_q'(\hfg)$-crystal 
    to the direct sum of crystal bases of irreducible 
    highest weight $U_q(\hfg)$-modules.
    Moreover, the restriction of the given isomorphism on $b_{\gL_0} \otimes \mB(\gl)_{\cl}$ preserves the $\hP$-weights,
    and maps $b_{\gL_0} \otimes \mB(\gl)_{\cl}$ onto $\coprod_{1 \le i \le k} \mcB(1, \mu_i)[m_i]$.
  \end{Thm3}

  Note that this theorem is a generalization of the following statement proved in \cite{MR2323538}:
  if $\fg$ is simply-laced, the crystal graph of $b_{\gL_0} \otimes \mB(\gl)_{\cl}$ coincides with the 
  graph of $\mcB(1,\gl)[0]$.
  We remark that, similarly as Theorem A, we need not only results on crystals
  but also some results on Weyl modules to show Theorem B. 
  
  In order to prove Theorem B, we show the following proposition in advance (Corollary \ref{Cor: decomposition} and 
  Proposition \ref{Prop: Demazure_crystal_decomposition1}):

  \begin{Prop2} Let $\gL$ be an arbitrary dominant integral weight of $\hfg$. \\
    {\normalfont(i)} $\mcB(\gL) \otimes \mB_0(\gl)$ is isomorphic as a $U_q(\hfg)$-crystal
      to a direct sum of the crystal bases of irreducible highest weight $U_q(\hfg)$-modules. \\
    {\normalfont(ii)} 
      The isomorphism given in {\normalfont(i)} maps the subset $b_{\gL} \otimes \mB_0(\gl)$
      onto a disjoint union of some Demazure crystals.
  \end{Prop2}
  
  Since it is known that, at least in some cases, $\mB_0(\gl)$ is related to the crystal basis of 
  some level zero extremal weight $U_q(\hfg)$-module (\cite{MR2199630,MR2407814}),
  this proposition itself seems important and interesting.

  As seen from Theorem A and B, the Weyl module $W(\gl)$ and the crystal
  $\mB(\gl)_{\cl}$ have some connections.
  In particular, the following theorem is proved (Theorem \ref{Thm: Main_Theorem}):

  \begin{Thm4}\label{Thm:Main_Theorem_intro}
  For $\gl \in P_+$, we have 
  \[ \ch{\fh_d}W(\gl) = \sum_{\eta \in \mB(\gl)_{\cl}} e\big(\wt_{\hP}(\eta)\big).
  \]
  \end{Thm4}

  From our theorems, some connections are found between a Weyl module and polynomials defined in the crystal theory. 
  Let $W_q(\varpi_i)$ denote the level-zero fundamental $U_q'(\hfg)$-representation and
  $\mcB\big(W_q(\varpi_i)\big)$ its crystal basis.
  Let 
  $\mathbf{i} = (i_1,\dots,i_{\ell})$
  be a sequence of indices of simple roots of $\fg$, and put 
  \[ W_{\mathbf{i}} = W_q(\varpi_{i_1}) \otimes \dots \otimes W_q(\varpi_{i_\ell}), \ \ \ 
     \mcB_{\mathbf{i}} = \mcB\big(W_q(\varpi_{i_1})\big) \otimes \dots \otimes \mcB\big(W_q(\varpi_{i_\ell})\big).
  \]
  For each $\mu \in P_+$, the fermionic form $M(W_{\mathbf{i}},\mu,q)$ 
  and the classically restricted one-dimensional sum $X(\mcB_{\mathbf{i}}, \mu,q)$ are defined 
  as \cite{MR1745263,MR1903978}. 
  Using the results of Di Francesco and Kedem in \cite{MR2428305}, Naito and Sagaki in 
  \cite{MR2474320} and ours, the following corollary is shown (Corollary \ref{Cor: final_Cor}), 
  which implies a special case of the $X=M$ conjecture presented in \cite{MR1745263,MR1903978}:

  \begin{Cor4}
    We have
    \begin{align*} 
      M(W_{\mathbf{i}}, \mu ,q) &= q^{-D_{\mathbf{i}}^{\mathrm{ext}}}X(\mcB_{\mathbf{i}}, \mu ,q) \\
      & = [W(\gl) : V_{\fg}(\mu)]_{q^{-1}},
    \end{align*}
    where $D_{\mathbf{i}}^{\mathrm{ext}}$ denotes a certain constant defined in {\normalfont \cite{MR2474320}}, 
    $V_{\fg}(\mu)$ denotes the irreducible $\fg$-module of highest weight 
    $\mu$ and $[W(\gl) : V_{\fg}(\mu)]_{q}$ denotes the $\Z$-graded multiplicity of $V_{\fg}(\mu)$ in $W(\gl)$.
  \end{Cor4}

  In \cite{MR2271991}, the authors proved that the $\Z$-graded multiplicity of $W(\gl)$ is equal to a Kostka polynomial 
  when $\fg = \mathfrak{sl}_n$.
  Since fermionic forms and classically restricted one-dimensional sums
  are known to coincide with Kostka polynomials in this case, 
  the above corollary is a generalization of their result.

  The plan of this article is as follows. In Section 2, we fix basic notation used in the article.
  In Section 3, we review some results on finite-dimensional representations of a current algebra, 
  almost of which have already been known.
  Section 4 is the main part in the first half of this article. 
  We give defining relations of a Demazure module of level $1$, and show using this defining relations 
  the existence of a filtration on a Weyl module whose successive quotients are homomorphic images of Demazure modules.

  In Section 5, we review the theory of path models.
  In Section 6, we show that $\mcB(\gL) \otimes \mB_0(\gl)$ is isomorphic to a direct sum of the crystal bases 
  of highest weight modules, and in Section 7 that the image of $b_{\gL} \otimes \mB_0(\gl)$ 
  under this isomorphism is a disjoint union of Demazure crystals. 
  In the final part of Section 7, we also show that $\mcB(\gL) \otimes \mB(\gl)_{\cl}$ has similar properties.
  From this, we see in particular that $b_{\gL_0} \otimes \mB(\gl)_{\cl}$ 
  decomposes to a disjoint union of some Demazure crystals.
  In Section 8, we study this decomposition of $b_{\gL_0} \otimes \mB(\gl)_{\cl}$ in more detail. 

  Finally in Section 9, we show the Theorem A, B and C, and Corollary A, B and C. \\

\noindent \textbf{Index of notation}  \\

We provide for the reader's convenience a brief index of the notation which is used repeatedly in this paper:\\
 
\noindent Section 2: $\fg$, $\fh$, $\fb$, $\gD$, $\gD_{\pm}$, $\Pi$, $\ga_i$, $I$, $\gt$, $\mathrm{ht}\, \ga$, $\ga ^{\vee}$,
    $\fg_{\ga}$,  $e_{\ga}$, $f_{\ga}$, $e_i$, $f_i$, $\fn_{\pm}$, $( \ , \ )$, $\nu$, $Q$, $Q_+$, $\varpi_i$, $P$, $P_+$, $W$,
    $s_{\ga}$, $s_i$, $w_0$, $r$, $V_{\fg}(\gl)$, $\hfg$, $K$, $d$, $\hfh$, $\hfn_+$, $\hfb$, $\hat{\gD}$, $\gd$, $\hat{\Pi}$, 
    $\hI$, $\gL_i$, $\hP$, $\hP_+$, $\hW$, $\Cg$, $\Cgd$, $\fhd$, $\Pi^{\sh}$, $I^{\sh}$, $\gD^{\sh}$, $\gD^{\sh}_{\pm}$,
    $Q^{\sh}$, $Q^{\sh}_+$, $W^{\sh}$, $\fh^{\sh}$, $\fg^{\sh}$, $\fn^{\sh}_{\pm}$, $\hfg^{\sh}$, $\Cg^{\sh}$, $\Cgd^{\sh}$, 
    $\hfh^{\sh}$, $\fhd^{\sh}$, $\ol{\gl}$, $i_{\sh}$, $\ol{P}$, $\ol{P}_+$, $U_q(\hfg)$, $U_q'(\hfg)$, $U_q(\hfg^{\sh})$, 
    $U_q'(\hfg^{\sh})$, $M_{\gl}$, $\wt_{H}$, $\ch{H}$, $P_S$.\\
\noindent Subsection 3.1: $W(\gl)$. \\
\noindent Subsection 3.2: $V(\gL)$, $V_w(\gL)$, $\D(\ell, \gl)[m]$, $\D(\ell, \gl)$. \\
\noindent Subsection 3.3: $W_{c_1}^1 * \cdots * W_{c_k}^k$.\\
\noindent Subsection 4.1: $M^{\gl}$. \\
\noindent Subsection 4.3: $V_q(\gL)$, $V_{q,w}(\gL)$, $\preceq$, $(M:\D(\ell, \gl)[m])$. \\
\noindent Subsection 4.4: $\D^{\sh}(\ell, \nu)[m]$, $\le$ on $\Z[\fhd^*]$. \\
\noindent Subsection 5.1: $[a,b]$, $\mP$, $(\ul{\mathstrut \mu}, \ul{\mathstrut \gs})$, $H_i^{\pi}$, 
  $m_i^{\pi}$, $\mPint$, $\te_i$, $\tf_i$, $\0$, $\wt$, $\gee_i$, $\gph_i$, $S_i$, $S_w$, $\pi_1 * \pi_2$, $\mB_1 * \mB_2$. \\
\noindent Subsection 5.2: $C(b)$, $\pi_{\gl}$, $\mB_0(\gl)$, $\mcB(\gL)$, 
  $b_{\gL}$, $\hP_{\cl}$, $\cl$, $\mP_{\cl}$, $\eta_{\xi}$, $\cl(\pi)$,
  $\mB(\gl)_{\cl}$, $W_q(\varpi_i)$, $\mcB(W_q(\varpi_i))$. \\
\noindent Subsection 5.3: $\iota(\pi)$, $\iota(\eta)$, $d_{\gl}$, $i_{\cl}$, $\pi_{\eta}$, $\Deg(\eta)$, $\wt_{\hP}$.\\
\noindent Subsection 6.1: $\hW_J$. \\
\noindent Subsection 6.2: $\mB_0(\gl)^{\gL}$. \\
\noindent Subsection 7.1: $\mcF_i(\mathcal{C})$, $\mcF_{\mathbf{i}}(\mathcal{C})$, $\mcB_w(\gL)$. \\
\noindent Subsection 7.2: $\hI_{\gL}$. \\
\noindent Subsection 7.3: $\mB(\gl)_{\cl}^{\gL}$. \\
\noindent Subsection 8.1: $\mcB(\ell, \gl)[m]$, $\gk$. \\
\noindent Subsection 8.2: $\gt^{\sh}$, $\ga_0^{\sh}$, $s_0^{\sh}$, $\hW^{\sh}$, $\hI^{\sh}$, $\hP^{\sh}$, $\mP^{\sh}$, 
  $\mPint^{\sh}$, $\te_i^{\sh}$, $\tf_i^{\sh}$, $H_i^{\sh, \pi}$, $m_i^{\sh, \pi}$, $\mB_0^{\sh}(\gl)$.

\section{Notation and elementary lemmas}\label{Notation}

Let $\fg$ be a complex simple Lie algebra.
Fix a Cartan subalgebra $\fh$ and a Borel subalgebra $\fb \supseteq \fh$ of $\fg$. 
Let $\gD \subseteq \fh^*$ denote the root system of $\fg$, and $\gD_+$ (resp.\ $\gD_-$) 
the set of positive roots (resp.\ negative roots) corresponding to $\fb$.
Denote by $\Pi = \{ \ga_1, \dots, \ga_n \}$ the set of simple roots and by $I = \{ 1, \dots, n \}$ its index set.
Let $\gt$ be the highest root of $\gD$.
For $\ga = \sum_i n_i \ga_i\in \gD$, we define the \textit{height} of $\ga$ by $\height \ga = \sum_i n_i$,
and denote the coroot of $\ga$ by $\ga ^{\vee} \in \fh$. 
Let $Q = \sum_i \Z \ga_i$ be the root lattice of $\fg$ and  $Q_+ = \sum_i \Z_{\ge 0} \ga_i \subseteq Q$.
We denote by $\varpi_i$ the fundamental weight corresponding to $\ga_i$.
Let $P = \sum_i \Z \varpi_i$ be the weight lattice of $\fg$, $P_+ = \sum_i \Z_{\ge 0} \gp_i\subseteq P$ 
the subset of dominant integral weights, and $W$ the Weyl group of $\fg$.
For $\ga \in \gD_+$ we denote by $s_{\ga} \in W$ the reflection associated with $\ga$, 
and abbreviate $s_i =s_{\ga_i}$.
We denote by $w_0$ the longest element of $W$.

Denote the root space associated with $\ga \in \gD$ by $\fg_{\ga}$, and set 
\[ \fn_{\pm} = \bigoplus_{\ga \in \gD_{\pm}} \fg_{\ga}.
\]
For each $\ga \in \gD_+$ we fix $e_{\ga} \in \fg_{\ga}$ and $f_{\ga} \in \fg_{-\ga}$ 
such that $[e_{\ga}, f_{\ga}] = \ga^{\vee}$,
and abbreviate $e_i = e_{\ga_i}, f_i = f_{\ga_i}$.
Let $( \ , \ )$ be the unique non-degenerate invariant symmetric bilinear form on $\fg$ normalized by 
$(\gt^{\vee}, \gt^{\vee}) = 2$,
and $\nu\colon \fh \to \fh^*$ the linear isomorphism defined by the restriction of $( \ , \ )$ to $\fh$.
We define a bilinear form on $\fh^*$ by $(\nu^{-1}(*), \nu^{-1}(*))$, which is also denoted by $( \ , \ )$.
Note that we have $(\gt, \gt) = 2$.
In this article, we say $\ga \in \gD$ is a \textit{long root} if $(\ga, \ga)=(\gt, \gt) (= 2)$, 
and a \textit{short root} otherwise.
Note that all roots are long if $\fg$ is simply laced. 

When $\fg$ is non-simply laced, by $r$ we denote the number $2 \cdot ($square length of a short root$)^{-1}$,
and put $r = 1$ when $\fg$ is simply laced. 
Then we have 
\begin{equation} \label{eq: number r} r = \begin{cases} 
                     1 & \text{if $\fg$ is simply laced}, \\
                     2 & \text{if $\fg$ is of type $B_n,C_n, F_4$,} \\
                     3 & \text{if $\fg$ is of type $G_2$.}
       \end{cases}
\end{equation}

For $\gl \in P_+$, we denote by $V_{\fg}(\gl)$ the irreducible $\fg$-module of highest weight $\gl$.
The following lemma is well-known:

\begin{Lem} \label{Lem: defining_rel}
  Let $v \in V_{\fg}(\gl)$ be a highest weight vector.
  Then $V_{\fg}(\gl)$ is generated by $v$ with the defining relations:
  \[ \fn_+.v = 0, \ \ \ h.v=\langle \gl, h \rangle v, \ \ \ f_i^{\langle \gl, \ga_i^{\vee} \rangle +1}.v = 0
  \]
  for $h \in \fh$ and $i \in I$. 
\end{Lem}

Let $\hat{\fg}$ be the non-twisted affine Lie algebra corresponding to the extended Dynkin diagram of $\fg$:
\[ \hat{\fg} = \fg \otimes \C[t, t^{-1}] \oplus \C K \oplus \C d,
\]
where $K$ denotes the canonical central element and $d$ is the degree operator.
The Lie bracket of $\hat{\fg}$ is given by 
\begin{align*} [x \otimes t^m + & a_1 K + b_1 d, y \otimes t^n + a_2 K + b_2 d] \\
               &= [x, y] \otimes t^{m+n} + n b_1 y \otimes t^{n} - m b_2 x \otimes t^{m} + m\gd_{m, -n} (x,y)K.
\end{align*}
We naturally consider $\fg$ as a Lie subalgebra of $\hat{\fg}$.
The Lie subalgebras $\hfh$, $\hfn_+$ and $\hfb$ of $\hfg$ are defined as follows: 
\[ \hfh = \fh \oplus \C K \oplus \C d, \ \ \ \hfn_+ = \fn \oplus \fg \otimes t \C[t], \ \ \ \hfb = \hfh \oplus \hfn_+.
\]
Denote by $\hat{\gD}$ the root system of $\hat{\fg}$. Considering $\gD$ naturally as a subset of $\hat{\gD}$,
we have $\hat{\gD} = \{ \ga + s\gd \mid \ga \in \gD, s \in \Z \} \cup \{ s \gd \mid s \in \Z \setminus \{ 0 \} \}$,
where $\gd \in \hat{\fh}^*$ is a unique element satisfying $\langle \gd, \fh + \C K \rangle = 0, \langle \gd,d \rangle = 1$.
The set of simple roots of $\hat{\gD}$ are denoted by $\hat{\Pi}= \{ \ga_0, \ga_1, \dots, \ga_n\}$ where $\ga_0 = \gd - \gt$.
We denote by $\hat{I} = \{ 0, 1, \dots ,n \}$ its index set.

Let $\gL_0, \dots, \gL_n \subseteq \hat{\fh}^*$ be the fundamental weights of $\hfg$,
$\hP = \sum_{i \in \hI} \Z \gL_i + \Z\gd$ the weight lattice, and 
$\hP_+ = \sum_{i \in \hI} \Z_{\ge0} \gL_i + \Z \gd \subseteq \hP$ the subset of dominant integral weights.
For $\gL \in \hP$, the \textit{level} of $\gL$ is defined by the integer $\langle \gL, K \rangle$.
Let $\hat{W}$ be the Weyl group of $\hat{\fg}$ with simple reflections $s_i \ ( i \in \hI)$.
We see $W$ naturally as a subgroup of $\hat{W}$.
We denote the Bruhat order on $\hW$ by $\le$.

Define Lie subalgebras of $\hfg$ by
\[ \mathcal{C}\fg = \fg \otimes \C[t] \subseteq \hfg,\ \ \ \mathcal{C} \fg_{d} = \Cg \oplus \C d \subseteq \hfg.
\]
We write $\fh_d= \fh \oplus \C d \subseteq \hat{\fh}$.
We usually consider $ \fh^*$ and $\fh_d^{*}$ as subspaces of $\hat{\fh}^*$ canonically.
Note that under this identification $\varpi_i$ is an element of $\hat{\fh}^*$ satisfying
\[ \langle \varpi_i, \ga_j^{\vee} \rangle = \gd_{ij} \ \text{for} \ j \in I, 
   \ \ \ \langle \varpi_i , \C K \oplus \C d \rangle = 0,
\]
and $P$ is a subgroup of $\hat{P}$.

In this article, we need to consider the subset of $\Pi$ consisting of short simple roots:
\[ \Pi^{\sh} = \{ \ga_i \in \Pi \mid \ga_i \ \text{is short} \}.
\]
Note that $\Pi^{\sh} = \emptyset$ if $\fg$ is simply laced.
Let $I^{\sh} = \{ i \in I \mid \ga_i \in \Pi^{\sh} \}$ be its index set,
and
\[ \gD^{\sh} = \gD \cap \sum_{i \in I^{\sh} } \Z \ga_i, \ \ \ \gD^{\sh}_{\pm} = \gD^{\sh} \cap \gD_{\pm}.
\]
Set
\[ Q^{\sh} = \sum_{i \in I^{\sh}}\Z \ga_i, \ \ \ Q_+^{\sh} = \sum_{i \in I^{\sh}}\Z_{\ge 0} \ga_i, \ \ \ \text{and} \ \ \
   W^{\sh} = \langle s_i \mid i \in I^{\sh} \rangle \subseteq W. \\
\]
Later we need the following elementary lemma:

\begin{Lem} \label{Lem:_preserve_sh}
  If $\ga \in \gD_+ \setminus  \gD^{\sh}_+$ and $w \in W^{\sh}$,
  then $w\ga \in \gD_+ \setminus \gD^{\sh}_+$ holds.
\end{Lem}

\begin{proof}
  It suffices to show the assertion for $w = s_k$ for $k \in I^\sh$.
  If we write $\ga = \sum_{i \in I} n_i \ga_i$, there exists some $j \in I \setminus I^{\sh}$ such that $n_j > 0$.
  Since the coefficient of $s_{k} \ga = \ga -\langle \ga, \ga_{k}^{\vee} \rangle \ga_{k}$
  on $\ga_j$ is $n_j > 0$, the assertion holds.  
\end{proof}

Put $\fh^{\sh} = \bigoplus_{i \in I^{\sh}} \C \ga^{\vee}_{i} \subseteq \fh$, 
and denote the simple Lie subalgebra corresponding to $\gD^{\sh}$ by $\fg^{\sh}$:
\[ \fg^{\sh} = \fh^{\sh} \oplus \bigoplus_{\ga \in \gD^{\sh}} \fg_{\ga}.
\]
Let $\fn_{\pm}^{\sh} = \bigoplus_{\ga \in \lpishr_{\pm}} \fg_{\ga}$.
Note that the type of $\fg^{\sh}$ is as follows:
\[ 
  \begin{cases} \{ 0 \} & \text{if $\fg$ is simply laced,} \\
                A_1     & \text{if $\fg$ is of type $B_n, G_2$,} \\
                A_2     & \text{if $\fg$ is of type $F_4$,} \\
                A_{n-1} & \text{if $\fg$ is of type $C_n$.} 
  \end{cases}   
\]
Let $\hfg^{\sh}$ be the non-twisted affine Lie algebra corresponding to the extended Dynkin diagram of $\fg^{\sh}$
(if $\fg$ is simply laced, set $\hfg^\sh = \C K \oplus \C d$).
We consider $\hfg^{\sh}$ naturally as a Lie subalgebra of $\hfg$.
We denote 
\[ \Cg^{\sh}= \fg^{\sh} \otimes \C[t] \subseteq \hfg^{\sh},\ \ \ \Cgd^{\sh} = \Cg^{\sh} \oplus \C d \subseteq \hfg^{\sh},
\]
and $\hfh^{\sh} = \fh^{\sh} \oplus \C K \oplus \C d$, $\fh_d^{\sh} = \fh^{\sh} \oplus \C d$.

Throughout this article, we denote by $\ol{\gl}$ the image of $\gl \in \hfh^*$ 
under the canonical projection $\hfh^* \to \big(\hfh^{\sh}\big)^*$.
Let $i_{\sh}$ denote the splitting $\big(\hfh^{\sh}\big)^* \to \hfh^*$ defined by
\[ i_{\sh}(\ol{\ga}_i) = \ga_i \ \text{for} \ i \in I^{\sh}, \ \ \ i_{\sh}(\ol{\gL}_0) = \gL_0, \ \ \ i_{\sh}(\ol{\gd}) = \gd.
\]
We write
\[ \ol{P}= \{ \ol{\gl} \mid \gl \in P\} = \sum_{i \in I^{\sh}} \Z \ol{\varpi}_i, \ \ \ 
   \ol{P}_+ = \sum_{i \in I^{\sh}} \Z_{\ge 0} \ol{\varpi}_i.
\]

We denote by $U_q(\hfg)$ the quantum affine algebra associated with $\hfg$ over $\C(q)$, 
and by $U_q'(\hfg)$ the quantum affine algebra
without the degree operator.
Let $U_q(\hfg^{\sh})$ and $U_q'(\hfg^{\sh})$ denote the ones associated with $\hfg^{\sh}$.

Let $H$ be one of the vector spaces $\hfh$, $\fhd$, $\fh$, $\hfh^{\sh}$, $\fh^{\sh}_d$ and $\fh^{\sh}$.
For an $H$-module $M$ and $\gl \in H^*$, we denote the weight space of weight $\gl$ by
\[ M_{\gl} = \{ v \in M \mid h.v = \langle \gl , h \rangle v \ \text{for} \ h \in H \},
\]
and if $M = \bigoplus_{\gl \in H^*} M_{\gl}$, we say $M$ is an $H$-weight module.
We denote the set of $H$-weights by $\wt_{H}(M) = \{ \gl \in H^* \mid M_{\gl} \neq 0 \}.$
If $M$ is a finite-dimensional $H$-weight module,
we define its character by
\[ \ch{H}M = \sum_{\gl \in H^*} (\dim M_{\gl}) e(\gl),
\]
where $e(\gl)$ are basis of the group algebra $\C [H^*]$ with multiplication defined by $e(\gl)e(\mu) = e(\gl + \mu)$.
For a subset $S \subseteq H^*$, we denote by $P_S$ the linear map $\C [H^*] \to \C[H^*]$ defined by
\[ P_S\big( e(\gl)\big) = \begin{cases} e(\gl) & \text{if $\gl \in S$}, \\
                              0      & \text{if $\gl \notin S$}.
                \end{cases}
\]
The map $i_{\sh}\colon \big(\hfh^{\sh}\big)^* \to \hfh^*$ induces naturally a linear map 
$\C[\big(\hfh^{\sh}\big)^*] \to \C[\hfh^*]$, which is also 
denoted by $i_{\sh}$.
The following lemma is used later:

\begin{Lem} \label{Lem: character}
  Let $\gl \in \fh^{*}$.
  Assume that $M$ is a $\fh$-weight $\Cg$-module generated by $v \in M_{\gl}$ and $v$ is annihilated 
  by $\fn_+ \otimes \C[t] \oplus \fh \otimes t\C [t]$.
  Then the subspace $W = U(\Cg^{\sh}).v$ satisfies
  \begin{equation}\label{eq:character}
    P_{\gl-Q^{\sh}_+}\, \ch{\fh}M = \ch{\fh} W = e\big(\gl - i_{\sh}(\ol{\gl})\big)i_{\sh}\ch{\fh^{\sh}} W.
  \end{equation}
\end{Lem}

\begin{proof}
  Put $\fn_-' = \sum_{\ga \in \gD_- \setminus \lpishr_-} \fg_{\ga}$. 
  Then we have $\fn_- = \fn_-^{\sh} \oplus \fn_-'$, and
  \begin{align*} 
   M &= U(\fn_- \otimes \C [t]).v \\
     &= U(\fn_-^{\sh} \otimes \C [t]).v + U(\fn_-^{\sh} \otimes \C [t])U(\fn_-' \otimes \C [t])_+.v \\
     &= W + U(\fn_-^{\sh} \otimes \C [t])U(\fn_-' \otimes \C [t])_+.v,
  \end{align*}
  where $U(\fn_-' \otimes \C [t])_+$ is the augmentation ideal.
  It is obvious that $\mathrm{wt}_{\fh}(W) \subseteq \gl - Q^{\sh}_+$, and 
  \[ \mathrm{wt}_{\fh} \Big(U(\fn_-^{\sh} \otimes \C [t])U(\fn_-' \otimes \C [t])_+.v\Big) \cap (\gl - Q^{\sh}_+) = \emptyset.
  \]
  Hence the first equality of (\ref{eq:character}) holds.
  The second equality follows from the following fact which is easily checked:
  if $\mu \in \gl -Q^{\sh}_+$, then we have 
  \[ \mu = \gl - i_{\sh}(\ol{\gl - \mu}) = i_{\sh}(\ol{\mu}) + \big(\gl - i_{\sh}(\ol{\gl})\big).
  \]
\end{proof}

\section{Weyl modules, Demazure modules and fusion product}

\subsection{Weyl modules}

In this article, we study the following $\Cgd$-module:

\begin{Def} \label{Def:Weyl_module}\normalfont
  For $\gl \in P_+$, the $\Cgd$-module $W(\gl)$ is defined by the module generated by an element $v$ with the relations:
  \begin{equation} \label{eq: rel for the Weyl module}
    \fn_+\otimes \C [t].v= 0,\ \ h \otimes t^s.v = \gd_{s0} \langle \gl,h \rangle v \ 
    \text{for} \ h \in \fh,  s \in \Z_{\ge 0},\ \ d.v = 0,
  \end{equation}
  and
  \[ f_i^{\langle \gl, \ga_i^{\vee} \rangle +1}.v = 0 \ \ \ \text{for} \ i \in I.
  \]
  We call $W(\gl)$ the \textit{Weyl module} for $\Cgd$ associated with $\gl \in P_+$.
\end{Def}

\begin{Rem} \label{Rem:_defining_rel} \normalfont
  It is easily seen that $W(\gl)$ is also cyclic as a $\Cg$-module and the defining relations of $W(\gl)$ as a $\Cg$-module 
  are the ones in the above definition with the relation $d.v = 0$ removed.
  Hence the restriction of $W(\gl)$ to $\Cg$ is the Weyl module for $\Cg$ introduced in \cite{MR2271991,MR2323538}.
\end{Rem}

\begin{Thm}[{\cite[Theorem 1.2.2]{MR2271991}}]\label{Thm:finite-dim}
  For every $\gl \in P_+$, the Weyl module $W(\gl)$ is finite-dimensional.
  Moreover, every finite-dimensional $\Cgd$-module generated 
  by an element $v$ satisfying the relations in {\normalfont{(\ref{eq: rel for the Weyl module})}}
  is a quotient of $W(\gl)$.
\end{Thm}

\subsection{Demazure modules}\label{Demazure modules}

We denote by $V(\gL)$ the irreducible highest weight $\hfg$-module of highest weight $\gL \in \hat{P}_+$.
Recall that for any $w \in \hat{W}$ we have $\dim V(\gL)_ {w\gL} = 1$.

\begin{Def} \normalfont
  For $w \in \hat{W}$, the $\hfb$-module
  \[ V_w(\gL) = U(\hfb).V(\gL)_{w\gL} = U(\hat{\fn}_+).V(\gL)_{w\gL}
  \]
  is called the \textit{Demazure submodule} of $V(\gL)$ associated with $w$. 
\end{Def}

\begin{Rem}\normalfont
  Note that $V_w(\gL) = V_{w'}(\gL)$ if $w\gL = w' \gL$.
\end{Rem}

In this article, we are mainly interested in the Demazure modules which are $\fg$-stable.
Since $f_i.V(\gL)_{w\gL} = 0$ holds if and only if $\langle w\gL, \ga_i^{\vee} \rangle \le 0$,
we see that 
$V_w(\gL)$ is $\fg$-stable if and only if 
$\langle w\gL , \ga_i^{\vee} \rangle \le 0 \ \text{for all} \ i \in I$,
which is equivalent to that $w\gL \in -P_+ + \ell \gL_0 + \Z \gd$ where $\ell$ is the level of $\gL$.
For the notational convenience, we use the following alternative notation:
let $\gl \in P_+$, $\ell \in \Z_{> 0}$ and $m \in \Z$.
There exists a unique $\gL \in \hat{P}_+$ such that $w_0 \gl + \ell \gL_0 + m\gd \in \hW\gL$.
For an element $w \in \hW$ such that $w\gL = w_0 \gl + \ell \gL_0 + m\gd$,
we write
\[ \D(\ell, \gl)[m] = V_w(\gL),
\]
which is a $\Cgd \oplus \C K$-module as stated above.
We usually consider only the $\Cgd$-module structure of $\D(\ell, \gl)[m]$.
For every $\gL \in \hP_+$ and $m \in \Z$, we have $V(\gL) \cong V(\gL + m \gd)$ 
as $\fg \otimes \C[t, t^{-1}] \oplus \C K$-modules.
Therefore the $\Cg$-module structure of $\D(\ell, \gl)[m]$ is independent of $m$,
and we denote this isomorphism class of $\Cg$-modules simply by $\D(\ell,\gl)$.

$\D(\ell, \gl)$ and $\D(\ell,\gl)[m]$ have descriptions in terms of generators and relations:

\begin{Prop} \label{Prop:_rel_of_Demazure}
  \noindent {\normalfont{(i)}} $\D(\ell, \gl)$ is isomorphic as a $\Cg$-module to the cyclic module 
    generated by an element $v$ with relations:
    \begin{equation} \label{eq:_Demazure_relation_1}
      \fn_+ \otimes \C[t]. v = 0, \ h \otimes t^s .v = \gd_{s0} \langle \gl, h \rangle v \  
      \text{for $h \in \fh, s \in \Z_{\ge 0}$},
    \end{equation}
    and for $\gg \in \gD_+$ and $s \in \Z_{\ge 0}$,
    \begin{equation} \label{eq:_Demazure_relation_2}
      (f_{\gg} \otimes t^s)^{k_{\gg, s}+1}.v = 0 \ \text{where} \ k_{\gg,s} 
               = \begin{cases} \mathrm{max}\{0, \langle \gl, \gg^{\vee}\rangle - \ell s \} & \text{if $\gg$ is long}, \\
                                \mathrm{max}\{0, \langle \gl, \gg^{\vee}\rangle -r \ell s \} & \text{if $\gg$ is short},
                  \end{cases}
    \end{equation} 
    where $r$ is the number defined in {\normalfont(\ref{eq: number r})}. \\
  \noindent {\normalfont{(ii)}} $\D(\ell, \gl)[m]$ is isomorphic as a $\Cgd$-module to the cyclic module 
    generated by an element $v$ 
    with relations {\normalfont(\ref{eq:_Demazure_relation_1}), (\ref{eq:_Demazure_relation_2})} and $d.v = mv$.
\end{Prop}

\begin{proof}
  The assertion (ii) easily follows from (i).
  The assertion (i) can be proved by the same way as the proof of \cite[Corollary 1]{MR2323538} using
  \begin{align*}
    -\langle \gl + \ell\gL_0 + m\gd, (-\gg + s \gd)^\vee \rangle & 
     = \langle \gl +\ell\gL_0 +m\gd , \gg^{\vee} - \frac{2s}{(\gg,\gg)}K \rangle \\
     &= \langle \gl, \gg^{\vee}\rangle -\frac{2\ell s}{(\gg,\gg)}.
  \end{align*}
\end{proof}

The following theorem is a reformulation of \cite[Theorem 7]{MR2323538} for our setting:

\begin{Thm} \label{Thm: FL's_Thm}
  Assume that $\fg$ is simply laced. 
  Then the Weyl module $W(\gl)$ is isomorphic to $\D(1, \gl)[0]$ as a $\Cgd$-module.
\end{Thm}

\begin{proof}
  Note that the notation ``$D(\ell, \gl^{\vee})$'' in \cite{MR2323538} coincides with 
  $\D(\ell, \ell \nu(\gl^{\vee}))$ in our notation.
  Since $\nu(\gl^{\vee}) = \gl$ holds in the simply laced case, \cite[Theorem 7]{MR2323538} says 
  that $W(\gl) \cong \D(1, \gl)$ as $\Cg$-modules.
  Then $W(\gl) \cong \D(1, \gl)[0]$ as $\Cgd$-modules easily follows.
\end{proof}

Later, we need the following lemma:

\begin{Lem} \label{Lem: lemma_of_A}
  Assume that $\fg$ is of type $A_n$. 
  Then for every $\ell \in \Z_{>0}$ and $i \in I$, $\D(\ell,\varpi_i)$ is isomorphic to $V_{\fg}(\varpi_i)$ as a $\fg$-module.
\end{Lem}

\begin{proof}
  Although this lemma can be shown directly from the definition, we shall 
  prove this using Proposition \ref{Prop:_rel_of_Demazure}.
  Since $\fg$ is of type $A_n$, we have $\langle \varpi_i, \gg^{\vee} \rangle = 0 \ \text{or} \ 1$ for all $\gg \in \gD_+$.
  From this and Proposition \ref{Prop:_rel_of_Demazure}, the generator $v \in \D(\ell, \varpi_i)$ 
  satisfies $f_{\gg} \otimes t\C[t].v = 0$
  for all $\gg \in \gD_+$, which implies
  \[ \D(\ell, \varpi_i) = U(\Cg).v = U(\fg).v.
  \]
  Since $v$ satisfies the defining relations of $V_{\fg}(\varpi_i)$ in Lemma \ref{Lem: defining_rel}, 
  $\D(\ell, \varpi_i)$ is a $\fg$-module quotient of $V_{\fg}(\varpi_i)$.
  Since $\D(\ell, \varpi_i)$ is non-trivial, the assertion is proved.
\end{proof}

\subsection{Fusion product}

We recall the definition of fusion products of $\Cg$-modules introduced in \cite{MR1729359} and some facts on them.
Let $W$ be a $\Cg$-module. For $a \in \C$, we define a $\Cg$-module $W_a$ by the pullback $\varphi_a^*W$, where 
$\varphi_a$ is an automorphism of $\Cg$ defined by $x \otimes t^s \mapsto x \otimes (t +a)^s$.
$U(\Cg)$ has a natural grading defined by
\[G^s\big(U(\Cg)\big) = \{ X \in U(\Cg) \mid [d, X] = sX \},
\]
from which we define a natural filtration on $U(\Cg)$ by
\[ F^s\big(U(\Cg)\big) = \bigoplus_{p \le s} G^p\big(U(\Cg)\big).
\]
Let now $W$ be a cyclic $\Cg$-module with a cyclic vector $w$.
Let $W_s$ be the subspace $F^s\big(U(\Cg)\big).w$ of $W$ for $s \ge -1$ (note that $W_{-1} = \{ 0 \}$), 
and denote the associated $\Cg$-module by 
$\mathrm{gr}(W)$:
\[ \mathrm{gr}(W) = \bigoplus_{s \ge 0} W_s / W_{s-1}.
\]

Now we recall the definition of fusion products. 
Let $W^1, \dots, W^k$ be $\Z$-graded cyclic finite-dimensional $\Cg$-modules with cyclic vectors $w_1, \dots, w_k$, and 
$c_1, \dots, c_k$ pairwise distinct complex numbers.
As shown in \cite{MR1729359}, $W^1_{c_1} \otimes \dots \otimes W^k_{c_k}$ is a cyclic $U(\Cg)$-module 
generated by $w_1 \otimes \dots \otimes w_k$.

\begin{Def}[\cite{MR1729359}]  \normalfont 
  The $\Cg$-module 
  \[ W^1_{c_1} * \dots * W^k_{c_k} = \mathrm{gr}(W^1_{c_1} \otimes \dots \otimes W^k_{c_k})
  \]
  is called the \textit{fusion product}.
\end{Def}

\begin{Rem} \normalfont
  Put $X = W^1_{c_1} \otimes \dots \otimes W^k_{c_k}$. 
  By letting $d$ act on $X_s /X_{s-1}$ by a scalar $s$, we sometimes regard $W^1_{c_1} * \dots * W^k_{c_k}$ as a $\Cgd$-module.
\end{Rem}

\begin{Lem} \label{Lem: fusion_product}
  {\normalfont(i)} 
    \[ \ch{\fh} W^1_{c_1} * \dots * W^k_{c_k} = \prod_{1 \le i \le k} \ch{\fh}W^i.
    \]
  {\normalfont(ii)}
    Let $\gl_1, \dots, \gl_k \in P_+$ and $\gl = \gl_1 + \dots + \gl_k$.
    Then there exists a surjective $\Cgd$-module homomorphism from $W(\gl)$ to 
    $\D(1, \gl_1)_{c_1} * \dots * \D(1, \gl_k)_{c_k}$. \\[0.1cm]
  {\normalfont(iii)} If $\fg$ is simply laced, the surjection in {\normalfont(ii)} is an isomorphism.
\end{Lem}

\begin{proof}
  Since $W^1_{c_1} * \dots * W^k_{c_k}$ is isomorphic to $W^1 \otimes \dots \otimes W^k$ as a $\fg$-module, 
  the assertion (i) follows.
  The assertion (ii) is proved by the same way as the proof of \cite[Lemma 5]{MR2323538}. 
  The assertion (iii) follows from Theorem \ref{Thm: FL's_Thm} and \cite[Theorem 8]{MR2323538}.
\end{proof}

\section{Filtrations on Weyl modules}

\subsection{Defining relations of Demazure modules}

The goal of this section is to show that a Weyl module admits a filtration 
whose successive quotients are homomorphic images of Demazure modules.
To do this, however, the defining relations of $\D(\ell, \gl)$ given in Proposition \ref{Prop:_rel_of_Demazure} are
insufficient, and we need to reduce the relations in the case $\ell = 1$.
Thus, we devote this and the next subsections to prove the following refined version of the defining relations for $\D(1,\gl)$:

\begin{Prop}\label{Prop:_rel_of_Demazure1}
   For $\gl \in P_+$, $\D(1, \gl)$ is isomorphic as a $\Cg$-module to the cyclic module generated by an element $v$ with the 
  following relations: \\
  {\normalfont{(D1)}} $\fn_+ \otimes \C[t].v= 0$, \\
  {\normalfont{(D2)}} $ h \otimes t^s.v = \gd_{s0}\langle \gl, h\rangle v$ for $h \in \fh, s\in \Z_{\ge 0}$,  \\
  {\normalfont{(D3)}} $f_i^{\langle \gl, \ga_i^{\vee}\rangle +1}.v = 0 $ for $i \in I$, \\
  {\normalfont{(D4)}} $(f_{\gg} \otimes t^s)^{\mathrm{max}\{0, \langle \gl,\gg^{\vee}\rangle - rs\}+1}.v = 0$ 
                               for $\gg \in \lpishr_+, s\in \Z_{\ge 0}$.
\end{Prop}

This proposition obviously implies the following corollary:

\begin{Cor} \label{Cor: rel_of_Demazure}
  For $\gl \in P_+$ and $m \in \Z$, $\D(1, \gl)[m]$ is isomorphic as a $\Cgd$-module to the cyclic module 
  generated by an element $v$ with the relations {\normalfont(D1)--(D4)} and $d.v = mv$.
\end{Cor}

If $\fg$ is simply laced, the relations in Proposition \ref{Prop:_rel_of_Demazure1} are just the defining relations of 
the Weyl module $W(\gl)$ (see Remark \ref{Rem:_defining_rel}).
Hence in this case, the proposition follows from Theorem \ref{Thm: FL's_Thm}.
From now until the end of the proof of the proposition, we assume that $\fg$ is non-simply laced.

We denote by $M^{\gl}$ the cyclic $\Cg$-module generated by an element $v$
with the relations (D1)--(D4).
By Proposition \ref{Prop:_rel_of_Demazure}, to prove the above proposition 
we need to show that $v \in M^{\gl}$ satisfies for all $\gg \in \gD_+$ and $s \in \Z_{\ge 0}$ that
\begin{equation} \label{eq:lem of rel}
  (f_{\gg} \otimes t^{s})^{k_{\gg,s}+1}.v = 0 \ \text{where} \
   k_{\gg,s} = \begin{cases} \mathrm{max}\{ 0, \langle \gl,\gg^{\vee}\rangle-s \}& \text{if $\gg$ is long}, \\
                             \mathrm{max}\{ 0, \langle \gl,\gg^{\vee}\rangle - rs \} & \text{if $\gg$ is short}.
               \end{cases}
\end{equation}
By separating the proof into several cases, we shall prove the equation (\ref{eq:lem of rel}).
First, the following case is elementary:

\begin{Lem}\label{Lem:_case1}
  For $\gg \in \gD_+$ and $s = 0$, {\normalfont(\ref{eq:lem of rel})} follows.
\end{Lem}

\begin{proof}
  By (D1)--(D3), $U(\fg).v \subseteq M^{\gl}$ is a $\fg$-module quotient of $V_{\fg}(\gl)$. 
  Then the assertion follows from the representation theory of $\mathfrak{sl}_2$.
\end{proof}

The proof of the following case is similar to that of \cite[Theorem 7]{MR2323538}.
We give it for completeness:

\begin{Lem}\label{Lem:_case2}
  If $\gg$ is a long root, {\normalfont(\ref{eq:lem of rel})} follows for all $s \in \Z_{\ge 0}$.
\end{Lem}

\begin{proof}
  Take a Lie subalgebra 
  \[ \mathfrak{sl}_{2, \gg} = \C e_{\gg} \oplus \C \gg^{\vee} \oplus \C f_{\gg} \subseteq \fg
  \]
  which is isomorphic to $\mathfrak{sl}_2$.
  Let $\mC \mathfrak{sl}_{2, \gg} = \mathfrak{sl}_{2, \gg} \otimes \C [t]$,
  and $N = U(\mC \mathfrak{sl}_{2, \gg}).v$.
  Note that $v$ satisfies the relations
  \[ e_\gg \otimes \C[t].v = 0, \ \ \ \gg^{\vee} \otimes t^s.v = \gd_{s0}\langle \gl, \gg^{\vee}\rangle v, \ \ \ 
     f_{\gg}^{\langle \gl, \gg^{\vee}\rangle+1}.v =0,
  \]
  which are the defining relations of the Weyl module $W_{\gg}\big(\langle \gl, \gg^{\vee}\rangle\big)$ 
  for $\mC \mathfrak{sl}_{2, \gg}$ (see Remark \ref{Rem:_defining_rel}).
  Here we identify the weight lattice of $\mathfrak{sl}_{2,\gg}$ with $\Z$.
  Therefore, $N$ is a quotient of this module.
  By Theorem \ref{Thm: FL's_Thm}, $W_{\gg}\big(\langle \gl, \gg^{\vee}\rangle\big)$ is isomorphic to 
  the Demazure module $\D_{\gg}\big(1, \langle \gl, \gg^{\vee}\rangle\big)$ for $\mC \mathfrak{sl}_{2,\gg}$.
  Hence $v$ satisfies the defining relations of $\D_{\gg}\big(1, \langle \gl,\gg^{\vee}\rangle\big)$ 
  in Proposition \ref{Prop:_rel_of_Demazure},
  which contain the relations
  \[ (f_{\gg} \otimes t^{s})^{\mathrm{max}\{0, \langle\gl,\gg^{\vee}\rangle - s \}+1}.v = 0
  \]
  for all $s \in \Z_{\ge 0}$.
  The assertion is proved.
\end{proof}

Before starting the proof of remaining cases, we prepare an elementary lemma:

\begin{Lem} \label{Lem:rank2}
  Assume that the rank of $\fg$ is $2$ with $\Pi = \{\ga, \gb\}$.
  Let $N$ be a $U(\fn_-)$-module, and assume that an element $v  \in N$ satisfies
  \[ f_{\ga}^{a+1}.v = 0, \ \ \ f_{\gb}^{b+1}.v = 0
  \]
  for some $a, b \in \Z_{\ge 0}$.
  Then for $\gg \in \gD_+$ such that $\gg^{\vee} = n_1 \ga^{\vee} + n_2 \gb^{\vee}$, we have
  \begin{equation}\label{eq:gg}
    f_{\gg}^{n_1 a + n_2 b +1}.v =0.
  \end{equation}
\end{Lem}

\begin{proof}
  Let $\varpi_{\ga}$ and $\varpi_{\gb}$ denote the fundamental weights corresponding to $\ga$ and $\gb$ respectively,
  and $\mu = a\varpi_\ga + b \varpi_{\gb}$.
  The following isomorphism as $U(\fn_-)$-modules is well-known:
  \[ V_{\fg}(\mu) \cong U(\fn_-)/ U(\fn_-)(\C f_{\ga}^{a+1} + \C f_{\gb}^{b+1}).
  \]
  Therefore, there exists a $U(\fn_-)$-module homomorphism from $V_{\fg}(\mu)$ to $N$ which maps a highest weight vector to $v$.
  Since a highest weight vector of $V_{\fg}(\mu)$ satisfies the relation (\ref{eq:gg}), so does $v$. 
\end{proof}

It remains to prove that the equation (\ref{eq:lem of rel}) follows for short $\gg \in \gD_+$. 
In the rest of this subsection, we shall prove that this statement is true if $\fg$ is of type $B_n, C_n$ or $F_4$,
and $G_2$ case is proved in the next subsection.
Note that $r=2$ if $\fg$ is of type $B_n, C_n$ or $F_4$.

\begin{Lem} \label{eq:existance1}
  Assume that $\fg$ is of type $B_n, C_n$ or $F_4$.
  If $\gg \in \gD_+ \setminus \lpishr_+$ is a short root,
  there exists a short root $\ga \in \gD_+$ and a long root $\gb \in \gD_+$ such that $\gg = \ga + \gb$.
\end{Lem}

\begin{proof}
  We prove the assertion by induction on $\mathrm{ht}\, \gg$.
  Put
  \[ S = \{ \ga \in \gD_+ \mid  \ga \notin \lpishr_+, \ \ga\text{ is short} \},
  \]
  and take arbitrary $\gg \in S$.
  Since $\gg \in \sum_{i \in I} \Z_{\ge 0} \ga_i$ and $( \gg, \gg) > 0$,
  there exists some $i \in I$ such that $(\gg, \ga_i) > 0$.
  Since $\gg \notin \Pi^{\sh}$ and $\gg$ is short, we have $\gg \notin \Pi$ and in particular $\gg \neq \ga_i$.
  Then we have $\langle \gg, \ga_i^{\vee} \rangle =1$ since $\gg$ is short, which implies $\gg = s_i(\gg) + \ga_i$.
  Note that $s_i (\gg)$ is a short root.
  If $\ga_i$ is long, the assertion is proved.
  Assume that $\ga_i \in \Pi^{\sh}$. 
  Then we have $s_i (\gg) \in S$ by Lemma \ref{Lem:_preserve_sh},
  and by the induction hypothesis there exist short $\ga \in \gD_+$ and long $\gb \in \gD_+$ 
  such that $s_i(\gg) = \ga + \gb$. 
  If $\ga = \ga_i$ then we have $\gg = \gb + 2\ga$, which contradicts that $\gg$ is short.
  Hence we have $\ga \neq \ga_i$, which implies $s_i (\ga) \in \gD_+$.
  Now the lemma is proved since we have $\gg = s_i(\ga) + s_i(\gb)$.  
\end{proof}

Now, the following proposition completes the proof of Proposition \ref{Prop:_rel_of_Demazure1} 
for $\fg$ of type $B_n, C_n$ or $F_4$:

\begin{Prop} \label{Prop:_proof_of_nonsimply}
  Assume that $\fg$ is of type $B_n, C_n$ or $F_4$.
  Then {\normalfont(\ref{eq:lem of rel})} follows for all short $\gg \in \gD_+$ and $s \in \Z_{\ge 0}$.
\end{Prop}

\begin{proof}
  We have to show that
  \begin{equation} \label{eq:_rel_to_be_proved}
    (f_{\gg} \otimes t^s)^{\mathrm{max}\{ 0, \langle \gl,\gg^{\vee}\rangle -2s \} +1}. v = 0.
  \end{equation}
  We show this by induction on $\height \gg$.
  If $\height \gg =1$, this trivially follows from (D4) since $\gg \in \Pi^{\sh}$.
  Assume $\height \gg >1$. We may also assume that $\gg \notin \lpishr_+$.
  By Lemma \ref{eq:existance1}, there exist short $\ga \in \gD_+$ and long $\gb \in \gD_+$ such that $\gg = \ga + \gb$.
  Put $a = \langle \gl, \ga^{\vee} \rangle$, $b = \langle \gl, \gb^{\vee} \rangle$, and 
  \[ q = \mathrm{min}\{b, s\}, \ \ \ p = s - q
  \]
  for given $s \in \Z_{\ge 0}$.
  We have from the induction hypothesis that
  \begin{equation} \label{eq:_ann1}
    (f_{\ga} \otimes t^p)^{\max\{0, a - 2p\}+1}.v = 0,
  \end{equation}
  and from Lemma \ref{Lem:_case2} that
  \begin{equation} \label{eq:_ann2}
    (f_{\gb} \otimes t^q)^{b-q+1}.v = 0.
  \end{equation}
  It is easily checked that the root subsystem $(\Z \ga + \Z \gb) \cap \gD$ is the root system of type $B_2$ 
  with basis $\{\ga, \gb \}$.
  Hence the Lie subalgebra
  \[ \C f_{\ga} \otimes t^p + \C f_{\gb} \otimes t^q + \C f_{\gg} \otimes t^{p+q} + 
     \C f_{2\ga + \gb} \otimes t^{2p + q} \subseteq \Cg
  \]
  is isomorphic to the nilradical of the Borel subalgebra of $\mathfrak{so}_{5}$.
  Since $\gg^{\vee} = \ga^{\vee} + 2 \gb^{\vee}$, we have from Lemma \ref{Lem:rank2}, (\ref{eq:_ann1}) and (\ref{eq:_ann2}) that
  \[ (f_{\gg} \otimes t^{p + q} )^{\mathrm{max}\{0, a-2p\} + 2(b - q) +1 }.v = 0.
  \]
  It is easily seen that this is equivalent to (\ref{eq:_rel_to_be_proved}),
  and the proposition is proved.
\end{proof}

\subsection{Proof of Proposition \ref{Prop:_rel_of_Demazure1} for type $G_2$}

In this subsection we assume $\fg$ is of type $G_2$,
and denote by $\ga$ the short simple root and by $\gb$ the long simple root.
Note that 
\[ \gD_+ = \{ \ga, \gb, \ga + \gb, 2\ga + \gb, 3\ga + \gb, 3\ga + 2\gb \}.
\]
Let $\varpi_{\ga}, \varpi_{\gb}$ be the corresponding fundamental weights.

Since Lemma \ref{Lem:_case2} is already proved, 
in order to complete the proof of Proposition \ref{Prop:_rel_of_Demazure1} for $\fg$
it suffices  to show the equation (\ref{eq:lem of rel}) for $\gg = \ga + \gb, 2\ga +\gb$.
The first one is easily proved:

\begin{Lem} \label{Lem:case_3}
  For $\gg = \ga + \gb$ and $s \in \Z_{\ge 0}$, {\normalfont(\ref{eq:lem of rel})} follows.
\end{Lem}

\begin{proof}
  This proof is similar to the one given in  Proposition \ref{Prop:_proof_of_nonsimply},
  and we shall only give the sketch of it.
  Assume that $\gl = a \varpi_\ga + b \varpi_\gb$ with $a, b \in \Z_{\ge 0}$.
  Since $(\ga+ \gb)^{\vee} = \ga^{\vee} + 3\gb^{\vee}$, what we have to show is that $v \in M^{\gl}$ satisfies
  \begin{equation} \label{eq:_to_be_proved2}
    (f_{\ga + \gb} \otimes t^s)^{\mathrm{max}\{0, a + 3b -3s\} +1} .v = 0.
  \end{equation}
  Put $q= \mathrm{min}\{b, s\}$, $p= s -q$ for given $s \in \Z_{\ge 0}$.
  Since we have
  \[ (f_{\ga} \otimes t^p)^{\mathrm{max}\{0, a-3p\}+1}.v = 0 \ \text{and} \ (f_{\gb} \otimes t^q)^{b-q+1}.v = 0,
  \]
  it follows from Lemma \ref{Lem:rank2} that
  \[ (f_{\ga+\gb} \otimes t^{p+q})^{\mathrm{max}\{0, a-3p \}+3(b-q)+1}.v=0,
  \]
  which is equivalent to (\ref{eq:_to_be_proved2}).
\end{proof}

It remains to prove (\ref{eq:lem of rel}) for $\gg = 2\ga + \gb$,
which is a relatively troublesome task.
Throughout the rest of this subsection, we abbreviate $X \otimes t^k$ as $Xt^k$ for $X \in \fg$ 
and $\mathrm{max}\{k_1, k_2\}$ as $\{k_1, k_2\}$ to simplify the notation.

For $a,b \in \Z_{\ge 0}$, we define a subspace $I_{a,b}$ of $U(\Cg)$ by
\[ I_{a,b} = \fn_+\C[t] + \fh t\C[t] + \C(\ga^{\vee} - a) + \C(\gb^{\vee} -b) + 
   \sum_{s \ge 0} \C (f_{\ga}t^s)^{\{0, a-3s\}+1} + \C f_{\gb}^{b+1}.
\]
Note that $U(\Cg)I_{a,b}$ is the left ideal generated by the relations in Proposition \ref{Prop:_rel_of_Demazure1}
with $\gl = a\varpi_\ga + b \varpi_\gb$.
We shall prove the statement
\begin{equation} \label{eq:identity1}
  (f_{2\ga+\gb}t^s)^{\{0,2a+ 3b-3s\} +1} \in U(\Cg)I_{a,b}
\end{equation}
for all $a,b \in \Z_{\ge 0}$ and $s \in \Z_{\ge 0}$, 
which is equivalent to (\ref{eq:lem of rel}) with $\gl = a\varpi_\ga + b\varpi_\gb$ and $\gg = 2\ga + \gb$.
In the course of the proof, we use repeatedly the fact that $X \in U(\Cg)$ annihilates $v \in M^{\gl}$ 
if and only if $X \in U(\Cg)I_{a,b}$ without further comment.

Before starting the proof of (\ref{eq:identity1}), let us prepare two lemmas:

\begin{Lem} \label{Lem:relation}
  For $s_1, s_2 \in \Z_{\ge 0}$, we have
  \[ (f_{2\ga + \gb}t^{2 s_1 + s_2})^{2\{ 0, a-3s_1\} + 3\{0, b-s_2\} + 1} \in U(\Cg)I_{a,b}.
  \]
\end{Lem}

\begin{proof}
  By Lemma \ref{Lem:_case2}, we have
  \[ (f_\gb t^{s_2})^{\{0, b -s_2\}+1} \in U(\Cg)I_{a,b}.
  \]
  Then since $\gg^{\vee} = 2\ga^{\vee} + 3 \gb^{\vee}$, we have from Lemma \ref{Lem:rank2} that
  \begin{align*} (f_{2\ga+\gb}&t^{2s_1 + s_2})^{2\{0,a-3s_1 \}+ 3\{0, b-s_2\}+1} \\
     &\in U(\Cg)\big(\C (f_{\ga} t^{s_1})^{\{0, a-3s_1\} + 1} + \C (f_\gb t^{s_2})^{\{0, b -s_2\}+1}\big) 
      \subseteq U(\Cg)I_{a,b}.
  \end{align*}
\end{proof}

\begin{Lem} \label{Lem:sl2}
  Let $\{e, h, f\}$ be the Chevalley basis of $\mathfrak{sl}_2$, $a \in \Z_{\ge 0}$ and $\ell \in \Z_{>0}$.
  Define a Lie subalgebra $\fa$ of $\mC \mathfrak{sl}_2 = \mathfrak{sl}_2 \otimes \C[t]$ by
  \[ \fa = e t \C[t] + h  \C[t] + f \C[t],
  \]
  and let $I$ be a subspace of $U(\fa)$ defined by
  \[ I = e t\C[t] + ht\C[t] + \C(h-a) + \sum_{s \ge 0} \C (f t^{s})^{\{0, a-\ell s\}+1}.
  \]
  Then we have for all $p \in \Z_{\ge 0}$ and $s \in \Z_{\ge 0}$ that 
  \[ e^p (ft^s)^{\{0, a-\ell s\} + 1} \in U(\fa)I + U(\mC \mathfrak{sl}_2)e.
  \]  
\end{Lem}

\begin{proof}
  By applying an involution of $\mC \mathfrak{sl}_2$ defined by $e t^k \leftrightarrow f t^k$, $h t^k \leftrightarrow -h t^k$,
  we see that the assertion of the lemma is equivalent to the following:
  if we put $\fa' = ft\C[t] + h \C[t] + e \C[t]$ and 
  \[ I' = f t\C[t] + h t\C[t] + \C (h+a) + \sum_{s\ge 0} \C (et^s)^{\{0, a -\ell s \} +1},
  \]
  then we have for all $p \in \Z_{\ge 0}$ and $s \in \Z_{\ge 0}$ that
  \[ f^p (et^s)^{\{0, a-\ell s\} + 1} \in U(\fa')I' + U(\mC \mathfrak{sl}_2)f.
  \]
  We shall prove this.
  Fix arbitrary $p$ and $s$.
  Since $U(\mC \mathfrak{sl}_2) = U(\fa') \oplus U(\mC \mathfrak{sl}_2)f$, there exists $X \in U(\mC \mathfrak{sl}_2)$ such that 
  \[ f^p (et^s)^{\{0, a-\ell s\} + 1} -Xf \in U(\fa').
  \]
  Let $V_w(\gL)$ be the $\hat{\mathfrak{sl}_2}$-Demazure module with $w\gL = \ell \gL_0 -a\varpi_1$,
  and $v_{w\gL}$ a nonzero vector in $V_w(\gL)_{w\gL}$. 
  By \cite[Theorem 1]{MR2323538}, the annihilating ideal of $v_{w\gL}$ in $U(\fa')$ is $U(\fa')I'$.
  Since
  \[ \big(f^p (et^s)^{\{0, a-\ell s\} + 1} -Xf\big).v_{w\gL}=0,
  \]
  this implies $f^p (et^s)^{\{0, a- \ell s\} + 1} -Xf \in U(\fa')I'$. The assertion is proved.
\end{proof}

We fix arbitrary $b\in\Z_{\ge0}$, and prove (\ref{eq:identity1}) for this $b$ by induction on $a \in \Z_{\ge 0}$.
To begin the induction, we first prove (\ref{eq:identity1}) for $a =0,1,2$.

\noindent(i) When $a = 0$,
  Lemma \ref{Lem:relation} with $s_1 = 0, s_2 = s$ implies (\ref{eq:identity1}).
  
\noindent(ii) Assume $a =1$. If $s \le b$, Lemma \ref{Lem:relation} with $s_1 = 0, s_2 = s$ implies (\ref{eq:identity1}).
If $s \ge b+2$, Lemma \ref{Lem:relation} with $s_1 = 1, s_2 = s -2$ implies (\ref{eq:identity1}).
Let $s = b +1$. 
By Lemma \ref{Lem:_case2} and $\langle \varpi_{\ga} + b \varpi_\gb, (3\ga + \gb)^{\vee}  \rangle  =b + 1$, we have 
\[ f_{3\ga + \gb}t^{b+1} \in U(\mC \fg)I_{1,b}.
\]
Then (\ref{eq:identity1}) follows since
\[ [e_\ga, f_{3\ga+\gb}t^{b+1}] = e_\ga f_{3\ga + \gb}t^{b+1} - f_{3\ga + \gb}t^{b+1}e_{\ga} \in U(\mC \fg)I_{1,b}.
\]

\noindent(iii) Assume $a = 2$. If $s \neq b + 1$, (\ref{eq:identity1}) is proved in the same way as in (ii).
Let $s = b +1$. By Lemma \ref{Lem:_case2} and \ref{Lem:case_3}, we have 
\[ (f_{3\ga + \gb}t^{b + 1})^2 \in U(\Cg)I_{2,b} \ \text{and} \ f_{\ga + \gb} t^{b+1} \in U(\Cg) I_{2,b}.
\]
Then (\ref{eq:identity1}) follows from the following calculation:
\begin{align*}
 U(\Cg)I_{2,b} \ni &e_{\ga}^2(f_{3\ga + \gb}t^{b+1})^2 = e_{\ga}(f_{3\ga+\gb}t^{b+1})^2e_{\ga}
            + 2[e_{\ga}, f_{3\ga + \gb}t^{b+1}]^2 \\  &+ 2f_{3\ga + \gb}t^{b+1}[e_{\ga}, f_{3\ga + \gb}t^{b+1}]e_{\ga}
             + 2f_{3\ga + \gb}t^{b+1}\mathrm{ad}(e_{\ga})^2(f_{3\ga+\gb}t^{b+1}). \\
\end{align*}

To proceed the induction we need to show that, if (\ref{eq:identity1}) holds for given $a,b \in \Z_{\ge 0}$ 
and all $s \in \Z_{\ge 0}$, then (\ref{eq:identity1}) with $a$ replaced by $a+3$ also holds.
To show this we shall prepare one lemma.
Define Lie subalgebras $\fa$ and $\fa_0$ of $\Cg$ by
\[ \fa = \fn_-\C[t] + \fh\C[t] + \sum_{\gg \in \gD_+}e_{\gg}t^{\langle\gg, \varpi_{\ga}^{\vee}\rangle} \C[t] \ \ \ \text{and}
   \ \ \ \fa_0 = \sum_{\begin{smallmatrix} \gg \in \gD_+ \setminus \{\gb\} \\ 0 \le s < \langle \gg, \varpi_{\ga}^{\vee}\rangle
                   \end{smallmatrix}} \C e_{\gg}t^s,
\]
where $\varpi_\ga^{\vee}$ is the fundamental coweight corresponding to $\ga$.
Note that $\Cg = \fa \oplus \fa_0$.
Let $I_{a,b}'$ be a subspace of $I_{a,b}$ defined by
\begin{align*}
  I_{a,b}' &= I_{a,b} \cap U(\fa)  \\
           & = \sum_{\gg \in \gD_+}e_{\gg} t^{\langle \gg, \varpi_\ga^{\vee}\rangle}\C[t] 
               + \fh t\C[t] + \C(\ga^{\vee} - a) + \C(\gb^{\vee} -b)\\ 
           &+ \sum_{s \ge 0} \C (f_{\ga}t^s)^{\{0, a-3s\}+1} + \C f_{\gb}^{b+1}.
\end{align*}

\begin{Lem} \label{Lem:crucial}
  \[ U(\Cg)I_{a,b} \subseteq U(\fa)I_{a,b}' \oplus U(\Cg)\fa_0.
  \]
\end{Lem}

\begin{proof}
  Set 
  \[ I = I_{a,b}, \ \ I' = I_{a,b}', \ \ J = U(\fa)I' \oplus U(\Cg)\fa_0.
  \]
  Since $U(\Cg) = U(\fa)U(\fa_0)$ and $U(\fa)J = J$, 
  it suffices to show that $U(\fa_0)I \subseteq J$.
  First we prove that 
  \[ I_1 = \fn_+\C[t] + \fh t\C[t] + \C(\ga^{\vee} -a) + \C(\gb^{\vee}-b)
  \]
  satisfies $U(\fa_0)I_1 \subseteq J$,
  which is equivalent to that for any $k \ge 0$ and a sequence $X_1, \dots, X_k$ of vectors of $\fa_0$
  we have $X_1\cdots X_k I_1 \subseteq J$.
  We prove a little stronger statement 
  \begin{equation} \label{eq:containment1}
    X_1\cdots X_k I_1 \subseteq (I' \cap I_1) \oplus U(\Cg)\fa_0
  \end{equation}
  by induction on $k$.
  If $k = 0$, this follows since $I_1 = (I' \cap I_1) \oplus \fa_0$.
  We can easily check that 
  \[ \mathrm{ad}(\fa_0)I_1 \subseteq \fn_+ \C[t] \subseteq (I' \cap I_1) \oplus \fa_0,
  \]
  and hence if $k > 0$ we have
  \begin{align*}
    X_1 \cdots X_k I_1 &\subseteq X_1 \cdots X_{k-1} I_1 X_k + X_1 \cdots X_{k-1}\big(\mathrm{ad}(X_k)I_1\big) \\
                      &\subseteq X_1 \cdots X_{k-1} (I' \cap I_1) + U(\Cg)\fa_0.
  \end{align*}
  This together with the induction hypothesis implies (\ref{eq:containment1}). $U(\fa_0)I_1 \subseteq J$ is proved.
  
  Next we prove $U(\fa_0)f_{\gb}^{b+1} \subseteq J$.
  Since $f_{\gb}^{b+1} \in I'$, it is enough to show $U(\fa_0)_+f_{\gb}^{b+1} \subseteq J$
  where $U(\fa_0)_+$ denotes the augmentation
  ideal.
  The $\fh$-weight set of $U(\fa_0)_+f_{\gb}^{b+1}$ with respect to the adjoint action obviously satisfies
  \begin{equation} \label{eq:containment2}
    \mathrm{wt}_{\fh}\big(U(\fa_0)_+f_{\gb}^{b+1}\big)\subseteq \Z_{>0} \ga + \Z \gb.
  \end{equation} 
  Since $\fa_0 \oplus \C f_{\gb}$ is a Lie subalgebra and 
  $U(\fa_0 \oplus \C f_{\gb}) = \C [f_{\gb}] \oplus \C [f_{\gb}] U(\fa_0)_+$,
  (\ref{eq:containment2}) implies by weight consideration that
  \[ U(\fa_0)_+f_{\gb} ^{b+1} \subseteq \C[f_\gb] U(\fa_0)_+ \subseteq J.
  \]
  The assertion is proved.
  
  Set
  \[ I_2 = \sum_{s \ge 0} \C (f_\ga t^s)^{\{0, a-3s\} +1}.
  \]
  Since $I = I_1 + \C f_{\gb}^{b+1} + I_2$, to complete the proof of the lemma 
  it suffices to show $U(\fa_0)I_2 \subseteq J$.
  To do this we put
  \[ \fa_0' = \sum_{\begin{smallmatrix} \gg \in \gD_+\setminus \{ \ga, \gb \} \\ 0 \le s < \langle \gg, \varpi_\ga^{\vee} \rangle
                     \end{smallmatrix}}  \C  e_\gg t^s,
  \]
  and prove $U(\fa_0')_+ I_2 \subseteq U(\Cg)I_1$ first.
  Note that $\fa_0 = \fa_0' \oplus \C e_{\ga}$.
  The $\fh$-weight set of $U(\fa_0')_+ I_2$ with respect to the adjoint action satisfies
  \begin{equation} \label{eq: contatinment3}
    \mathrm{wt}_{\fh} \big(U(\fa_0')_+I_2\big) \subseteq \Z \ga + \Z_{>0}\gb.
  \end{equation}
  Put $\fn_{+}^{(\ga)}= \sum_{\gg \in \gD_+\setminus \{ \ga \}} \fg_{\gg}$.
  Since $f_{\ga} \C[t] \oplus \fn_{+}^{(\ga)} \C[t]$
  is a Lie subalgebra and
  \[ U\big(f_{\ga} \C[t] \oplus\fn_{+}^{(\ga)} \C[t]\big)
     = U\big(f_{\ga} \C[t]\big) \oplus U\big(f_{\ga} \C[t]\big) U\big(\fn_{+}^{(\ga)} \C[t]\big)_+,
  \]
  (\ref{eq: contatinment3}) implies by weight consideration that
  \[ U(\fa_0')_+ I_2 \subseteq U\big(f_\ga \C[t]\big)  U\big(\fn_{+}^{(\ga)} \C[t]\big)_+ 
     \subseteq U(\Cg)I_1.
  \]
  Then since Lemma \ref{Lem:sl2} with $\ell = 3$ implies $\C[e_\ga]I_2 \subseteq J$, we have
  \[ U(\fa_0)I_2 \subseteq \C[e_\ga] I_2 + \C[e_{\ga}] U(\fa_0')_+ I_2
     \subseteq J + U(\Cg) I_1 \subseteq J.
  \]  
  The proof is complete.
\end{proof}

Now we show the following proposition, which completes the proof of Proposition \ref{Prop:_rel_of_Demazure1} for
type $G_2$:

\begin{Prop}
  {\normalfont(\ref{eq:identity1})} follows for all $a, b \in \Z_{\ge 0}$ and $s \in \Z_{\ge 0}$.
\end{Prop}

\begin{proof}
  As stated above, it suffices to show that if (\ref{eq:identity1})
  holds for given $a,b \in \Z_{\ge 0}$ and all $s \in \Z_{\ge 0}$, then we have
  \begin{equation} \label{eq:identity3}
    (f_{2\ga + \gb} t^s)^{\{0, 2(a +3) + 3b -3s \} +1} \in U(\Cg) I_{a+3,b}
  \end{equation}
  for all $s \in \Z_{\ge 0}$.

  Since $U(\Cg) = U(\fa) \oplus U(\Cg)\fa_0$,
  (\ref{eq:identity1}) and Lemma \ref{Lem:crucial} imply 
  \begin{equation} \label{eq:identity4}
    (f_{2\ga + \gb} t^s)^{\{ 0, 2a + 3b -3s\} +1} \in U(\Cg)I_{a,b} \cap U(\fa) \subseteq U(\fa)I_{a,b}'
  \end{equation}
  for all $s \in \Z_{\ge 0}$.
  Define a $\C$-linear map $\Phi\colon \fa \to U(\Cg)$ by
  \[ \Phi(e_\gg t^k) = e_{\gg} t^{k-\langle \gg, \varpi_\ga^{\vee} \rangle}, \ 
     \Phi(f_{\gg} t^k) = f_\gg t^{k+\langle \gg, \varpi_\ga^{\vee}\rangle} \ \text{for} \ \gg \in \gD_+,
  \]
  \[ \Phi(\ga^{\vee}t^k) = \ga^{\vee}t^k -3\gd_{k0}, \ \Phi(\gb^{\vee}t^k) =\gb^{\vee}t^k.
  \]
  It is easily checked  that $\Phi$ satisfies $\Phi([X_1, X_2]) = [\Phi(X_1), \Phi(X_2)]$ for $X_1, X_2 \in \fa$.
  Hence $\Phi$ induces a $\C$-algebra homomorphism $U(\fa) \to U(\Cg)$, which we also denote by $\Phi$.
  Applying $\Phi$ to (\ref{eq:identity4}), we have
  \begin{multline*} 
     (f_{2\ga + \gb} t^{s+2})^{\{0, 2(a+3) + 3b -3(s +2)\}+1} \in \Phi(U(\fa)I_{a,b}') \subseteq U(\Cg)\Phi(I_{a,b}') \\
     \subseteq U(\Cg) \Big(\fn_+\C[t]+\fh t\C[t] + \C(\ga^{\vee} -(a+3)) + \C(\gb^{\vee} -b)  \\
     + \sum_{s \ge 0}\C(f_\ga t^{s+1})^{\{0,(a+3)-3(s+1)\}+1} +\C f_\gb^{b+1}\Big) \subseteq U(\Cg)I_{a+3,b},
  \end{multline*}
  and (\ref{eq:identity3}) is proved for $s \ge 2$.
  (\ref{eq:identity3}) for $s = 0$ follows from Lemma \ref{Lem:_case1}.
  It remains to prove (\ref{eq:identity3}) for $s =1$, that is,
  \[ (f_{2\ga + \gb} t)^{2a +3b + 4} \in U(\Cg)I_{a+3,b}.
  \]
  If $b \ge 1$, Lemma \ref{Lem:relation} with $s_1 =0, s_2=1$ implies this.
  Assume $b = 0$ and put
  \[ N = U(\Cg)/ U(\Cg)I_{a+3, 0}.
  \]
  We shall prove $(f_{2\ga +\gb}t)^{2a +4}.\bar{1} = 0$, where $\bar{1} \in N$ denotes the image of $1$.
  Since $N$ is a quotient of the Weyl module $W\big((a+3)\varpi_\ga\big)$, it is finite-dimensional 
  by Theorem \ref{Thm:finite-dim}.
  As the $\fh$-weight of $(f_{2\ga + \gb}t)^{2a+4}.\bar{1}$ is $-(a+1)\varpi_\ga$ and 
  $\big\langle (2\ga + \gb)^{\vee}, -(a+1) \varpi_\ga \big\rangle = -2a-2$,
  it is enough to prove
  \[ e_{2\ga + \gb}^{2a+2}(f_{2\ga + \gb}t)^{2a+4}.\bar{1}= 0
  \]
  by the representation theory of $\mathfrak{sl}_2$.
  By \cite[Lemma 7.1]{MR502647}, we have 
  \[ e_{2\ga + \gb}^{2a + 2} (f_{2\ga + \gb}t)^{2a+4}.\bar{1} \in 
     \sum_{s_1 + s_2 = 2a+4}\C f_{2\ga+\gb}t^{s_1} f_{2\ga + \gb} t^{s_2}.\bar{1}.
  \]
  Using (\ref{eq:identity3}) for $s \ge 2$ which are already proved, we can see that the right hand side is $\{0\}$.
\end{proof}

\subsection{Quantized Demazure modules and Joseph's results}\label{subsection: Quantized_Demazure}

The quantized version of Demazure modules also can be defined in a similar manner as the classical ones.
For $\gL \in \hat{P}_+$, denote by $V_q(\gL)$ the irreducible highest weight $U_q(\hat{\fg})$-module of highest weight $\gL$.
Similarly as the classical case, we have $\dim_{\C(q)}V_q(\gL)_{w\gL} =1$
for all $w \in \hat{W}$. Denote by $U_q(\hat{\fn}_+)$ the positive part of $U_q(\hfg)$.

\begin{Def} \normalfont
  We call $V_{q,w}(\gL)= U_q(\hat{\fn}_+).V_q(\gL)_{w\gL}$ the \textit{quantized Demazure submodule} 
  of $V_q(\gL)$ associated with $w$.
\end{Def}

Joseph posed in \cite[\S 5.8]{MR826100} a question which asked if the tensor product of a one-dimensional 
Demazure module by an arbitrary Demazure module admits a filtration whose successive quotients are 
isomorphic to Demazure modules.
Polo \cite{MR1021515} and Mathieu \cite{MR1016897} gave the positive answer to this question 
in the case of semisimple Lie algebras,
and Joseph \cite{MR2214249} himself gave the positive answer in the case of the quantized enveloping algebras associated
with simply laced Kac-Moody Lie algebras.
Here we briefly recall the Joseph's result since we use it later. 
Although his result is applicable to any quantized enveloping algebras associated with simply laced Kac-Moody Lie algebras,
we concentrate only on affine case here. 

Let $A = \Z[q, q^{-1}]$, $U_q^{\Z}(\hat{\fn}_+)$ and $U_q^{\Z}(\hat{\fg})$ the $A$-forms of 
$U_q(\hat{\fn}_+)$ and $U_q(\hfg)$ respectively, 
and $T$ the Cartan part of $U_q^{\Z}(\hat{\fg})$ (for the precise definitions, see \cite[\S 2.2]{MR2214249}).
We denote by $U_q^{\Z}(\hat{\fb})$ the subring of $U_q^{\Z}(\hat{\fg})$ generated by $U_q^{\Z}(\hat{\fn}_+)$ and $T$.
For $w \in \hat{W}$, let $u_{w\gL}$ be a nonzero element of weight $w\gL$ in $V_q(\gL)$,
and set
\[ V_{q,w}^{\Z}(\gL) = U_{q}^{\Z}(\hat{\fn}_+).u_{w\gL},
\]
which is obviously a $U_q^{\Z}(\hat{\fb})$-module.
Taking the classical limit, we have
\[ \C \otimes_A V_{q,w}^{\Z}(\gL) \cong V_w(\gL),
\]
where $A$ acts on $\C$ by letting $q$ act by $1$.
Joseph has proved the following theorem:

\begin{Thm}[{\cite[Theorem 5.22]{MR2214249}}] \label{Thm: Joseph's_Theorem}
  Assume that $\fg$ is simply laced,
  and let $\gL, \gL' \in \hP_+$, and $w \in \hW$.
  Then the $U_q^{\Z}(\hat{\fb})$-submodule $u_{\gL} \otimes_A V_{q,w}^\Z(\gL')$ of $V_q(\gL) \otimes_{\C(q)}V_q(\gL')$
  has a $U_q^{\Z}(\hat{\fb})$-module filtration
  \[ 0 = Y_0 \subseteq Y_1 \subseteq \dots\subseteq  Y_k = u_{\gL}\otimes_A V_{q,w}^\Z (\gL')
  \]
  such that each subquotient $Y_i /Y_{i-1}$ satisfies
  \[ Y_i /Y_{i-1} \cong V_{q,y_i}^\Z(\gL^{i}) \ \text{for some} \ \gL^i \in \hP_+ \ \text{and} \ y_i \in \hW.
  \] \\[-20pt]
\end{Thm}

\begin{Rem} \normalfont
  (i) In \cite[Theorem 5.22]{MR2214249}, a given Kac-Moody Lie algebra is assumed to be simply laced, 
    and this assumption excludes the case where $\hat{\fg}$ is of type $A_1^{(1)}$. 
    In Joseph's proof, however, this assumption is only used in \cite[Lemma 3.14]{MR2214249} to apply 
    a positivity result of Lusztig.
    We can check that the proof of this positivity result in \cite[\S 22.1.7]{MR1227098} is also applicable 
    to $A_1^{(1)}$ without any modification, and hence the above theorem is true for any simply laced $\fg$. \\
  (ii) \cite[Theorem 5.22]{MR2214249} claims only that the above filtration is a $U_q^{\Z}(\hat{\fn}_+)$-module filtration.
    It is easily seen, however, that this is in fact a $U_q^{\Z}(\hat{\fb})$-module filtration since each $Y_i$ is defined 
    as an $A$-span of weight vectors (\cite[\S 5.7]{MR2214249}).  
\end{Rem}

Taking the classical limit, the following result is obtained:

\begin{Cor}\label{Cor: Joseph's_result}
  Assume that $\fg$ is simply laced.
  For $\ell' > \ell$, $\D(\ell, \gl)[m]$ has a $\Cgd$-module filtration
  \[ 0 = D_0 \subseteq D_1 \subseteq \dots \subseteq D_k = \D(\ell, \gl)[m]
  \]
  such that each subquotient $D_i/D_{i-1}$ satisfies
  \[ D_i/D_{i-1} \cong \D(\ell', \mu_i)[m_i] \ \text{for some} \ \mu_i \in P_+ \ \text{and} \ m_i \in \Z_{\ge m}.
  \]  
\end{Cor}

\begin{proof}
  We shall prove the assertion for $\ell' = \ell +1$. 
  The results for general $\ell'$ can be obtained by applying this case repeatedly.
  Take $\gL' \in \hP_+$ and $w \in \hat{W}$ so that $w\gL' = w_0 \gl + \ell \gL_0 + m \gd$.
  By Theorem \ref{Thm: Joseph's_Theorem}, $u_{\gL_0} \otimes V_{q,w}^{\Z}(\gL')$ has a $U_q^{\Z}(\hat{\fb})$-module filtration
  $0=Y_0 \subseteq Y_1 \subseteq \dots \subseteq Y_k = u_{\gL_0} \otimes V_{q,w}^{\Z}(\gL')$ 
  such that 
  \[ Y_i / Y_{i-1} \cong V_{q,y_i}^\Z(\gL^i) \ \text{for some} \ \gL^i \in \hP_+ \ \text{and} \ y_i \in \hW.
  \]
  Put $D_i = \C \otimes_A Y_i$ for $0 \le i \le k$.
  Then we have 
  \[ D_k = \C \otimes_A (u_{\gL_0} \otimes_A V_{q,w}^{\Z}(\gL')) = \C_{\gL_0} \otimes_{\C} V_w(\gL') = 
     \C_{\gL_0} \otimes_{\C} \D(\ell, \gl)[m],
  \]
  where
  $\C_{\gL_0}$ denotes a $1$-dimensional $\hat{\fb}$-module
  spanned by a vector of weight $\gL_0$ on which $\hat{\fn}_+$ acts trivially.
  Since all $Y_i$ and $Y_i/Y_{i-1}$ are free $A$-modules (\cite[\S 5.7]{MR2214249}), 
  $\C_{\gL_0} \otimes \D(\ell, \gl)[m]$ has the following $\hat{\fb}$-module filtration
  \[ 0=D_0 \subseteq D_1 \subseteq \dots \subseteq D_k = \C_{\gL_0} \otimes \D(\ell, \gl)[m],
  \]
  and each successive quotient satisfies $D_i / D_{i-1} \cong V_{y_i}(\gL^i)$.
  Obviously each $\gL^i$ is of level $\ell + 1$.
  By \cite[Theorem 5.9]{MR2214249}, for each $1 \le i \le k$ there exists 
  a $U_q^{\Z}(\hfg)$-submodule $Z_i$ of $V_q(\gL) \otimes_{\C(q)}V_q(\gL')$ such that
  \[ Y_i = Z_i \cap (u_{\gL_0} \otimes_A V_{q,w}^{\Z}(\gL')).
  \]
  Then since $\C_{\gL_0} \otimes _{\C} \D(\ell, \gl)[m]$ is a $\big(\Cgd \oplus \C K\big)$-module
  we see that each $D_i$ is a $\big(\Cgd \oplus \C K\big)$-module,
  and so is $D_i / D_{i-1}$.
  Hence each $D_i/D_{i-1}$ is isomorphic to $\D(\ell + 1,\mu_i)[m_i]$ for some $\mu_i \in P_+$ and $m_i \in \Z$.
  Each $m_i$ obviously satisfies $m_i \ge  m$ since $\wt_{\fhd}\D(\ell, \gl)[m] \in \gl - Q_+ + \Z_{\ge m} \gd$.
  Now the assertion of corollary is proved by restricting these results to $\Cgd$ 
  since $\C_{\gL_0}$ is a trivial $\Cgd$-module.
\end{proof}

We define a partial order $\preceq$ on $P_+ + \Z\gd$ by $\gl_1 + m_1 \gd \preceq \gl_2 + m_2 \gd$
if $\gl_2 - \gl_1 \in Q_+$ and $m_1 \ge m_2$.
Since $\D(\ell, \gl)[m]$ is $U(\fn_-\otimes \C[t])$-cyclic, 
$\gl_1 + m_1 \gd \in \mathrm{wt}_{\fhd}\D(\ell, \gl)[m]$ implies $\gl_1 +m_1\gd \preceq \gl + m\gd$.
From this we see that $\{\ch{\fhd}\D(\ell, \gl)[m]\mid \gl \in P_+, m \in \Z \}$ are linearly independent 
for each $\ell \in \Z_{>0}$.

When a $\fhd$-weight $\Cgd$-module $M$ has a filtration $0 = D_0 \subseteq D_1 \subseteq \dots \subseteq D_k = M$
such that 
\[ D_i / D_{i-1} \cong \D(\ell, \mu_i)[m_i] \ \text{for some} \ \mu_i \in P_+ \ \text{and} \ m_i \in \Z
\]
with fixed $\ell \in \Z_{>0}$, we define 
\[ (M : \D(\ell, \gl)[m]) = \# \{ i \mid D_i / D_{i-1} \cong \D(\ell, \gl)[m] \},
\]
which is independent of the choice of the filtration by the linearly independence of the characters.

\subsection{Filtrations on the Weyl modules}

Throughout this subsection, we assume that $\fg$ is non-simply laced.
We denote by $\D^{\sh}(\ell,\nu)$ the $\Cg^{\sh}$-Demazure module and by $\D^{\sh}(\ell, \nu)[m]$ the 
$\Cgd^{\sh}$-Demazure module for $\nu \in \ol{P}_+$, $\ell \in \Z_{> 0}$ and $m \in \Z$.

\begin{Lem} \label{Lem: subisom}
  Let $v$ be the generator of $W(\gl)$ in Definition \ref{Def:Weyl_module}, 
  and set $W =  U(\Cg^{\sh}_d).v \subseteq W(\gl)$.
  Then $W$ is isomorphic to $\D^{\sh}(1, \ol{\gl})[0]$ as a $\Cgd^{\sh}$-module.
\end{Lem}

\begin{proof}
  By Theorem \ref{Thm: FL's_Thm}, $\D^{\sh}(1,\ol{\gl})[0]$ is isomorphic to the $\Cgd^{\sh}$-Weyl module
  $W^{\sh}(\ol{\gl})$.
  Hence there exists a surjective homomorphism 
  $\varphi\colon \D^{\sh}(1, \ol{\gl})[0] \twoheadrightarrow W$ of $\Cgd^{\sh}$-modules by Theorem \ref{Thm:finite-dim}.
  We need to show that $\varphi$ is injective.
  In this proof, by $V(\mu)$ for $\mu \in P_+$ we denote the $\Cgd$-module obtained from $V_{\fg}(\mu)$
  by letting $\fg \otimes t\C[t] \oplus \C d$ act trivially.
  Write $\gl = \sum_{i \in I} \gl_i \varpi_i$, and put $p = \sum_i \gl_i$.
  Let $c_1,\dots,c_p$ be pairwise distinct complex numbers, and define a $\Cg$-module $W_1$ by
  \[ W_1 = V(\varpi_1)_{c_1} * \dots * V(\varpi_1)_{c_{\gl_1}} * \dots * V(\varpi_n)_{c_{p - \gl_n +1}} * \dots
      * V(\varpi_n)_{c_p},
  \]
  where each $V(\varpi_i)$ occurs $\gl_i$ times.
  By the same way as \cite[Lemma 5]{MR2323538}, we can show that there exists a surjective homomorphism 
  $\psi\colon W(\gl) \twoheadrightarrow W_1$ of $\Cg$-modules.
  It suffices to show that the homomorphism $\psi \circ \varphi\colon
  \D^{\sh}(1, \gl)[0] \to W_1$ of $\Cg^{\sh}$-modules is injective.
  By Lemma \ref{Lem: character} and Lemma \ref{Lem: fusion_product} (i), we have
  \[ \ch{\fh}(\Img \psi \circ \varphi) = P_{\gl - Q^{\sh}_+}\ch{\fh}W_1 
     = P_{\gl -Q^{\sh}_+}\prod_{i \in I} \ch{\fh}V(\varpi_i)^{\gl_i},
  \]
  and using $\mathrm{wt}_{\fh}V(\varpi_i)  \subseteq \varpi_i -Q_+$ we have
  \[  P_{\gl -Q^{\sh}_+}\prod_{i \in I} \ch{\fh}V(\varpi_i)^{\gl_i} 
      = \prod_{i \in I} \big(P_{\varpi_i - Q^{\sh}_+}\ch{\fh} V(\varpi_i)\big)^{\gl_i}.
  \] 
  We have from Lemma \ref{Lem: character} that $P_{\varpi_i -Q^{\sh}_+} \ch{\fh}V(\varpi_i)  = \ch{\fh} U(\fg^{\sh}).v_i$
  where $v_i$ is a nonzero highest weight vector,
  and we can easily see that
  \[ U(\fg^{\sh}).v_i \cong V_{\fg^{\sh}}(\ol{\varpi}_i) = \begin{cases} 
                               V_{\fg^{\sh}}(\ol{\varpi}_i) & \text{if} \ i \in I^{\sh}, \\
                               V_{\fg^{\sh}}(0) & \text{if} \ i \notin I^{\sh}
                             \end{cases}
  \]
  as $\fg^{\sh}$-modules.
  From these equations, we have
  \[ \mathrm{dim}\, (\Img \psi \circ \varphi)= \prod_{i \in I^{\sh}} \dim V_{\fg^{\sh}}(\ol{\varpi}_i)^{\gl_i}.
  \]
  On the other hand, we have that 
  \begin{align*} \dim \D^{\sh}(1, \ol{\gl}) = \prod_{i \in I^{\sh}} \dim \D^{\sh}(1, \ol{\varpi}_i)^{\gl_i}
                                            = \prod_{i \in I^{\sh}} \dim V_{\fg^{\sh}}(\ol{\varpi}_i)^{\gl_i},
  \end{align*}
  where the first equality follows from \cite[Theorem 1]{MR2235341}
  and the second follows from Lemma \ref{Lem: lemma_of_A}.
  Hence we have $\mathrm{dim}\, (\Img \psi \circ \varphi) = \dim \D^{\sh}(1,\ol{\gl})$, 
  which implies that $\psi \circ \varphi$ is injective.
  Our assertion is proved.
\end{proof}

Let $\gl \in P_+$. 
By Corollary \ref{Cor: Joseph's_result}, $\D^{\sh}(1, \ol{\gl})[0]$ has a $\Cgd^{\sh}$-module filtration 
$0=D_0\subseteq D_1 \subseteq \dots \subseteq D_k = \D^{\sh}(1, \ol{\gl})[0]$ such that 
\[ D_i /D_{i-1} \cong \D^{\sh}(r, \nu_i)[m_i] \ \text{for some} \ \nu_i \in \ol{P}_+ \ \text{and} \ m_i \in \Z_{\ge 0}.
\]
Now we show the following proposition using this filtration, which is the main result of the first half of this article:

\begin{Prop} \label{Prop: Main_Theorem1}
  Let $\mu_i =i_{\sh}(\nu_i) + \big(\gl - i_{\sh}(\ol{\gl})\big)$ for each $1 \le i \le k$.
  Then the Weyl module $W(\gl)$ has a $\Cgd$-module filtration 
  $0= W_0 \subseteq W_1 \subseteq \dots \subseteq W_k = W(\gl)$ such that 
  each subquotient $W_i / W_{i-1}$ is a quotient of $\D(1, \mu_i)[m_i]$.
\end{Prop}

\begin{proof}
  By Lemma \ref{Lem: subisom}, $W(\gl) \supseteq W = U(\Cgd^{\sh}).v $ has a $\Cgd^{\sh}$-module
  filtration $0 = X_0 \subseteq X_1 \subseteq \dots \subseteq X_k =W$ such that $X_i / X_{i-1} \cong \D^{\sh}(r, \nu_i)[m_i]$
  for $1 \le i \le k$.
  For each $i$, take $x_i \in X_i$ so that the image under the homomorphism
  $X_i \twoheadrightarrow X_i / X_{i-1} \cong \D^{\sh}(r, \nu_i)[m_i]$ 
  coincides with the generator in Proposition \ref{Prop:_rel_of_Demazure}.
  We may assume that each $x_i$ is a $\fhd^{\sh}$-weight vector of weight $\nu_i + m_i\ol{\gd}$.
  Then each $x_i$ satisfies
  \[ \fn^{\sh}_+ \otimes \C[t].x_i \subseteq X_{i-1}, \ \fh^{\sh} \otimes t\C[t].x_i \subseteq X_{i-1}, \
     h.x_i = \langle \nu_i + m_i \ol{\gd}, h \rangle x_i \ \text{for} \ h \in \fhd^{\sh},
  \]
  and 
  \[ (f_{\gg}\otimes t^s)^{\mathrm{max}\{0, \langle \nu_i,\gg^{\vee}\rangle - rs\}+1}.x_i \in X_{i-1}
     \ \ \ \text{for} \ \gg \in \lpishr_+ \ \text{and}\ s \in \Z_{\ge 0}.
  \]
  Let us show first that each $x_i$ is a $\fh$-weight vector of weight $\mu_i$.
  Set $(\fh^{\sh})^{\bot} =\{ h \in \fh \mid (h, h_1) = 0 \ \text{for} \ h_1 \in \fh^{\sh} \}$,
  which obviously satisfies $\fh = \fh^{\sh} \oplus (\fh^{\sh})^{\bot}$ and $[(\fh^{\sh})^{\bot}, \Cgd^{\sh}] = 0$. 
  Then since $x_i \in U(\Cgd^{\sh}).v$, we have 
  \[ h.x_i =\langle \gl, h\rangle x_i \ \ \ \text{for} \ h \in (\fh^{\sh})^{\bot}.
  \]
  On the other hand, we can easily check from the definition of $i_{\sh}$ that 
  $\langle \mu_i, h \rangle = \langle \gl,h\rangle$ for $h \in (\fh^{\sh})^{\bot}$
  and $\langle \mu_i ,h \rangle = \langle \nu_i, h \rangle $ for $h \in \fh^{\sh}$.
  Hence it is proved that each $x_i$ is a $\fh$-weight vector of weight $\mu_i$.
  Now set $W_i = U(\Cgd).X_i$ for each $i$,
  and let $\overline{x}_i$ be the image of $x_i$ under $W_i \twoheadrightarrow W_i/W_{i-1}$.
  Note that we have $W_k = W(\gl)$ and
  \[ U(\Cgd).x_i + W_{i-1} = U(\Cgd).(\C x_i + X_{i-1}) = U(\Cgd).X_i = W_i
  \]
  for each $i$. Hence to finish the proof, 
  it suffices to show that each $\overline{x}_i$ satisfies the defining relations of $\D(1, \mu_i)[m_i]$ 
  in Corollary \ref{Cor: rel_of_Demazure}. 
  From the relations which $x_i$ satisfies, 
  we can see that $\ol{x}_i$ is a weight vector of weight $\mu_i + m_i\gd$ and satisfies (D4).
  In order to prove (D1), it suffices to show that
  \begin{equation}\label{eq: for_e}
    e_{\gg} \otimes t^s.x_i = 0 \ \ \ \text{for} \ \gg \in \gD_+ \setminus \lpishr_+, s \in \Z_{\ge 0},
  \end{equation}
  which follows since the $\fh$-weight of $e_{\gg} \otimes t^s.x_i$ is $\mu_i + \gg \in  \gl -Q_+^{\sh} + \gg$ 
  and hence $\mu_i + \gg \notin \gl - Q_+$.
  Then (D3) follows since $W_i / W_{i-1}$ is finite-dimensional.
  To prove (D2), it suffices to show that 
  \begin{equation} \label{eq: for_h}
    (\fh^{\sh})^{\bot} \otimes t\C[t].x_i = 0,
  \end{equation}
  which is easily checked since $x_i \in U(\Cgd^{\sh}).v$ and $\big[(\fh^{\sh})^{\bot} \otimes t\C[t], \Cgd^{\sh}\big] = 0$.
\end{proof}

\begin{Rem} \normalfont
  In Section \ref{Section: Main_Theorem}, we show that the successive quotients of the given filtration on $W(\gl)$ 
  are actually isomorphic to the Demazure modules.
\end{Rem}

For $F_1, F_2 \in\Z[\fhd^*]$, we write $F_1 \le F_2$ if $F_2 - F_1 \in \Z_{\ge 0} [\fhd^*]$.
The above proposition implies the following corollary:

\begin{Cor}\label{Cor: Main_corollary1}
  Let $\gl \in P_+$, and set $\gl' = \gl - i_{\sh}(\ol{\gl})$.
  Then we have
  \[ \ch{\fhd}W(\gl) \le \sum_{\nu \in \ol{P}_+, m \in \Z_{\ge 0}} \big(\D^{\sh}(1,\ol{\gl})[0]: \D^{\sh}(r, \nu)[m]\big) 
     \ch{\fhd}\D\big(1, i_{\sh}(\nu)+ \gl'\big)[m].
  \] 
\end{Cor}

Before ending this section, let us prove the following lemma for later use:

\begin{Lem}\label{Lem: short_Demazure}
  Let $\gl \in P_+, m \in \Z$, and set $\gl' = \gl - i_{\sh}(\ol{\gl})$. 
  Then we have
  \[  P_{\gl - Q^{\sh}_+ + \Z \gd} \ch{\fhd}\D(1, \gl)[m] =e(\gl')i_{\sh} \ch{\fhd^{\sh}}\D^{\sh}(r,\ol{\gl})[m]. \]
\end{Lem}

\begin{proof}
  Let $v$ be the generator of $\D(1,\gl)[m]$ in Corollary \ref{Cor: rel_of_Demazure}, and set $D' = U(\Cgd^{\sh}).v$.
  Similarly as Lemma \ref{Lem: character}, we have
  \[ P_{\gl -Q^{\sh}_++\Z \gd} \ch{\fhd} \D(1, \gl)[m] = e(\gl')i_{\sh} \ch{\fhd^{\sh}} D'.
  \]
  We see from proposition \ref{Prop:_rel_of_Demazure} that $D'$ is a quotient of $\D^{\sh}(r, \ol{\gl})[m]$ 
  as a $\Cgd^{\sh}$-module.
  Hence we have 
  \[ P_{\gl - Q^{\sh}_+ + \Z \gd} \ch{\fhd}\D(1, \gl)[m] \le e(\gl')i_{\sh}\ch{\fhd^{\sh}}\D^{\sh}(r,\ol{\gl})[m].
  \]
  In order to show the opposite inequality, it is enough to show in the case $m = 0$. 
  Here we use the notation in the proof of Proposition \ref{Prop: Main_Theorem1}.
  Similarly as (\ref{eq: for_e}) and (\ref{eq: for_h}), we can show for each $1 \le i \le k$ that
  \[ e_{\gg} \otimes \C[t].X_i = 0 \ \text{for} \ \gg \in \gD_+ \setminus \lpishr_+\ \ \ \text{and} \ \ \
     (\fh^{\sh})^{\bot} \otimes t\C[t].X_i = 0,
  \]
  which implies
  \[ W_i = U(\Cgd).X_i = U(\fn_-' \otimes \C[t]).X_i,
  \]
  where we set $\fn'_- = \bigoplus_{\gg \in \gD_- \setminus \lpishr_-} \fg_{\gg}$.
  From this we have 
  \[ X_i \cap W_{i-1} = X_i \cap \big(U(\fn'_-\otimes \C[t]\big).X_{i-1}) = X_{i-1},
  \]
  and hence it follows that 
  \[ W_i /W_{i-1} \supseteq U(\Cgd^{\sh}).\overline{x}_i = X_i / X_i \cap W_{i-1} = X_i / X_{i-1} \cong \D^{\sh}(r, \nu_i)[m_i].
  \]
  Similarly as Lemma \ref{Lem: character}, we have
  \begin{align*} P_{\gl -Q^{\sh}_+ + \Z\gd} \ch{\fhd} W_i / W_{i -1} &= \ch{\fhd} U(\Cgd^{\sh}).\overline{x}_i \\
                                                                     &= e\big(\mu_i - i_{\sh}(\ol{\mu}_i)\big)
                                                                         i_{\sh}\ch{\fhd^{\sh}}\D^{\sh}(r, \nu_i)[m_i].
  \end{align*}
  Since $W_i / W_{i-1}$ is a quotient of $\D(1,\mu_i)[m_i]$ and
  $\mu_i - i_{\sh}(\ol{\mu}_i) = \gl - i_{\sh}(\ol{\gl}) = \gl'$, we have
  \begin{equation*}
     P_{\gl - Q^{\sh}_+ + \Z\gd} \ch{\fhd} \D(1,\mu_i)[m_i] \ge e(\gl')i_{\sh}\ch{\fhd^{\sh}}\D^{\sh}(r, \nu_i)[m_i].
  \end{equation*}
  In particular, we have
  \[  P_{\gl - Q^{\sh}_+ + \Z\gd} \ch{\fhd} \D(1,\gl)[0] \ge e(\gl')i_{\sh}\ch{\fhd^{\sh}}\D^{\sh}(r, \ol{\gl})[0]
  \]
  since there exists some $j$ such that $\nu_j = \ol{\gl}$ and $m_j = 0$, which is easily seen from the equation
  \[ \sum_{1 \le i \le k} \ch{\fhd^{\sh}} \D^{\sh}(r, \nu_i)[m_i]= \ch{\fhd^{\sh}}\D^{\sh}(1, \ol{\gl})[0].
  \]
\end{proof}

\section{Path models}\label{Path_crystals}

In this section, we review the theory of path models originally introduced by Littelmann \cite{MR1253196}, \cite{MR1356780}
(this theory can be applied to any Kac-Moody algebras, but we only state it for affine ones).
We do not review the definition of (abstract) crystals, but refer the reader to \cite[\S 4.5]{MR1881971}.

\subsection{Definition of path models}\label{Def_of_path}

For $a, b \in \R$ with $a < b$, we set $[a,b] = \{ t \in \R \mid a \le t \le b \}$.
A \textit{path with weight in $\hat{P}$} is, by definition, a piecewise linear, continuous map 
$\pi\colon [0,1] \to \hat{\fh}_{\R}^* = \R \otimes_\Z \hat{P}$
such that $\pi(0) = 0, \pi(1) \in \hat{P}$.
We denote by $\mP$ the set of all paths with weights in $\hat{P}$.
For $\pi_1,\pi_2 \in \mP$, define $\pi_1 + \pi_2 \in \mP$ by $(\pi_1+\pi_2)(t) = \pi_1(t) + \pi_2(t)$.

\begin{Rem} \normalfont
  In \cite{MR1253196} and \cite{MR1356780}, paths are considered modulo reparametrization. 
  In this article, however, we do not do this since there is no need to do so.
  Indeed, it can be checked that all the results in \cite{MR1253196} and \cite{MR1356780} 
  used in this article still hold in this setting.
\end{Rem}

Let $\pi \in \mP$.
A pair $\mugs$ of a sequence $\underline{\mu}: \mu_1,\mu_2, \dots, \mu_N$ of elements of $\hat{\fh}^*_\R$ and 
a sequence $\underline{\gs}: 0 = \gs_0 < \gs_1 < \dots < \gs_N =1$ of real numbers is called an \textit{expression} 
of $\pi$ if the following equation holds:
\[ \pi(t) = \sum_{p' = 1}^{p-1} (\gs_{p'} - \gs_{p'-1})\mu_{p'} + (t-\gs_{p-1})\mu_p \ \ \ 
   \text{for} \ t \in [\gs_{p-1}, \gs_p],\ 1 \le p \le N.
\]
In this case, we write $\pi = \mugs$.

For $\pi \in \mP$ and $i \in \hI$, define $H_i^\pi\colon [0,1] \to \R$ and $m_i^\pi \in \R$ by
\begin{equation*} 
  H_i^\pi(t) = \langle \pi(t), \ga^{\vee}_i \rangle, \ \ \ m_i^\pi = \min \{ H_i^\pi (t) \mid t \in [0,1] \}.
\end{equation*}
Denote by $\mPint$ the subset of $\mP$ consisting of paths $\pi$ such that for every $i \in I$, all local minimums 
of $H_i^\pi$ are integers.

Littelmann introduced root operators $\te_i, \tf_i$ $(i \in \hI)$ on $\mP$ in \cite{MR1253196} and \cite{MR1356780}
(in these articles, they are denoted by $e_\ga, f_\ga$).
Here for simplicity, we recall their actions only on elements of $\mPint$,
which are enough for this article since all the paths we consider below belong to $\mPint$.
For $\pi \in \mPint$ and $i \in \hI$, we define $\te_i \pi$ as follows: if $m_i^{\pi} = 0$, 
then $\te_i\pi = \0$ where $\0$ is an additional element corresponding to `$0$' in the 
theory of crystals. 
If $m_i^{\pi} \le -1$, then $\te_i \pi \in \mP$ is given by
\[ (\te_i\pi)(t) = \begin{cases}
                     \pi(t) & \text{for} \ t \in [0, t_0], \\
                     \pi(t_0) + s_i \big(\pi(t) - \pi(t_0)\big) & \text{for} \ t \in [t_0, t_1], \\
                     \pi(t) + \ga_i & \text{for} \ t \in [t_1, 1],
                   \end{cases}
\]
where we set
\begin{align*}
  t_1 &= \min\{t\in[0,1]\mid H_i^{\pi}(t) = m_i^{\pi}\}, \\
  t_0 &= \max\{t\in[0,t_1]\mid H_i^{\pi}(t) = m_i^{\pi} + 1\}.
\end{align*}
Similarly, we define $\tf_i\pi \in \mP \cup \{\0\}$ as follows:
if $H_i^{\pi}(1) = m_i^{\pi}$, then $\tf_i\pi = \0$.
If $H_i^{\pi}(1) \ge m_i^{\pi} +1$, then $\tf_i \pi$ is given by
\[ (\tf_i\pi)(t) = \begin{cases}
                     \pi(t) & \text{for} \ t \in [0, t_0], \\
                     \pi(t_0) + s_i \big(\pi(t)-\pi(t_0)\big) & \text{for} \ t \in [t_0, t_1], \\
                     \pi(t) - \ga_i & \text{for} \ t \in [t_1, 1],
                   \end{cases} \]
where we set
\begin{align*}
  t_0 & = \max\{t \in [0,1] \mid H_i^{\pi}(t) = m_i^{\pi} \}, \\
  t_1 & = \min\{t \in [t_0,1] \mid H_i^{\pi}(t) = m_i^{\pi} + 1 \}.
\end{align*}
We set $\mathrm{wt}(\pi)= \pi (1) \in \hP$ for $\pi \in \mPint$, and define $\gee_i\colon \mPint \to \Z_{\ge 0}$
and $\gph_i\colon \mPint \to \Z_{\ge 0}$ for $i \in \hI$ by
\[ \gee_i(\pi) = \max\{ k \in \Z_{\ge 0} \mid \te_i^k\pi \neq \0 \}, 
   \ \ \ \gph_i(\pi) = \max\{ k \in \Z_{\ge 0}\mid \tf_i^k\pi \neq \0 \}.
\]

\begin{Thm}[{\cite[\S 2]{MR1356780}}] \label{Thm: crystal_structure}
  Let $\mB$ be a subset of $\mPint$ such that $\te_i \mB \subseteq \mB \cup \{ \0 \}$ and
  $\tf_i \mB \subseteq \mB \cup \{ \0 \}$ for all $i \in \hI$.
  Then $\mB$, together with the root operators $\te_i, \tf_i$ for $i \in \hI$ and the maps $\mathrm{wt}, \gee_i, \gph_i$ 
  for $i \in \hI$,
  becomes a $U_q(\hat{\fg})$-crystal {\normalfont(\cite[Definition 4.5.1]{MR1881971})}.
  Moreover we have 
  \begin{equation} \label{eq: map_on_crystal}
    \gee_i(\pi) = -m_i^{\pi}, \ \ \ \gph_i(\pi) = H_i^{\pi}(1) - m_i^{\pi} \ \ \ \text{for} \ \pi \in \mB \ 
    \text{and} \ i \in \hI. 
  \end{equation}
\end{Thm}

The following lemma is easily checked from the definition of the root operators.

\begin{Lem}\label{Lem: root_operator}
  Let $\pi \in \mPint, i \in \hI$, and $0< u \le 1$ a real number. \\
  {\normalfont(i)} If $\pi$ satisfies $H_i^{\pi}(t) \ge m_i^\pi +1 $ for all $t\in [0, u]$, then we have
    \[ \te_i\pi(t) = \pi(t) \ \ \ \text{for all} \ t\in [0,u].
    \]
  {\normalfont(ii)} Let $M \in \Z_{\ge 0}$. If $\pi$ satisfies 
    \[ \tf_i(\pi) \neq \0, \ \ \ H_i^{\pi}(t) \ge -M \ \text{for all} \ t \in [0, u], \ \ \ \text{and} \ \ \ H_i^{\pi}(u) = -M,
    \]
    then we have 
    \[ \tf_i\pi(t) = \pi(t) \ \ \ \text{for all} \ t\in [0,u].
    \]
\end{Lem}

Assume that a subset $\mB \subseteq \mPint$ satisfies the assumption of Theorem \ref{Thm: crystal_structure}.
For each $i \in \hI$, we define $S_i\colon \mB \to \mB$ by
\[ S_i(\pi)= \begin{cases} \tf_i^\ell\pi & \text{if} \ \ell= \langle \pi(1), \ga_i^{\vee} \rangle \ge 0, \\
                          \te_i^{-\ell}\pi & \text{if} \ \ell= \langle \pi(1), \ga_i^{\vee} \rangle < 0.
            \end{cases}
\]

\begin{Thm}[{\cite[Theorem 8.1]{MR1356780}}]\label{Thm: action_of_Weyl}
  The map $s_i \mapsto S_i$ on the simple reflections of $\hW$ extends to a unique group action 
  of $\hat{W}$ on $\mB\colon w \mapsto S_w$. 
\end{Thm}

\begin{Lem}\label{Lem: action_of_Weyl}
  Let $\pi \in \mB$ and $i \in \hI$. 
  If $H_i^{\pi}$ is non-decreasing or non-increasing,
  then $S_i(\pi)$ satisfies
  \[ S_i(\pi)(t) =s_i\big(\pi(t)\big) \ \ \ \text{for all} \ t \in [0,1].
  \] 
\end{Lem}

\begin{proof}
  Assume that $H_i^{\pi}$ is non-decreasing.
  Then we have $\gph_i(\pi) = H_i^{\pi}(1) \in \Z_{\ge 0}$ by (\ref{eq: map_on_crystal}). 
  Put $M = \gph_i(\pi)$, and define $\gs_k \in [0,1]$ for each $k \in \{0,\dots,M\}$ by 
  \[ \gs_k = \max\{t \in [0,1] \mid H_i^{\pi}(t) = k\}.
  \]
   Then we have $0 \le \gs_0 < \gs_1 < \dots < \gs_M = 1$, and we can show inductively from the definition of $\tf_i$ that
   \[ \tf_i^k\pi(t) = \begin{cases} s_i\big(\pi(t)\big) & \text{for} \ t \in [0, \gs_k], \\
                                      \pi(t) -k\ga_i & \text{for} \ t \in [\gs_k, 1].
                        \end{cases}
   \] 
   Hence $\tf_i^M \pi(t) = s_i\big(\pi(t)\big)$ holds for all $t \in [0,1]$. 
   When $H_i^{\pi}$ is non-increasing, the assertion is proved similarly.
\end{proof}

We recall the definition of a concatenation of paths in $\mP$ (cf.\ \cite[\S 1]{MR1356780}).
For $\pi_1, \pi_2 \in \mP$, we define their \textit{concatenation} $\pi_1 * \pi_2 \in \mP$ by
\[ (\pi_1 * \pi_2)(t) =\begin{cases} \pi_1 (2t) & \text{if} \ t \in [0,\frac{1}{2}], \\
                                    \pi_1(1) + \pi_2(2t-1) & \text{if} \ t \in [\frac{1}{2}, 1].
                      \end{cases} \]
It is obvious that if $\pi_1, \pi_2 \in \mPint$, then $\pi_1*\pi_2 \in \mPint$.
For the notational convenience, we set $\0 * \pi = \pi * \0 = \0$ for any $\pi \in \mP$.

\begin{Lem}[{\cite[Lemma 2.7]{MR1356780}}] 
  For $\pi_1, \pi_2 \in \mPint$ and $i \in \hI$, we have
  \begin{equation} \label{eq: tensor_rule}
    \te_i (\pi_1 * \pi_2) = \begin{cases} (\te_i\pi_1)*\pi_2 & \text{if} \ \varphi_i(\pi_1) \ge \gee_i(\pi_2), \\
                                           \pi_1 * (\te_i\pi_2) & \text{if}\ \varphi_i(\pi_1) < \gee_i(\pi_2),
                             \end{cases}
  \end{equation}
  and
  \begin{equation} \label{eq: tensor_rule2}
    \tf_i (\pi_1 * \pi_2) = \begin{cases} (\tf_i\pi_1)*\pi_2 & \text{if} \ \varphi_i(\pi_1) > \gee_i(\pi_2), \\
                                           \pi_1 * (\tf_i\pi_2) & \text{if}\ \varphi_i(\pi_1) \le \gee_i(\pi_2).
                             \end{cases}
  \end{equation}
\end{Lem}

This lemma implies the following:
    
\begin{Prop}\label{Prop: concatenation}
  Assume that two subsets $\mB_1, \mB_2$ of $\mPint$ satisfy the assumption of Theorem \ref{Thm: crystal_structure}.
  Set 
  \[ \mB_1 * \mB_2 = \{ \pi_1 * \pi_2 \mid \pi_1 \in \mB_1, \pi_2 \in \mB_2 \} \subseteq \mPint.
  \]
  Then $\mB_1 * \mB_2$ also satisfies the assumption of Theorem \ref{Thm: crystal_structure},
  and $\mB_1 * \mB_2$ is isomorphic to $\mB_1 \otimes \mB_2$ as a $U_q(\hat{\fg})$-crystal
  {\normalfont(}see {\normalfont\cite[Definition 4.5.3]{MR1881971}}
  for a tensor product of crystals{\normalfont)},
  where the isomorphism is given by $\pi_1 * \pi_2 \mapsto \pi_1 \otimes \pi_2$.
\end{Prop}

\begin{Lem}\label{Lem: concatenation}
  Let $\mB_1, \mB_2$ be as in the above proposition, and let $\pi_1 \in \mB_1, \pi_2 \in \mB_2$.
  There exists some $p \in \Z_{\ge 0}$ such that
  \[ \te_i^{p+1} (\pi_1 * \pi_2) = (\te_i \pi_1) * (\te_i^p \pi_2) \ \ \ \text{and} \ \ \ \te_i^p \pi_2 \neq \0.
  \]
\end{Lem}

\begin{proof}
  We can see from (\ref{eq: tensor_rule}) that $p = \max\{0, \gee_i(\pi_2)-\gph_i(\pi_1)\}$ satisfies the assertion. 
\end{proof}

\subsection{Relations between crystal bases and path models}\label{subsection: crystal and bases}

The crystal bases of $U_q(\hfg)$-modules and $U_q'(\hfg)$-modules are 
typical and very important examples of crystals (see \cite[Chapter 4]{MR1881971}).
In this subsection, we review some realizations of crystal bases using path models.

We prepare some notation. 
For a $U_q(\hfg)$-crystal $\mcB$ and an element $b \in \mcB$, 
we denote by $C(b)$ the connected component of $\mcB$ containing $b$,
that is, the subset of $\mcB$ consisting of elements obtained from $b$ by 
applying Kashiwara operators several times. Note that $C(b)$ is a connected $U_q(\hfg)$-crystal.
For $\gl \in \hP$, denote by $\pi_{\gl}$ the straight line path $\pi_\gl(t) = t\gl$,
and write
\[ \mB_0(\gl) = C(\pi_{\gl}).
\]
It is known that $\mB_0(\gl) \subseteq \mPint$ for all $\gl \in \hat{P}$ (\cite[Lemma 4.5 (d) and Corollary 2]{MR1356780}).
It is easily seen from the definition of the root operators that for any $\pi = (\mu_1,\dots, \mu_N;\ul{\gs}) \in \mB_0(\gl)$,
we have $\mu_j \in \hat{W}\gl$ for all $1 \le j \le N$.

It is well-known that the integrable highest weight $U_q(\hfg)$-module $V_q(\gL)$ of highest weight $\gL \in \hP_+$ 
has a crystal basis, which we denote by $\mathcal{B}(\gL)$. 
Let $b_{\gL}$ denote the highest weight element of $\mathcal{B}(\gL)$.
From the construction of $\mcB(\gL)$ (\cite[Chapter 5]{MR1881971}), it follows that
\begin{equation}
  \mcB(\gL) = \{ \tf_{i_1} \cdots \tf_{i_s} b_{\gL} \mid s \ge 0, i_k \in \hI \} \setminus \{0\},
\end{equation}
and
\begin{equation} \label{eq: uniqueness_of_highest}
   \{ b \in \mcB(\gL) \mid \te_ib = 0 \ \text{for all} \ i \in \hI \} = \{b_{\gL}\}.
\end{equation}

\begin{Thm} \label{Thm: crystal_isom} Let $\gL \in \hP_+$. \\
  {\normalfont(i) (\cite[\S 7]{MR1356780})$\bm{.}$} 
    If $\pi \in \mP$ satisfies $m_i^{\pi} = 0$ for
    all $i \in \hI$ and $\pi(1) = \gL$, then
    there exists a unique isomorphism of $U_q(\hfg)$-crystals from $C(\pi)$ to $\mB_0(\gL)$
    which maps $\pi$ to $\pi_{\gL}$. \\
  {\normalfont(ii) (\cite{MR1395599}, \cite{MR1315966})$\bm{.}$} 
    There exists a unique isomorphism of $U_q(\hfg)$-crystals 
    from $\mB_0(\gL)$ to $\mathcal{B}(\gL)$ which maps $\pi_{\gL}$ to $b_{\gL}$.
\end{Thm}

\begin{Cor}\label{Cor: crystal_isom}
  If $\pi \in \mP$ satisfies $m_i^{\pi} = 0$ for all $i \in \hI$ and $\pi(1) = \gL \in \hP_+$,
  then there exists a unique isomorphism of $U_q(\hfg)$-crystals 
  from $C( \pi)$ to $\mathcal{B}(\gL)$ which maps $\pi$ to $b_{\gL}$.
\end{Cor}

\begin{Rem} \normalfont
  It is known that $\mB_0(\gL)$ with $\gL \in \hP_+$ coincides with the set $\mB(\gL)$ of all 
  Lakshmibai-Seshadri paths (LS paths for short) of shape $\gL$ (\cite{MR1356780}).
  We do not recall the definition of LS paths here, since we do not need it in this article.
\end{Rem}

Let $\hP_{\cl} = \hP/\Z \gd$, and denote the canonical projection $\hP \to \hP_{\cl}$ by $\cl$.
Note that $\hP_{\cl}$ is the weight lattice of $U_q'(\hfg)$.
A \textit{path with weight in $\hP_{\cl}$} is a piecewise linear, continuous map 
$\eta\colon [0,1] \to \R \otimes_{\Z} \hP_{\cl}$ 
such that $\eta(0) = 0$, $\eta(1) \in \hP_{\cl}$.
Let $\mP_{\cl}$ denote the set consisting of paths with weights in $\hP_{\cl}$.
We define an expression  of $\eta \in \mP_{\cl}$ similarly as that of a path with weight in $\hP$.
For $\xi \in \hP_{\cl}$, denote by $\eta_{\xi} \in \mP_{\cl}$ the straight line path
$\eta_{\xi} (t) = t \xi$.

For $\pi \in \mP$, $\cl (\pi) \in \mP_{\cl}$ is defined by $\cl(\pi)(t) = \cl\big(\pi(t)\big)$ for $t \in [0,1]$.
For $\gl \in P_+ ( \subseteq \hP)$, we define a subset $\mB(\gl)_\cl \subseteq \mP_{\cl}$ by
\[ \mB(\gl)_{\cl} = \{ \cl(\pi) \mid \pi \in \mB_0(\gl) \},
\]
which is known to be a finite set.
Now we define a $U_q'(\hfg)$-crystal structure on $\mB(\gl)_{\cl}$.
Define a weight map $\wt\colon \mB(\gl)_{\cl} \to \hP_{\cl}$ by $\wt(\eta) = \eta(1)$, root operators 
$\te_i, \tf_i\colon \mB(\gl)_{\cl} \to \mB(\gl)_{\cl} \cup \{ \0 \}$ for $i \in \hI$ by
\[ \te_i \big(\cl(\pi)\big) = \cl(\te_i\pi), \ \ \ \tf_i\big(\cl(\pi)\big) = \cl(\tf_i\pi),
\]
where $\cl(\0)$ is understood as $\0$,
and $\gee_i,\gph_i\colon \mB(\gl)_\cl \to \Z_{\ge 0}$ for $i \in \hI$ by 
$\gee_i\big(\cl(\pi)\big) = \gee_i(\pi)$, $\gph_i\big(\cl(\pi)\big) = \gph_i(\pi)$.
These maps are all well-defined, and $\mB(\gl)_{\cl}$ together with these maps becomes a finite $U_q'(\hfg)$-crystal 
(\cite[\S 3.3]{MR2199630}).

\begin{Rem} \normalfont
  For $\gl \in P_+$, $\mB_0(\gl)$ does not necessarily coincide with the set $\mB(\gl)$ of all LS paths of shape $\gl$.
  However, it is known that the set $\{ \cl (\pi) \mid \pi \in \mB(\gl)\}$ coincides with $\mB(\gl)_{\cl}$ 
  defined above (\cite[Lemma 4.5 (1)]{MR2407814}).
  This is why we use the notation `$\mB(\gl)_{\cl}$' in stead of `$\mB_0(\gl)_{\cl}$'.
\end{Rem}

In \cite[\S 5.2]{MR1890649}, Kashiwara introduced a finite-dimensional irreducible integrable $U_q'(\hfg)$-module
$W_q(\varpi_i)$ for each $i \in I$ called a \textit{level zero fundamental representation}, 
and proved that it has a crystal basis.
We denote the crystal basis of $W_q(\varpi_i)$ by $\mathcal{B}\big(W_q(\varpi_i)\big)$.
The following facts were verified by Naito and Sagaki:

\begin{Thm} \label{Thm: isom_of_finite_crystals}
  {\normalfont (i) (\cite[Theorem 3.2]{MR2146858})$\bm{.}$}
    Let $\gl = \sum_{i \in I} \gl_i \varpi_{i} \in P_+$.
    Then there exists a unique isomorphism of $U_q'(\hfg)$-crystals
    from $\mB(\gl)_{\cl}$ to $\bigotimes_{i \in I} \mB(\varpi_i)_{\cl}^{\otimes{\gl_i}}$ 
    {\normalfont(}which does not depend on the choice of the ordering of the tensor factors up to isomorphism{\normalfont)}. \\
  {\normalfont (ii) (\cite[Corollary of Theorem 1]{MR2199630})$\bm{.}$} For all $i \in I$,
    $\mB(\varpi_i)_{\cl}$ is isomorphic to $\mcB\big(W_q(\varpi_i)\big)$ as a $U_q'(\hfg)$-crystal. 
\end{Thm}

\subsection{Degree function on $\mB(\gl)_{\cl}$}\label{Degree function}

For $\pi = (\mu_1,\dots,\mu_N; \ul{\gs}) \in \mP$, we call $\mu_1 \in \R \otimes_{\Z} \hP$
the initial direction of $\pi$ (which does not depend on the choice of an expression),
and denote the initial direction of $\pi$ by $\gi(\pi)$.
The initial direction of $\eta \in \mP_{\cl}$ is defined similarly, and denoted by $\gi(\eta) \in \R \otimes _\Z\hP_\cl$.
Note that $\iota\big(\cl(\pi)\big) = \cl\big(\iota(\pi)\big)$ holds.

For $\gl \in P_+$, let $d_{\gl}$ be the nonnegative integer satisfying 
\[ \hW \gl \cap (\gl + \Z \gd) = \gl + d_{\gl} \Z \gd.
\]
Then we have 
\begin{equation}\label{eq: id_of_Weyl_gp}
   \hW \gl = W \gl + d_{\gl}\Z  \gd.
\end{equation}

\begin{Lem} \label{Lem: contain_wt}
  Let $\gl \in P_+$. \\
  {\normalfont(i)} For any $\pi \in \mB_0(\gl)$, we have $\wt(\pi) \in \gl - Q_+ + \Z \gd$. \\
  {\normalfont(ii)} For any $\eta \in \mB(\gl)_{\cl}$, we have $\wt(\eta) \in \cl(\gl -Q_+)$.
\end{Lem}

\begin{proof}
  Let $(\mu_1,\dots,\mu_N; \ul{\gs})$ be an expression of $\pi$. 
  Since $\mu_j \in \hW\gl \subseteq W \gl + \Z \gd \subseteq \gl - Q_+ + \Z \gd$ for each $j$, we have
  \[ \wt(\pi)=\sum_{p = 1}^N (\gs_p - \gs_{p-1})\mu_p \in \gl - \sum_{i \in I} \R_{\ge 0} \ga_i + \R\gd.
  \]
  On the other hand, we have $\wt(\pi) \in \gl + \sum_{i \in \hI}\Z \ga_i = \gl + Q + \Z\gd$ 
  by the definition of $\mB_0(\gl)$.
  Hence (i) follows. The assertion (ii) obviously follows from (i).
\end{proof}

\begin{Lem} \label{Lem: lem_for_degree_fucnction}
  Let $\eta \in \mB(\gl)_{\cl}$. \\
  {\normalfont(i)} For arbitrary $\pi \in \mB_0(\gl) \cap \cl^{-1}(\eta)$, we have
    \[ \mB_0(\gl) \cap \cl^{-1}(\eta) = \{ \pi + \pi_{kd_{\gl}\gd} \mid k \in \Z\}.
    \]
  {\normalfont(ii)} Assume $\mu \in \hW\gl$ satisfies $\cl(\mu) = \iota(\eta)$.
    Then there exists a unique element $\pi' \in \mB_0(\gl) \cap \cl^{-1} (\eta)$ such that $\gi(\pi') = \mu$.
\end{Lem}

\begin{proof}
  The assertion (i) is just \cite[Lemma 4.5]{MR2146858}.
  Let us prove the assertion (ii). 
  Let $\pi$ be an arbitrary element of $\mB_0(\gl) \cap \cl^{-1}(\eta)$.
  Since $\cl\big(\iota(\pi)\big) = \iota(\eta) = \cl(\mu)$ and $\gi(\pi) \in \hat{W}\gl$, 
  there exists some $s \in \Z$ such that $\iota(\pi) = \mu + sd_{\gl}\gd$ by (\ref{eq: id_of_Weyl_gp}).
  Then by (i), $\pi - \pi_{sd_{\gl} \gd}$ is the unique element of $\mB_0(\gl) \cap \cl^{-1}(\eta)$ 
  whose initial direction is $\mu$.  
  The assertion is proved.
\end{proof}

Now we recall the definition of the degree function on $\mB(\gl)_{\cl}$ $(\gl \in P_+)$ introduced in \cite{MR2474320}.
Let $i_{\cl}\colon \hP_{\cl} \to \hP$ be the unique section of $\cl$ satisfying $i_{\cl}(\hP_\cl) = P + \Z \gL_0$.
For $\eta \in \mB(\gl)_{\cl}$, we denote by $\pi_{\eta}$ 
the element in $\mB_0(\gl) \cap \cl^{-1}(\eta)$ such that $\iota(\pi_{\eta}) \in W \gl$, which is unique by 
Lemma \ref{Lem: lem_for_degree_fucnction} (ii),
and define $\Deg(\eta) \in \Z$ by an integer satisfying
\[ \wt(\pi_{\eta}) = i_{\cl} \circ \wt(\eta) - \gd \Deg(\eta).
\]
We call $\Deg(\eta)$ the \textit{degree} of $\eta$.
By \cite[Lemma 3.1.1]{MR2474320}, $\Deg(\eta) \le 0$ follows for all $\eta \in \mB(\gl)_{\cl}$.

\begin{Rem}\label{Rem: energy_function} \normalfont
  The main theorem in \cite{MR2474320} says that the degree function on $\mB(\gl)_{\cl}$ 
  where $\gl = \sum_{i \in I} \gl_i \varpi_i$ can be 
  identified with the energy function (see \cite{MR1903978}, \cite{MR1745263}) 
  on $\bigotimes_{i \in I} \mcB(W_q(\varpi_i))^{\otimes\gl_i}$
  up to some constant through the isomorphism given in Theorem \ref{Thm: isom_of_finite_crystals}.
\end{Rem}

For $\eta \in \mB(\gl)_{\cl}$, define
\[ \wt_{\hP} (\eta)= \wt(\pi_{\eta}) = i_{\cl} \circ \wt (\eta) - \gd \Deg(\eta) \in \hP.
\]
We call $\wt_{\hP}(\eta)$ the \textit{$\hP$-weight} of $\eta$.

\begin{Rem} \label{Rem: remark_of_character}\normalfont
  By the definition of $\pi_{\eta}$, we have
  \[ \{ \pi_{\eta} \mid \eta \in \mB(\gl)_{\cl} \} = \{ \pi \in \mB_0(\gl) \mid \iota(\pi) \in P \}.
  \]
  Hence it follows that
  \[ \sum_{\eta \in \mB(\gl)_{\cl}} e\big(\wt_{\hP}(\eta)\big)
     =\sum_{\eta \in \mB(\gl)_{\cl}} e\big(\wt(\pi_{\eta})\big) =
     \sum_{\begin{smallmatrix} \pi \in \mB_0(\gl) \\ \iota(\pi) \in P \end{smallmatrix}} e\big(\wt(\pi)\big).
  \]
\end{Rem}

\begin{Lem} \label{Lem: lemma_of_pi}
  The following statements hold for every $\eta \in \mB(\gl)_{\cl}$, where $\pi_{\0}$ is understood as $\0$.\\
  {\normalfont(i)} We have $\te_i \pi_{\eta} = \pi_{\te_i\eta}$ and $\tf_i\pi_{\eta} = \pi_{\tf_i\eta}$ for $i \in I$.\\
  {\normalfont(ii)} If $\te_0 \eta \neq \0$, then we have $\tf_0\pi_{\eta} = \pi_{\tf_0\eta}$.
\end{Lem}

\begin{proof}
  (i) By the definition of the root operator, we have $\iota(\te_i\pi_{\eta}) = \iota(\pi_\eta)$ or
  $\iota(\te_i\pi_{\eta})=s_i \iota(\pi_{\eta}).$
  Hence $\iota(\te_i\pi_{\eta}) \in W\gl$ follows. 
  Then since 
  \[ \cl(\te_i\pi_{\eta}) = \te_i\cl(\pi_{\eta}) = \te_i \eta,
  \] 
  $\te_i\pi_{\eta} = \pi_{\te_i\eta}$ holds by definition.
  The proof of $\tf_i \pi_{\eta} = \pi_{\tf_i \eta}$ is similar.
  (ii) If $\tf_0 \eta = \0$, (ii) follows since $\tf_0 \eta = \cl( \tf_0 \pi_{\eta})$. 
  Assume $\tf_0 \eta \neq \0$, which also implies $\tf_0 \pi_\eta \neq \0$.
  Then it follows from the assumption and \cite[Lemma 5.3]{MR1253196} that 
  \[ \iota(\tf_0\pi_{\eta}) = \iota(\pi_{\eta}) \in W \gl.
  \]
  From this and $\cl (\tf_0\pi_{\eta}) = \tf_0 \eta$, $\tf_0 \pi_{\eta} = \pi_{\tf_0 \eta}$ holds by definition.
\end{proof}

\section{Decomposition of $\mcB(\gL) \otimes \mB_0(\gl)$}

Throughout this section, we assume that $\gL \in \hP_+ \setminus \Z \gd$ and $\gl \in P_+$.
Our goal in this section is to show that  $\mcB(\gL) \otimes \mB_0(\gl)$ is isomorphic 
to a direct sum of the crystal bases of integrable highest weight $U_q(\hfg)$-modules.
This result is motivated by the tensor product rule in \cite{MR1356780} (see also \cite{MR2015316}).

\subsection{Technical lemmas}
Here we prepare several technical lemmas.

\begin{Lem} \label{Lem: Technical_Lemma}
  Let $u$ be a real number such that $0< u \le 1$, and assume that $\pi \in \mB_0(\gl)$ satisfies 
  \begin{equation} \label{eq: assumption}
    H_i^{\pi}(t) \ge - \langle \gL,\ga_i^{\vee}\rangle \ \ \ \text{for all} \ t \in [0,u], \ i \in \hI.
  \end{equation}
  {\normalfont(i)}
  For any $i \in \hI$ such that $\te_i(\pi_{\gL} * \pi) \neq \0$, we have $\te_i\pi \neq \0$ and
  \begin{equation} \label{eq: identity}
    \te_i \pi(t) = \pi(t)  \ \ \ \text{for all} \ t \in [0,u].
  \end{equation}
  {\normalfont(ii)}
  For any sequence $i_1, \dots, i_k$ of elements of $\hI$ such that $\te_{i_1} \cdots \te_{i_k}(\pi_{\gL} * \pi) \neq \0$,
  we have $\te_{i_1} \cdots \te_{i_k} \pi \neq \0$ and 
  \[ \te_{i_1} \cdots \te_{i_k} \pi(t) = \pi(t) \ \ \ \text{for all} \ t \in [0,u].
  \]
\end{Lem}

\begin{proof}
  It suffices to show (i) only since (ii) follows inductively from (i).
  Since $\te_i (\pi_{\gL} * \pi) \neq \0$ we have $m_i^{\pi_{\gL} * \pi} \le -1$, 
  which implies from the definition of the concatenation that $m_i^{\pi} \le -\langle \gL, \ga_i^{\vee}\rangle -1$.
  Hence $\te_i\pi \neq \0$ follows, and the Lemma \ref{Lem: root_operator} (i) together with the assumption 
  (\ref{eq: assumption}) implies (\ref{eq: identity}).
\end{proof}

\begin{Lem}
  There exist a positive integer $N$ and a sequence of real numbers $\ul{\gs}: 0= \gs_0<\gs_1<\dots<\gs_N=1$
  such that any $\pi \in \mB_0(\gl)$ has a unique expression in the form $(\mu_1,\dots,\mu_N; \ul{\gs})$.
\end{Lem}

\begin{proof}
  Since the set $\{ \pi_{\eta} \mid \eta \in \mB(\gl)_\cl\}$ is finite, 
  it is easily seen that there exist a positive integer $N$ and a sequence of real numbers
  $\ul{\gs}: 0 = \gs_0 <\gs_1<\dots<\gs_N=1$ such that for any $\eta \in \mB(\gl)_\cl$, $\pi_\eta$
  has an expression in the form $(\mu_1, \ldots,\mu_N; \ul{\gs})$.
  Note that we have
  \[ \mB_0(\gl) = \{ \pi_\eta + \pi_{kd_\gl\gd}\mid \eta \in \mB(\gl)_\cl, k \in \Z \}
  \]
  by Lemma \ref{Lem: lem_for_degree_fucnction} and 
  \[ (\mu_1, \ldots,\mu_N; \ul{\gs}) + \pi_{kd_\gl \gd} = (\mu_1 + kd_\gl \gd, \ldots, \mu_N + kd_\gl\gd; \ul{\gs}).
  \]
  Hence any $\pi \in \mB_0(\gl)$ has an expression in the form $(\mu_1,\ldots,\mu_N;\ul{\gs})$.
  Since the uniqueness is obvious, the assertion is proved.
\end{proof}

If a sequence $\ul{\gs}:0= \gs_0<\dots<\gs_N =1$ satisfies the condition of this lemma, 
we say that $\ul{\gs}$ is \textit{sufficiently fine} for $\gl$.

\begin{Rem} \normalfont
  Assume that $\ul{\gs}$ is sufficiently fine for $\gl$.
  Then for any $\pi \in \mB_0(\gl)$, it is obvious that the function $H_i^{\pi}(t)$ is 
  strictly increasing, strictly decreasing, or constant on each $[\gs_p,\gs_{p+1}]$. 
\end{Rem}

\begin{Lem} \label{Lem: sufficiently_fine}
  Assume that $\ul{\gs}: 0= \gs_0<\dots<\gs_N =1$ is sufficiently fine for $\gl$.
  Then for any $\pi \in \mB_0(\gl), i \in \hI$ and $M \in \Z_{\ge 0}$, 
  we have 
  \[ \max\{u \in [0,1]\mid H_i^\pi(t) \ge -M \ \text{for all} \ t \in [0,u]\} \in \{\gs_0, \gs_1,\dots,\gs_N\}.
  \]
\end{Lem}

\begin{proof}
  Set
  \[ u_0 = \max\{u \in [0,1]\mid H_i^\pi(t) \ge -M \ \text{for all} \ t \in [0,u]\},
  \]
  and assume that $u_0 \notin \{\gs_0, \gs_1,\dots,\gs_N\}$.
  Since $m_i^{\pi} \ge -M$ implies $u_0 = 1 = \gs_N$, we have $m_i^{\pi} \le -M-1$.
  Let $p \in \{ 0,\dots, N-1\}$ be the number such that $\gs_p < u_0 <\gs_{p+1}$.
  By the definition of $u_0$ and the assumption that $\ul{\gs}$ is sufficiently fine, we have that 
  \[ H_i^{\pi}(\gs_p) > -M, \ H_i^{\pi}(u_0) = -M \ \text{and} \ H_i^{\pi}(\gs_{p+1})<-M.
  \] 
  Let $q = -M - m_i^{\pi}$.
  We have from (\ref{eq: map_on_crystal}) that 
  \[ m^{\te_i^{r}\pi}_i \le -M-1 \ \text{for} \ r <q \ \text{and} \ m^{\te_i^{q}\pi} = -M.
  \]   
  From this and $H_i^\pi(t) \ge -M$ for all $t \in [0,u_0]$, we can show inductively using 
  Lemma \ref{Lem: root_operator} (i) that 
  \[ \te_i^q\pi(t) = \pi(t) \ \ \ \text{for all} \ t \in [0,u_0].
  \]
  Hence we have $H_i^{\te_i^q\pi}(\gs_p) > -M$ and $H_i^{\te_i^q\pi}(u_0) =-M$.
  On the other hand, we have $H_i^{\te_i^q\pi}(\gs_{p+1}) \ge -M$ since $m_i^{\te_i^q\pi} = -M$, 
  which contradicts the assumption that $\ul{\gs}$ is sufficiently fine.
  Hence the lemma follows.
\end{proof}

For a subset $J$ of $\hI$, we denote by $\hW_J$ the subgroup of $\hW$ generated by simple reflections $\{ s_i \mid i \in J\}$.
It is well-known that $\hW_J$ is finite if $J$ is proper.

\begin{Lem} \label{Lem: finiteness}
  Let $J$ be a proper subset of $\hI$. 
  Then for any $\pi \in \mB_0(\gl)$, the set
  $\{ \te_{i_1}\cdots\te_{i_s}\pi\mid s\ge 0, \ i_k \in J\} \setminus \{ \0 \}$
  is finite.
\end{Lem}

\begin{proof}
  Assume that $\ul{\gs}: 0=\gs_0<\dots<\gs_N=1$ is sufficiently fine for $\gl$, 
  and let $(\mu_1,\dots,\mu_N; \ul{\gs})$ be the expression of $\pi$.
  By the definition of root operators and the sufficiently fineness of $\ul{\gs}$, it follows that
  \[ \{ \te_{i_1}\cdots\te_{i_s}\pi\mid s\ge 0, \ i_k \in J \} \setminus \{ \0 \} \subseteq 
     \{ (w_1\mu_1,\dots,w_N\mu_N; \ul{\gs}) \mid w_j \in \hW_J\}.
  \]
  Since $\hW_J$ is a finite set so is the right hand side, which proves the assertion.
\end{proof}

\begin{Lem} \label{Lem: properness}
  For any $\pi \in \mB_0(\gl)$ and $u \in [0,1]$, a subset $\{i \in \hI \mid H_i^\pi(u) = -\langle \gL, \ga_i^{\vee}\rangle \}$
  of $\hI$ is proper.
\end{Lem}

\begin{proof}
  Let $(\mu_1,\dots,\mu_s;\ul{\gs})$ be an expression of $\pi$.
  Recall that if $\gs_{p-1} \le u \le \gs_p$, we have
  \[ \pi(u)= \sum_{1 \le p' \le p-1} (\gs_{p'} - \gs_{p'-1}) \mu_{p'} + (u - \gs_{p-1})\mu_p.
  \]
  Since $\mu_j \in \hW\gl = W\gl + d_{\gl}\Z \gd$, $\langle \mu_j, K \rangle = 0$ holds for all $1\le j \le s$.
  Hence $\langle \pi(u), K \rangle = 0$ holds.
  Then we have
  \[ \langle \pi(u) + \gL, K \rangle = \langle \gL, K \rangle > 0,
  \]
  which implies $\langle \pi(u) + \gL, \ga_i^{\vee} \rangle \neq 0$ for some $i \in \hI$.
  The assertion is proved.
\end{proof}

\subsection{Decomposition of $\mcB(\gL) \otimes \mB_0(\gl)$}

Set
\[ \mB_0(\gl)^{\gL} = \{ \pi \in \mB_0(\gl) \mid m_i^{\pi} \ge -\langle \gL, \ga_i^{\vee}\rangle \ \text{for all} \ i \in \hI \}.
\]
Note that $\gL + \pi_0(1) \in \hP_+$ if $\pi_0 \in \mB_0(\gl)^{\gL}$.

\begin{Prop}\label{Prop: equality}
  We have 
  \[ \mB_0(\gL) * \mB_0(\gl) = \bigoplus_{\pi_0 \in \mB_0(\gl)^{\gL}} C(\pi_{\gL} * \pi_0).
  \]
\end{Prop}

\begin{proof}
  If $\pi_0 \in \mB_0(\gl)^{\gL}$, $\pi_{\gL} * \pi_0$ satisfies 
  $m_i^{\pi_{\gL}*\pi_0} = 0$ for all $i \in \hI$, which implies by Corollary \ref{Cor: crystal_isom}
  that there exists an isomorphism of $U_q(\hfg)$-crystals
  from $C (\pi_{\gL} * \pi_0)$ to $\mcB\big(\gL +\pi_0(1)\big)$ that maps $\pi_{\gL} * \pi_0$ to $b_{\gL + \pi_0(1)}$.
  From this isomorphism and (\ref{eq: uniqueness_of_highest}), we have that
  \[ \{ \pi \in C(\pi_{\gL} * \pi_0) \mid \te_i\pi = 0 \ \text{for all} \ i \in \hI\} = \{ \pi_{\gL} * \pi_0 \},
  \]
  which implies that a sum $\bigcup_{\pi_0 \in \mB_0(\gl)^{\gL}} C(\pi_{\gL} * \pi_0)$ is disjoint. Hence we have
  \[  \mB_0(\gL) * \mB_0(\gl) \supseteq \bigoplus_{\pi_0 \in \mB_0(\gl)^{\gL}} C(\pi_\gL * \pi_0).
  \]
  We need to show the opposite containment. 
  By Lemma \ref{Lem: concatenation}, we can see for arbitrary $\pi_1 * \pi_2 \in \mB_0(\gL) * \mB_0(\gl)$ that there exists
  a sequence $j_1, \dots, j_s$ of elements of $\hI$ such that
  \[ \te_{j_1} \cdots \te_{j_s} (\pi_1 * \pi_2) = \pi_{\gL} * \pi_2'\ \ \ \text{for some} \ \pi_2' \in \mB_0(\gl).
  \] 
  Hence to show the opposite containment, it suffices to show for any $\pi_2 \in \mB_0(\gl)$ that 
  there exists a sequence $i_1,\dots,i_k$ of elements of $\hI$ such that
  \[ \te_{i_1}\cdots\te_{i_k} (\pi_\gL * \pi_2) = \pi_{\gL} * \pi_0 \ \ \ \text{for some} \ \pi_0 \in \mB_0(\gl)^{\gL}.
  \]
  Let $\ul{\gs}: 0= \gs_0 < \gs_1< \cdots < \gs_N =1$ be a sufficiently fine sequence for $\gl$.
  By Lemma \ref{Lem: sufficiently_fine}, there exists some $0 \le p_0 \le N$ such that
  \[ \gs_{p_0} = \max\{u \in [0,1] \mid H_i^{\pi_2}(t) \ge - \langle \gL, \ga_i^{\vee}\rangle \ 
     \text{for all} \ t \in [0,u], \ i \in \hI \}.
  \]
  We shall show the above assertion by the descending induction on $p_0$.
  If $p_0 = N$, we have $\pi_2 \in \mB_0(\gl)^{\gL}$ and there is nothing to prove.
  Assume $p_0 < N$, and set
  \[ J = \{ i \in \hI \mid H_i^{\pi_2}(\gs_{p_0}) = - \langle \gL, \ga_i^{\vee}\rangle \},
  \]
  which is a proper subset of $\hI$ by Lemma \ref{Lem: properness}.
  It obviously follows that
  \begin{align*}
    \{\te_{i_1}\cdots \te_{i_s} (\pi_{\gL} * \pi_2) & \mid s \ge 0, i_k \in J\}\setminus \{\0\} \\
                                                    & \subseteq \pi_{\gL} * \{ \te_{i_1}\cdots \te_{i_s} \pi_2 \mid s \ge 0, 
                                                       i_k \in J\}\setminus \{ \0\},
  \end{align*}
  and since the right hand side is a finite set by Lemma \ref{Lem: finiteness}, so is the left hand side. 
  Hence we can see by weight consideration that there exists 
  \[ \pi_{\gL} * \pi_2' \in  \{\te_{i_1}\cdots \te_{i_s} (\pi_{\gL} * \pi_2)  \mid s \ge 0, i_k \in J\}\setminus \{\0\}
  \]
  such that 
  \begin{equation} \label{eq: containment_in J}
   \te_i(\pi_\gL * \pi_2') = \0 \ \text{for all} \ i \in J.
  \end{equation}
  Note that we have from Lemma \ref{Lem: Technical_Lemma} (ii) and the definition of $p_0$ that
  \begin{equation} \label{eq: equality}
    \pi_2'(t)  = \pi_2(t) \ \ \ \text{for all} \ t \in [0, \gs_{p_0}].
  \end{equation}
  Let $0 \le p_0' \le N$ be an integer such that 
  \[ \gs_{p_0'} = \max\{u \in [0,1] \mid H_i^{\pi_2'}(t) \ge - \langle\gL,\ga_i^{\vee}\rangle \ \text{for all} 
     \ t \in [0,u], \ i \in \hI \}.
  \]
  We have $m_i^{\pi_\gL*\pi_2'} =0$ for $i \in J$ by (\ref{eq: containment_in J}), which implies that 
  $H_i^{\pi_2'} (t)\ge -\langle \gL, \ga_i^{\vee}\rangle$ for all $i \in J$ and $t \in [0,1]$.
  On the other hand if $i \in \hI \setminus J$, we have from (\ref{eq: equality}) and the definition of $J$ that 
  \[ H_i^{\pi_2'} (t) \ge - \langle\gL,\ga_i^{\vee}\rangle \ \text{for all} \ t\in[0,\gs_{p_0}] \ \ \ 
     \text{and} \ \ \ H_i^{\pi_2'}(\gs_{p_0}) > -\langle\gL,\ga_i^{\vee}\rangle.
  \]
  Hence we have $p_0' > p_0$, which together with the induction hypothesis completes the proof of the assertion.
  Now the proposition is proved.
\end{proof}

By Proposition \ref{Prop: concatenation}, Theorem \ref{Thm: crystal_isom}
and Corollary \ref{Cor: crystal_isom}, the above proposition implies the following corollary.
This is some sort of generalization of \cite[Theorem 1.6]{MR2015316} 
in which $\fg = \mathfrak{sl}_{\ell+1}$ and $\gl = m\varpi_1$ are assumed, 
and \cite[Theorem 4.3.2]{NS} in which $\fg$ is general and $\gl$ is a minuscule fundamental weight.

\begin{Cor} \label{Cor: decomposition}
  We have
  \[ \mcB(\gL) \otimes \mB_0(\gl) \stackrel{\sim}{\to} \bigoplus_{\pi_0 \in \mB_0(\gl)^{\gL}} \mcB\big(\gL + \pi_0(1)\big)
  \]
  as $U_q(\hfg)$-crystals, where the given isomorphism maps 
  $b_{\gL} \otimes \pi_0$ for each $\pi_0 \in \mB_0(\gl)^{\gL}$ to $b_{\gL + \pi_0(1)} \in \mcB(\gL + \pi_0(1))$. 
\end{Cor}

Note that we have shown the following fact in the proof of Proposition \ref{Prop: equality}, 
which is used again in the next section:

\begin{Lem} \label{Lem: used_again}
  Assume that $\ul{\gs}: 0= \gs_0 <\dots< \gs_N =1$ is sufficiently fine for $\gl$ and $\pi_2 \in \mB_0(\gl)$.
  Let $0 \le p_0 \le N$ be an integer such that 
  \[ \gs_{p_0} = \max \{ u \in [0,1] \mid H_i^{\pi_2}(t) \ge - \langle\gL,\ga_i^{\vee}\rangle \ 
  \text{for all} \ t \in [0,u], i \in \hI \},
  \]
  and assume $ p_0 < N$.
  Let $J = \{ i \in \hI \mid H_i^{\pi_2}(\gs_{p_0}) = -\langle\gL,\ga_i^{\vee}\rangle \}$.
  Then there exists $\pi_2' \in \mB_0(\gl)$ such that
  \[ \pi_{\gL} * \pi_2' \in \{ \te_{i_1}\cdots\te_{i_s}(\pi_\gL * \pi_2) \mid s \ge 0, i_k \in J \} \setminus \{ \0 \}
  \]
  and $p_0' > p_0$, where $p_0'$ denotes the integer defined by
  \[ \gs_{p_0'} = \max \{ u \in [0,1] \mid H_i^{\pi_2'}(t) \ge - \langle\gL,\ga_i^{\vee}\rangle \ 
     \text{for all} \ t \in [0,u], i \in \hI \}.
  \]
\end{Lem}

\section{Relations among Demazure crystals, $\mB_0(\gl)$ and $\mB(\gl)_{\cl}$}

\subsection{Demazure crystals}

For a $U_q(\hfg)$-crystal $\mcB$ and a subset $\mathcal{C} \subseteq \mcB$, 
we define a subset $\mcF_i\mathcal{C}$ of $\mcB$ for $i \in \hI$ by
\[ \mcF_i\mathcal{C} = \{ \tf_i^k b \mid b \in \mathcal{C}, k \ge 0 \} \setminus \{ 0 \},
\] 
and for a sequence $\mathbf{i} : i_1, i_2,\dots, i_m$ of elements of $\hI$,
we define $\mathcal{F}_{\mathbf{i}}\mathcal{C}$ by
\[ \mathcal{F}_{\mathbf{i}}\mathcal{C} = \mcF_{i_1}\mcF_{i_2}\cdots\mcF_{i_m}\mathcal{C}.
\]
For the notational convenience, we set $\mcF_{\emptyset} \mathcal{C} = \mathcal{C}$
and $\mcF_{\mathbf{i}} b= \mcF_{\mathbf{i}} \{ b \}$ for $b \in \mcB$.

Let $\gL \in \hP_+$, $w \in \hW$, and $w = s_{i_1}\cdots s_{i_m}$ a reduced expression.

\begin{Prop}[\cite{MR1240605}]
  The subset 
  \[ \mcB_w(\gL) = \mathcal{F}_{i_1,\dots,i_m} b_{\gL} \subseteq \mcB(\gL)
  \]
  is independent of the choice of the reduced expression. 
\end{Prop}

\begin{Def} \normalfont
  We call $\mcB_w(\gL)$ the \textit{Demazure crystal} associated with $\gL$ and $w$.
  (Note that $\mcB_w(\gL)$ does not have a $U_q(\hfg)$-crystal structure).
\end{Def}

Demazure crystals are known to have the following properties:

\begin{Prop}[\cite{MR1240605}] \label{Prop: Demazure crystal2}
  {\normalfont(i)} For any $i \in \hI$, we have $\te_i \mcB_w(\gL) \subseteq \mcB_w(\gL) \cup \{ 0 \}$.\\
  {\normalfont(ii)} We have
     \[  \ch{\hfh}V_w(\gL) =\sum_{b \in \mcB_w(\gL)} e\big(\wt(b)\big).
     \]
\end{Prop}

\begin{Lem} \label{Lem: Demazure_crystal}
  {\normalfont(i)} For any $i \in \hI$ and $w \in \hW$, we have
      \[ \mcF_{i} \mcB_w(\gL) = \begin{cases} \mcB_{s_iw}(\gL) & \text{if} \ s_i w > w, \\
                                              \mcB_{w}(\gL) & \text{if} \ s_i w < w.
                                \end{cases}
      \]
  {\normalfont(ii)} For an arbitrary sequence $\mathbf{i}:i_1,\dots,i_m$ of elements of $\hI$,
                    there exists some $w \in \hW$ such that $\mcF_{\mathbf{i}}b_{\gL} =\mcB_w(\gL)$.
\end{Lem}

\begin{proof}
  If $s_i w > w$, (i) follows by definition.
  If $s_i w < w$, then $w$ has a reduced expression in the form $w =s_i s_{j_1}\cdots s_{j_m}$ 
  by the exchange condition (\cite[Lemma 3.11]{MR1104219}),
  and hence (i) follows since
  \[ \mcF_{i} \mcB_w(\gL) = \mcF_{i} \mcF_{i,j_1,\dots,j_m} b_{\gL} = \mcF_{i,j_1,\dots,j_m} b_{\gL} = \mcB_w(\gL).
  \]
  The assertion (ii) can be shown inductively from (i).
\end{proof}

\begin{Prop}\label{Prop: stability}
  Let $J$ be a proper subset of $\hI$, and $\mathbf{i}: i_1,\dots,i_m$ a sequence of $\hI$.
  We assume that there exists some $1\le m' \le m$ such that $i_k \in J$ for all $1 \le k \le m'$ and 
  $s_{i_1} \cdots s_{i_{m'}}$ is a reduced expression of the longest element of $\hW_J$.
  Then there exists some element $w \in \hW$ satisfying 
  \[ \mcF_{\mathbf{i}}b_{\gL} = \mcB_w(\gL), \ \ \ \text{and} \ \ \ \langle w \gL, \ga_i^{\vee} \rangle \le 0 \ \ \ 
     \text{for all} \ i \in J.
  \]
  Moreover, this Demazure crystal $\mcB_w(\gL)$ satisfies
  \begin{equation}\label{eq: stability}
    \tf_i \mcB_w(\gL) \subseteq \mcB_w(\gL) \cup \{ 0 \} \ \ \ \text{for all} \ i \in J.
  \end{equation}
\end{Prop}  

\begin{proof}
  It suffices to show that there exists $w \in \hW$ satisfying
  \begin{equation}\label{eq:condition}
    \mcF_{\mathbf{i}}b_{\gL} = \mcB_w(\gL), \ \ \ \text{and} \ \ \ s_iw<w \ \ \ \text{for all} \ i \in J.
  \end{equation}
  Indeed if this is true, since $s_iw<w$ if and only if $w^{-1}\ga_i$ is a negative root of $\hat{\gD}$
  (\cite[Lemma 3.11]{MR1104219}), it follows that
  \[ \langle w\gL, \ga_i^{\vee} \rangle = \langle \gL, w^{-1}\ga_i^{\vee} \rangle \le 0 \ \text{for} \ i \in J,
  \]
  and (\ref{eq: stability}) also follows from Lemma \ref{Lem: Demazure_crystal} (i).
  We shall show by the descending induction on $k$ that there exists an element $w_k \in \hW$ for each $1\le k \le m'+1$ 
  satisfying the following two conditions: $\mcF_{i_k, \dots, i_m}b_{\gL} = \mcB_{w_k}(\gL)$ holds, 
  and $w_k$ has a reduced expression in the form
  \[ w_k =s_{i_k} \cdots s_{i_{m'}} s_{j_1} \cdots s_{j_\ell} \ \ \ \text{for some} \ \ell \in \Z_{\ge 0}, \ 
     j_1,\dots,j_\ell \in \hI.
  \]
  Since $w=w_1$ satisfies (\ref{eq:condition}), this completes the proof.
  The assertion for $k = m' +1$ follows from Lemma \ref{Lem: Demazure_crystal} (ii) 
  since the second condition is trivial in this case.
  Assume $k \le m'$. By the induction hypothesis, $w_{k+1}$ has a reduced expression in the form
  $ w_{k+1} = s_{i_{k+1}} \cdots s_{i_{m'}} s_{j_1} \cdots s_{j_{\ell}}.
  $
  If $s_{i_k}w_{k+1} > w_{k+1}$, $w_k = s_{i_k} w_{k+1}$ obviously satisfies the required conditions since
  we have
  \[ \mcF_{i_k,\dots,i_m}b_{\gL} =\mcF_{i_k} \mcB_{w_{k+1}} (\gL) = \mcB_{s_{i_k}w_{k+1}}(\gL)
  \]
  by Lemma \ref{Lem: Demazure_crystal} (i).
  Assume that $s_{i_k}w_{k+1} < w_{k+1}$. Then by the exchange condition, there exists some $k+1 \le p \le m'$ such that
  \[ s_{i_k} s_{i_{k+1}} \cdots s_{i_{p-1}} = s_{i_{k+1}} \cdots s_{i_p},
  \]
  or there exists some $1 \le q \le \ell$ such that
  \[ s_{i_k} s_{i_{k+1}} \cdots s_{i_{m'}} s_{j_1} \cdots s_{j_{q-1}} = s_{i_{k+1}} \cdots s_{i_{m'}} s_{j_1} \cdots s_{j_q}.
  \]
  However the first case cannot occur since $s_{i_1}\cdots s_{i_{m'}}$ is reduced,
  and hence the second case occurs. Then 
  \[ w_{k+1} = s_{i_k} s_{i_{k+1}} \cdots s_{i_{m'}} s_{j_1} \cdots s_{j_{q-1}} s_{j_{q+1}} \cdots s_{j_\ell}
  \]
  is a reduced expression of $w_{k+1}$.
  Since we have from Lemma \ref{Lem: Demazure_crystal} (i) that
  \[ \mcF_{i_k,\dots,i_m} b_{\gL} = \mcF_{i_k}\mcB_{w_{k+1}}(\gL) = \mcB_{w_{k+1}}(\gL),
  \]
  $w_k = w_{k+1}$ satisfies the required conditions, and the assertion is proved.
\end{proof}

\subsection{Demazure crystal decomposition of $b_{\gL} \otimes \mB_0(\gl)$}

Throughout the rest of this section, we assume that $\gL \in \hP_+ \setminus \Z \gd$, $\gl \in P_+$ 
and a sequence of real numbers $\ul{\gs}: 0= \gs_0 < \dots < \gs_N =1$ is sufficiently fine for $\gl$.
Denote by $\hI_{\gL}$ the proper subset of $\hI$ defined by
\[ \hI_{\gL} = \{ i \in \hI \mid \langle \gL, \ga_i^{\vee} \rangle = 0 \}.
\]

This subsection is devoted to show the following proposition.
The proof is carried out in the similar line as \cite[Proposition 12]{MR1887117}.

\begin{Prop}\label{Prop: Demazure_crystal_decomposition1}
  For each $\pi_0 \in \mB_0(\gl)^{\gL}$, there exists some $w_{\pi_0} \in \hW$ such that 
  the image of a subset $b_{\gL} \otimes \mB_0(\gl)$ of $\mcB(\gL) \otimes \mB_0(\gl)$ under the isomorphism 
  given in Corollary \ref{Cor: decomposition} coincides with the disjoint union of the Demazure crystals
  \[ \coprod_{\pi_0 \in \mB_0(\gl)^{\gL}} \mcB_{w_{\pi_0}}\big(\gL+\pi_0(1)\big).
  \]
  Moreover, each $w_{\pi_0}$ satisfies $\big\langle w_{\pi_0}\big(\gL + \pi_0(1)\big) , \ga_i^{\vee} \big\rangle \le 0$ 
  for all $i \in \hI_{\gL}$.
\end{Prop}

\begin{Lem} \label{Lem: Technical_Lemma2}
  Let $0 < u \le 1$ be a real number, $J$ a subset of $\hI$, and $\pi \in \mB_0(\gl)$ a path satisfying for all $i \in J$ that
  \begin{equation} \label{eq: assumption2}
    H_i^{\pi}(t) \ge - \langle \gL,\ga_i^{\vee}\rangle \ \ \ \text{for all} \ t \in [0,u] \ \ \ \text{and} \ \ \ H_i^{\pi}(u) 
    = - \langle \gL,\ga_i^{\vee}\rangle. 
  \end{equation}
  {\normalfont(i)} For $i \in J$ such that $\tf_i(\pi_{\gL} * \pi) \neq \0$, we have 
    \[ \tf_i (\pi_{\gL} * \pi) = \pi_{\gL} * \tf_i\pi \ \ \ \text{and} \ \ \ 
       \tf_i\pi(t) = \pi (t) \ \ \ \text{for all} \ t \in [0,u].
    \]
  {\normalfont(ii)} For a sequence $i_1, \dots, i_s$ of elements of $J$ such that 
    $\tf_{i_1}\cdots\tf_{i_s}(\pi_{\gL} * \pi) \neq \0$, we have 
    \[ \tf_{i_1} \cdots \tf_{i_k} (\pi_{\gL} * \pi) = \pi_{\gL} * (\tf_{i_1} \cdots \tf_{i_k}\pi) \ \text{and} \ 
       \tf_{i_1} \cdots \tf_{i_k} \pi(t) = \pi(t) \  \text{for all} \ t \in [0,u].
    \]
\end{Lem}

\begin{proof}
  It suffices to show (i) only since (ii) can be proved inductively from this.
  We have from (\ref{eq: map_on_crystal}) and (\ref{eq: assumption2}) that 
  \[ \varphi_i(\pi_{\gL}) = \langle\gL,\ga_i^{\vee}\rangle \ \ \ \text{and} \ \ \ 
     \gee_i(\pi) = -m_i^{\pi} \ge \langle\gL,\ga_i^{\vee}\rangle.
  \]
  Hence $\tf_i(\pi_{\gL} * \pi) = \pi_{\gL} * \tf_i \pi$ follows by (\ref{eq: tensor_rule2}).
  From Lemma \ref{Lem: root_operator} (ii) and (\ref{eq: assumption2}), $ \tf_i \pi(t) = \pi(t)$ follows for all $t \in [0,u]$,
  and (i) is proved.
\end{proof}

For each $\pi_0 \in \mB_0(\gl)^{\gL}$, we define a sequence $\mathbf{i}_{\pi_0}$ of elements of $\hI$ as follows.
For $0 \le p \le N-1$, set
\[ J^p= \{ i\in \hI \mid H_i^{\pi_0}(\gs_p) = -\langle \gL,\ga_i^{\vee}\rangle\}.
\]
Note that $J^p$ is a proper subset by Lemma \ref{Lem: properness}, and hence $\hW_{J^p}$ is finite.
When $J^p \neq \emptyset$, fix a sequence $\mathbf{i}^p : i^p_1,\dots, i^p_{m_p}$ of elements of $J^p$
so that $s_{i^p_1} \dots s_{i^p_{m_p}}$ is a reduced expression of the longest element of $\hW_{J^p}$.
When $J^p = \emptyset$ we set $\mathbf{i}^p = \emptyset$.
Then a sequence $\mathbf{i}_{\pi_0}$ is defined by $ \mathbf{i}^0,\mathbf{i}^1,\dots,\mathbf{i}^{N-1}$.
The following lemma is essential for the proof of Proposition \ref{Prop: Demazure_crystal_decomposition1}:

\begin{Lem} \label{Lem: essential_lemma}
  For $\pi_0 \in \mB_0(\gl)^{\gL}$, we have
  \[ C(\pi_{\gL} * \pi_0) \cap \big(\pi_{\gL} * \mB_0(\gl)\big) = \mcF_{\mathbf{i}_{\pi_0}}(\pi_{\gL} * \pi_0).
  \]
\end{Lem}

\begin{proof}
  For $0\le p \le N$, set $\mathbf{i}^{\ge p} = \mathbf{i}^p, \mathbf{i}^{p+1},\dots, \mathbf{i}^{N-1}$ and 
  \[ \mB_0(\gl)^{p} = \{ \pi \in \mB_0(\gl) \mid \pi(t) = \pi_0(t) \ \text{for all} \ t \in [0, \gs_p] \}.
  \]
  We shall show by the descending induction on $p$ that 
  \[ C(\pi_{\gL} * \pi_0) \cap \big(\pi_\gL * \mB_0(\gl)^p\big) = \mcF_{\mathbf{i}^{\ge p}}(\pi_\gL*\pi_0),
  \]
  which for $p = 0$ is the assertion of the lemma.
  If $p = N$, there is nothing to prove.
  Assume $p < N$. 
  Since 
  \[ \pi_\gL * \mB_0(\gl)^{p+1} \supseteq \mcF_{\mathbf{i}^{\ge p+1}} (\pi_\gL * \pi_0)
  \]
  follows by the induction hypothesis, in order to show the containment $\supseteq$ it suffices to show that
  \[ \pi_\gL * \mB_0(\gl)^p \supseteq \mcF_{\mathbf{i}^p} \big(\pi_\gL * \mB_0(\gl)^{p+1}\big).
  \]
  Let $\pi \in \mB_0(\gl)^{p+1}$.
  Since $H_i^{\pi}(t)=H_i^{\pi_0}(t)$ for all $t \in [0,\gs_{p+1}]$ and $i \in \hI$,
  we have for all $i \in J^p$ that
  \[  H_i^{\pi}(t) \ge - \langle \gL,\ga_i^{\vee}\rangle \ \ \ \text{for all} \ t \in [0,\gs_p] \ \ \ \text{and} \ \
      \ H_i^{\pi}(\gs_p) =-\langle\gL,\ga_i^{\vee}\rangle.
  \]
  Hence for any sequence $i_1, \dots, i_s$ of elements of $J^p$ satisfying
  $\tf_{i_1}\cdots\tf_{i_s}(\pi_{\gL}* \pi) \neq \0$, we have from Lemma \ref{Lem: Technical_Lemma2} (ii) that
  \[ \tf_{i_1} \cdots \tf_{i_s} (\pi_{\gL} * \pi) = \pi_{\gL} * (\tf_{i_1} \cdots \tf_{i_s}\pi)
  \]
  and
  \[ \tf_{i_1} \cdots \tf_{i_s} \pi(t) = \pi(t) = \pi_0(t) \ \ \ \text{for all} \ t \in [0,\gs_p],
  \]
  which implies $\pi_{\gL} * \mB_0(\gl)^p \supseteq \mcF_{\mathbf{i}^p}(\pi_{\gL} * \pi)$ as required.
  The containment $\supseteq$ is proved.
  
  Now we show the opposite containment $\subseteq$.
  Suppose that $\pi \in \mB_0(\gl)^p$ satisfies $\pi_{\gL} * \pi \in C(\pi_{\gL} * \pi_0)$,
  and define $p_0\in \{1,\ldots,N\}$ by an integer satisfying
  \[ \gs_{p_0} = \max\{u \in [0,1] \mid H_i^\pi(t) \ge -\langle \gL, \ga_i\rangle \ \text{for all} \ t \in [0,u],i\in \hI\}.
  \]
  Note that $p_0 \ge p$ since $\pi \in \mB_0(\gl)^{p}$.
  Since there exists a sequence $j_1, \ldots,j_\ell$ of elements of $\hI$ such that 
  \[ \te_{j_1} \cdots \te_{j_\ell} (\pi_\gL * \pi) = \pi_\gL * (\te_{j_1} \cdots \te_{j_\ell} \pi) = \pi_\gL * \pi_0,
  \]
  we have from Lemma \ref{Lem: Technical_Lemma} (ii) that $\pi_0(t) = \pi(t)$ for all $t \in [0,\gs_{p_0}]$.
  Hence $\pi \in \mB_0(\gl)^{p_0}$ follows.
  If $p_0 > p$, we have from the induction hypothesis that
  \[ \pi_\gL * \pi \in C(\pi_\gL * \pi_0) \cap \big( \pi_\gL * \mB_0(\gl)^{p_0}\big) \subseteq \mcF_{\mathbf{i}^{\ge p_0}}
     (\pi_\gL * \pi_0) \subseteq \mcF_{\mathbf{i}^{\ge p}}(\pi_\gL * \pi_0).
  \]
  Assume that $p_0 = p$.
  Then by Lemma \ref{Lem: used_again}, there exists $\pi' \in \mB_0(\gl)$ such that
  \begin{equation} \label{eq: path_containment}
    \pi_{\gL} * \pi' \in \{ \te_{i_1}\cdots\te_{i_s}(\pi_{\gL}*\pi)\mid s\ge 0, i_k \in J^p\}\setminus\{\0\}
  \end{equation}
  and $p_0' > p$, where $p_0'$ is the integer satisfying 
  \[ \gs_{p_0'} = \max\{u \in [0,1] \mid H_i^{\pi'}(t) \ge -\langle \gL, \ga_i\rangle \
     \text{for all} \ t \in [0,u],i\in \hI\}.
  \]
  Similar argument as above shows that $\pi' \in \mB_0(\gl)^{p'}$, and then 
  \[ \pi_\gL * \pi' \in \pi_\gL * \mB_0(\gl)^{p'} \subseteq \mcF_{\mathbf{i}^{\ge p'}}(\pi_\gL * \pi_0) \subseteq 
     \mcF_{\mathbf{i}^{\ge p}}(\pi_\gL * \pi_0)
  \] 
  follows by the induction hypothesis.
  By Proposition \ref{Prop: stability} and $C(\pi_{\gL} * \pi_0) \cong \mcB(\gL + \pi_0(1))$, we have for all $i \in J^p$ that 
  \[ \tf_i \mcF_{\mathbf{i}^{\ge p}} (\pi_{\gL} * \pi_0) \subseteq \mcF_{\mathbf{i}^{\ge p}}
     (\pi_{\gL} * \pi_0) \cup \{ \0 \}.
  \]
  Hence we have from (\ref{eq: path_containment}) that
  \[ \pi_{\gL} * \pi \in \{ \tf_{i_1} \cdots \tf_{i_s}(\pi_{\gL} * \pi') \mid s \ge 0, i_k \in J^{p} \}\setminus
     \{\0\} \subseteq \mcF_{\mathbf{i}^{\ge p}} (\pi_{\gL} * \pi_0),
  \]
  and the containment $\subseteq$ is proved.
\end{proof}

\noindent {\textit{Proof of Proposition \ref{Prop: Demazure_crystal_decomposition1}.\ }}
  For $\pi_0 \in \mB_0(\gl)^{\gL}$, 
  Lemma \ref{Lem: essential_lemma} implies that the image of $C(b_{\gL} \otimes \pi_0) \cap \big(b_{\gL} \otimes \mB_0(\gl)\big)$
  under the isomorphism given in Corollary \ref{Cor: decomposition} is $\mcF_{\mathbf{i}_{\pi_0}}b_{\gL + \pi_0(1)}$.
  By Proposition \ref{Prop: stability}, there exists an element $w_{\pi_0} \in \hW$
  satisfying
  \[ \mcF_{\mathbf{i}_{\pi_0}}b_{\gL + \pi_0(1)} = \mcB_{w_{\pi_0}}(\gL + \pi_0(1)), \ \ \ \text{and} 
  \] 
  \[ \big\langle w_{\pi_0}\big(\gL + \pi_0(1)\big), \ga_i^{\vee} \big\rangle \le 0 \ \ \ \text{for all} \ i \in J^0 = \hI_{\gL}.
  \]
  Now the proposition follows since we have
  \[ b_{\gL} \otimes \mB_0(\gl) = \coprod_{\pi_0 \in \mB_0(\gl)^{\gL}} C(b_{\gL} \otimes \pi_0) 
     \cap (b_{\gL} \otimes \mB_0(\gl)).
  \]
  \qed

\subsection{Demazure crystal decomposition of $b_{\gL} \otimes \mB(\gl)_{\cl}$}

First we make an elementary remark on crystals:

\begin{Rem} \normalfont
  Let $\mcB$ be a $U_q(\hfg)$-crystal with a weight map $\wt\colon \mcB \to \hP$. 
  Then we can regard $\mcB$ naturally as a $U_q'(\hfg)$-crystal
  by replacing the weight map with $\cl \circ \wt\colon \mcB \to \hP_{\cl}$.
\end{Rem}

Similarly as $\mB_0(\gl)^{\gL}$, we define $\mB(\gl)_{\cl}^{\gL}$ by
\[ \mB(\gl)_{\cl}^{\gL} = \{ \eta \in \mB(\gl)_{\cl} \mid \langle \eta(t), 
   \ga_i^{\vee} \rangle \ge -\langle \gL,\ga_i^{\vee}\rangle \ \text{for all} \ t \in [0,1],
   i \in \hI \}.
\]
It is easily checked that
\begin{equation} \label{eq: fix_highest}
  \mB_0(\gl)^{\gL} = \coprod_{\eta_0 \in \mB(\gl)_{\cl}^{\gL}} \cl^{-1}(\eta_0) \cap \mB_0(\gl).
\end{equation}
Similarly as a $U_q(\hfg)$-crystal, for a $U_q'(\hfg)$-crystal $\mcB$ and an element $b \in \mcB$
we denote by $C(b)$ the connected component of $\mcB$ containing $b$.

\begin{Lem}\label{Lem: isom_of_cl}
  Let $\eta_0 \in \mB(\gl)_{\cl}^{\gL}$. Then for any $\pi_0 \in \cl^{-1}(\eta_0) \cap \mB_0(\gl) \subseteq \mB_0(\gl)^{\gL}$,
  the map $\id \otimes \cl\colon \mcB(\gL) \otimes \mB_0(\gl) \to \mcB(\gL) \otimes \mB(\gl)_{\cl}$ induces an 
  isomorphism of $U_q'(\hfg)$-crystals from $C(b_{\gL} \otimes \pi_0)$ to $C(b_{\gL} \otimes \eta_0)$.
\end{Lem}

\begin{proof}
  By the definition of $\cl\colon \mP \to \mP_{\cl}$ in Subsection \ref{subsection: crystal and bases}, 
  we can see that $\id \otimes \cl$ preserves $\hP_{\cl}$-weights, $\gee_i$, $\gph_i$, and commutes with root operators.
  Hence it suffices to show that the induced map is bijective. 
  The surjectivity is obvious.
  Let us show the injectivity.
  Let $b \otimes \eta \in C(b_\gL \otimes \eta_0)$ with $b \in \mcB(\gL)$ and $\eta \in \mB(\gl)_{\cl}$,
  and take $\pi_1, \pi_2 \in \mB_0(\gl)$ such that
  \[ b \otimes \pi_j \in C(b_{\gL} \otimes \pi_0) \ \ \ \text{and} \ \ \ \cl(\pi_j) = \eta
  \]
  for $j=1,2$.
  By Lemma \ref{Lem: lem_for_degree_fucnction} (i), there exists $k \in \Z$ such that $\pi_2 = \pi_1 + \pi_{kd_{\gl}\gd}$.
  By $C(b_{\gL} \otimes \pi_0) \cong\mcB\big(\gL + \pi_0(1)\big)$, there exists a sequence 
  $i_1,\dots,i_s$ of elements of $\hI$ such that 
  $\te_{i_1} \cdots\te_{i_s} (b \otimes \pi_1) = b_{\gL} \otimes \pi_0$, and then it is easily seen 
  (cf.\ \cite[Lemma 3.3.1]{MR2199630}) that 
  \[ \te_{i_1} \cdots \te_{i_s}(b \otimes \pi_2) =\te_{i_1} \cdots \te_{i_s}\big(b \otimes (\pi_1 + \pi_{kd_{\gl}\gd})\big) 
     =b_{\gL} \otimes (\pi_0 + \pi_{kd_{\gl}\gd}).
  \]
  Hence $b_{\gL} \otimes (\pi_0 + \pi_{kd_{\gl}\gd}) \in C(b_{\gL} \otimes \pi_0)$, 
  which together with (\ref{eq: uniqueness_of_highest}) implies
  $k= 0$. Therefore we have $\pi_1 = \pi_2$, and the injectivity follows.  
\end{proof}

Recall that we have
\[ \mcB(\gL) \otimes \mB_0(\gl) = \bigoplus_{\pi_0 \in \mB_0(\gl)^{\gL}}C(b_{\gL} \otimes \pi_0)
\]
by Proposition \ref{Prop: equality}. 
Applying $\id \otimes \cl$ to this, we have from (\ref{eq: fix_highest}) that
\begin{equation} \label{eq: decomposition}
   \mcB(\gL) \otimes \mB(\gl)_{\cl} = \bigoplus_{\eta_0 \in \mB(\gl)_{\cl}^{\gL}} C(b_{\gL} \otimes \eta_0).
\end{equation}
For each $\eta_0 \in \mB(\gl)_{\cl}^{\gL}$, fix an arbitrary $\pi^{\eta_0} \in \cl^{-1}(\eta_0) \cap \mB_0(\gl)$.
Then the following proposition is obtained:

\begin{Prop}\label{Prop: decomposition2}
  {\normalfont(i)} We have 
      \[ \mcB(\gL) \otimes \mB(\gl)_{\cl} \stackrel{\sim}{\to} \bigoplus_{\eta_0 \in \mB(\gl)_{\cl}^{\gL}} 
         \mcB\big(\gL + \pi^{\eta_0}(1)\big)
      \]
      as $U_q'(\hfg)$-crystals, where the given isomorphism maps each 
      $b_{\gL} \otimes \eta_0 \in b_{\gL} \otimes \mB(\gl)_{\cl}^{\gL}$ to
      $b_{\gL + \pi^{\eta_0}(1)} \in \mcB\big(\gL + \pi^{\eta_0}(1)\big)$.\\
  {\normalfont(ii)}   Under the isomorphism given in {\normalfont(i)}, 
      the image of the subset $b_{\gL} \otimes \mB(\gl)_{\cl}$ coincides with 
      the disjoint union of Demazure crystals
      \[ \coprod_{\eta_0 \in \mB(\gl)_{\cl}^{\gL}} \mcB_{w_{\eta_0}}\big(\gL + \pi^{\eta_0}(1)\big) \ 
         \text{for some} \ w_{\eta_0} \in \hW. 
      \]
      Moreover each $w_{\eta_0}$ satisfies $\big\langle w_{\eta_0}\big(\gL + \pi^{\eta_0}(1)\big) , 
      \ga_i^{\vee} \big\rangle \le 0$ for all $i \in \hI_{\gL}$.
\end{Prop}

\begin{proof}
  The assertion (i) follows from (\ref{eq: decomposition}) 
  since for each $\eta_0 \in \mB(\gl)_{\cl}^{\gL}$, we have from Lemma \ref{Lem: isom_of_cl} 
  and Corollary \ref{Cor: decomposition} that
  \[  C(b_{\gL} \otimes \eta_0) \cong C(b_{\gL} \otimes \pi^{\eta_0}) \cong \mcB\big(\gL + \pi^{\eta_0}(1)\big)
  \]
  as $U_q'(\hfg)$-crystals.
  The assertion (ii) also follows from these isomorphisms and 
  Proposition \ref{Prop: Demazure_crystal_decomposition1}.
\end{proof}

\section{Study on the decomposition of $b_{\gL_0} \otimes \mB(\gl)_{\cl}$}

\subsection{Preliminaries about the weight sum of $\mB(\gl)_{\cl}$}

In the previous section, we have seen that $b_{\gL} \otimes \mB(\gl)_{\cl}$ 
coincides with the disjoint union of some Demazure crystals.
In this section, we study in more detail this result with $\gL = \gL_0$.

First we prepare some notation.
Let $\gL \in \hP_+$ and $w \in \hW$ be elements satisfying $w \gL = w_0\gl + \ell \gL_0 + m \gd$
for some $\gl \in P_+, \ell \in \Z_{> 0}, m \in \Z$.
Then we use the following notation which is compatible with that of modules:
\[ \mathcal{B}(\ell, \gl)[m] = \mcB_w(\gL).
\]
Note that we have from Proposition \ref{Prop: Demazure crystal2} (ii) that
\begin{equation*}\label{eq: character}
  \ch{\hfh} \D(\ell, \gl)[m] = \sum_{b \in \mcB(\ell, \gl)[m]} e\big(\wt(b)\big).
\end{equation*}

Let $\gl \in P_+$ and $\eta_0 \in \mB(\gl)_{\cl}^{\gL_0}$.
Then $\pi_{\eta_0} \in \cl^{-1}(\eta_0) \cap \mB_0(\gl)$ follows 
($\pi_{\eta_0}$ is defined in Subsection \ref{Degree function}).
Hence from Proposition \ref{Prop: decomposition2}, we obtain an isomorphism
\[ \gk\colon \mcB(\gL_0) \otimes \mB(\gl)_{\cl} \stackrel{\sim}{\to} 
   \bigoplus_{\eta_0 \in \mB(\gl)_{\cl}^{\gL_0}} \mcB\big(\gL_0 + \pi_{\eta_0}(1)\big)
\]
of $U_q'(\hfg)$-crystals which satisfies
$\gk(b_{\gL_0} \otimes \eta_0) = b_{\gL_0 + \pi_{\eta_0}(1)}$ for all $\eta_0 \in \mB(\gl)_{\cl}^{\gL_0}$.

\begin{Lem} 
  For each $\eta \in \mB(\gl)_{\cl}$, we have 
  \[ \wt\circ \gk(b_{\gL_0} \otimes \eta) = \wt(\pi_{\eta}) + \gL_0 = \wt_{\hP}(\eta) + \gL_0.
  \]
\end{Lem}

\begin{proof} 
  The second equality follows from the definition of the $\hP$-weight of $\eta$.
  Let us show the first one.
  By (\ref{eq: decomposition}), there exists some $\eta_0 \in \mB(\gl)_{\cl}^{\gL_0}$ 
  such that $b_{\gL_0} \otimes\eta \in C(b_{\gL_0} \otimes \eta_0)$. 
  Recall that $\gk$ is defined by the composition
  \[ C(b_{\gL_0} \otimes \eta_0) \stackrel{\sim}{\to} C(b_{\gL_0} \otimes \pi_{\eta_0}) \stackrel{\sim}{\to} 
     \mcB\big(\gL_0 + \pi_{\eta_0}(1)\big),
  \]
  and the second isomorphism is of  $U_q(\hfg)$-crystals. 
  Hence it suffices to show that the first one, which we denote by $\gk'$ here, satisfies
  $\gk'(b_{\gL_0} \otimes \eta) = b_{\gL_0} \otimes \pi_{\eta}$.
  Let $i_1,\dots,i_k$ be a sequence of elements of $\hI$ such that 
  \[ \tf_{i_1}\cdots\tf_{i_k}(b_{\gL_0} \otimes \eta_0) = b_{\gL_0} \otimes (\tf_{i_1}\cdots\tf_{i_k} \eta_0)
     = b_{\gL_0} \otimes \eta.
  \]
  We show $\gk'(b_{\gL_0} \otimes \eta) = b_{\gL_0} \otimes \pi_{\eta}$ by induction on $k$. 
  If $k = 0$, there is nothing to prove.
  Assume $k>0$. By the induction hypothesis, $\eta' =\tf_{i_2}\cdots\tf_{i_k}\eta_0$ satisfies 
  $\gk'(b_{\gL_0} \otimes \eta') = b_{\gL_0} \otimes \pi_{\eta'}$.
  Note that $\gee_{i_1}(\eta') \ge \gd_{i_10}$ follows since
  $\tf_{i_1} (b_{\gL_0} \otimes \eta') = b_{\gL_0} \otimes \tf_{i_1}\eta'$.
  Hence we have from Lemma \ref{Lem: lemma_of_pi} that
  \[ \gk'(b_{\gL_0} \otimes \eta) = \gk'\big(\tf_{i_1}(b_{\gL_0}\otimes \eta')\big) = \tf_{i_1}(b_{\gL_0} \otimes \pi_{\eta'})
     = b_{\gL_0} \otimes \pi_{\eta}
  \]
  as required.
\end{proof}

By Proposition \ref{Prop: decomposition2} (ii), $\gk\big(b_{\gL_0} \otimes \mB(\gl)_{\cl}\big)$ is 
the disjoint union of some Demazure crystals in the form
$\mcB_w(\gL')$ with  $\gL' \in \hP_+$ of level $1$ and $w \in \hW$ satisfying 
$\langle w\gL', \ga_i^{\vee} \rangle \le 0$ for all $i \in I$, 
which can be written as $\mcB(1, \mu)[n]$ for some $\mu \in P_+$ and $n \in \Z$.
In conclusion, we have the following:

\begin{Prop}\label{Prop: disjoint}
  Let $\gl \in P_+$. 
  Then there exist sequences $\mu_1,\dots, \mu_\ell \in P_+$ and $n_1,\dots,n_\ell \in \Z$ such that
  \[ \gk\big(b_{\gL_0} \otimes \mB(\gl)_{\cl}\big) = \coprod_{1 \le j \le \ell} \mcB(1, \mu_j)[n_j].
  \]
  Moreover we have $\wt \circ \gk(b_{\gL_0} \otimes \eta) = \wt_{\hP}(\eta) + \gL_0$ for all $\eta \in \mB(\gl)_\cl$.
\end{Prop}

\begin{Cor}\label{Cor: character_forumula}
  \begin{equation*}
    \sum_{\eta \in \mB(\gl)_{\cl}} e\big(\wt_{\hP} (\eta)\big) = \sum_{1 \le j \le \ell} \ch{\fh_d}\D(1,\mu_j)[n_j].
  \end{equation*}
\end{Cor}

If $\fg$ is simply laced, $\mu_j$'s and $n_j$'s in Corollary \ref{Cor: character_forumula} are easily determined 
from a result in \cite{MR2323538}:

\begin{Prop}\label{Prop: simply_laced}
  If $\fg$ is simply laced, then we have
  \begin{equation*}
    \sum_{\eta\in \mB(\gl)_{\cl}} e\big(\wt_{\hP}(\eta)\big) =\ch{\fh_d}\D(1,\gl)[0].
  \end{equation*}
\end{Prop} 

\begin{proof}
  From \cite[Proposition 3]{MR2323538}, we have $\gk\big(b_{\gL_0} \otimes \mB(\gl)_{\cl}\big) = \mcB(1,\gl)[n]$ 
  for some $n \in \Z$.
  Since $\pi_{\eta_{\cl(\gl)}} = \pi_{\gl}$ follows by definition, we have 
  \[ \wt \circ \gk\big(b_{\gL_0} \otimes \eta_{\cl(\gl)}\big) = \gl + \gL_0.
  \]
  Hence $\mcB(1, \gl)[n]$ contains an element of weight $\gl+\gL_0$, which forces $n = 0$. 
\end{proof}

The decomposition being more complicated when $\fg$ is non-simply laced, 
it is hard to determine $\mu_j$'s and $n_j$'s by straightforward calculations for general $\gl$.
In the fundamental weight case, however, the following proposition is obtained using a result 
in \cite{MR2115972} and Theorem \ref{Thm: isom_of_finite_crystals}:

\begin{Prop}\label{Prop: fundamental}
  For general $\fg$ and each $i \in I$, we have
  \[ \sum_{\eta\in \mB(\varpi_i)_{\cl}} e\big(\wt_{\hP}(\eta)\big) =\ch{\fh_d}\D(1,\varpi_i)[0].
  \]
\end{Prop}

\begin{proof}
  Since $\mB(\varpi_i)_{\cl}$ is isomorphic to $\mcB\big(W_q(\varpi_i)\big)$ by Theorem \ref{Thm: isom_of_finite_crystals} (ii),
  $\gk\big(b_{\gL_0} \otimes \mB(\varpi_i)_{\cl}\big) = \mcB(1,\varpi_i)[n]$ follows 
  for some $n \in \Z$ by \cite[Corollary 4.8]{MR2115972}.
  The rest of the proof is the same as that of Proposition \ref{Prop: simply_laced}.
\end{proof}

\begin{Cor}\label{Cor: h-weight}
  If $\gl = \sum_{i \in I} \gl_i \varpi_i \in P_+$, then we have
  \[ \sum_{\eta \in \mB(\gl)_{\cl}} e\big(p \circ \wt(\eta)\big) = \prod_{i \in I}\ch{\fh}\D(1,\varpi_i)^{\gl_i},
  \]
  where $p$ denotes the canonical projection $\hP_{\cl} \to P$.
\end{Cor}

\begin{proof}
  From the above proposition, we have 
  \[ \sum_{\eta \in \mB(\varpi_i)_{\cl}} e\big(p \circ \wt(\eta)\big) = \ch{\fh}\D(1,\varpi_i)
  \]
  for each $i \in I$.
  Then the assertion follows since $\mB(\gl)_{\cl} \cong \bigotimes_{i \in I} \mB(\varpi_i)_{\cl}^{\otimes \gl_i}$ 
  by Theorem \ref{Thm: isom_of_finite_crystals} (i).
\end{proof}

From the next subsection, we begin to determine $\mu_j$'s and $n_j$'s 
in the non-simply laced case using Demazure crystals for $U_q(\hfg^{\sh})$.

\subsection{Path models for $U_q(\hfg^{\sh})$}

In the rest of this section we assume that $\fg$ is non-simply laced, 
and apply the theory of path models for $U_q(\hfg^{\sh})$.
Here we fix some notation used throughout the rest of this section.

Let $\gt^{\sh}$ be the highest root in $\lpishr$ and $\ga_0^{\sh} = \gd - \gt^{\sh} \in  \hat{\gD}$,
which corresponds to a simple root of $\hfg^{\sh}$.
Note that $(\ga_0^{\sh})^{\vee} = rK -(\gt^{\sh})^{\vee}$.
Let $s_0^{\sh} \in \hW$ denote the reflection associated with $\ga_0^{\sh}$, 
and $\hW^{\sh}$ the subgroup of $\hW$ generated by $\{ s_0^{\sh}\} \cup \{ s_i \mid i\in I^{\sh} \}$.
Set $\hat{I}^{\sh} =\{ 0\} \cup I^{\sh}$, and 
\[ \hat{P}^{\sh} 
   = \sum_{i \in I^{\sh}} \Z \overline{\varpi_i} +r^{-1}\Z \overline{\gL}_0 + \Z \overline{\gd} \subseteq \big(\hfh^{\sh}\big)^*.
\]
Note that $s_0^{\sh}$ acts on $\hP^{\sh}$ by $s_0^{\sh}(\nu) = \nu - \langle \nu, (\ga_0^{\sh})^{\vee} \rangle \ol{\ga}_0^{\sh}$
for $\nu \in \hP^{\sh}$, and $s_i$ for $i \in I^{\sh}$ also acts similarly. 

Let $\mP^{\sh}$ be the set of paths with weights in $\hP^{\sh}$, and define $\mPint^{\sh}$ similarly as $\mPint$.
As described in Subsection \ref{Def_of_path}, root operators associated with $i \in \hI^{\sh}$
are defined on $\mPint^{\sh}$ using the above actions of simple reflections.
To distinguish them from $\te_i$ and $\tf_i$, we denote them by $\te_i^{\sh}$ and $\tf_i^{\sh}$ ($i \in \hI^{\sh}$).
The maps $\wt\colon \mPint^{\sh} \to \hP^{\sh}$ and $\gee_i$, $\gph_i\colon \mPint^{\sh} \to \Z_{\ge 0}$ 
for $i \in \hI^{\sh}$ are defined similarly.
Then Theorem \ref{Thm: crystal_structure} implies the following:

\begin{Prop}
  Let $\mB \subseteq \mPint^{\sh}$ be a subset such that $\te_i^{\sh} \mB \subseteq \mB \cup \{ \0 \}$
  and $\tf_i^{\sh} \mB \subseteq \mB \cup \{ \0 \}$ for all $i \in \hI^{\sh}$.
  Then $\mB$, together with the root operators $\te_i^{\sh}, \tf_i^{\sh}$ for $i \in \hI^{\sh}$ and the maps 
  $\mathrm{wt}, \gee_i, \gph_i$ for $i \in \hI^{\sh}$, becomes a $U_q(\hfg^{\sh})$-crystal.
\end{Prop}

Denote by $H_i^{\sh, \pi}: [0,1] \to \R$ and $m_i^{\sh, \pi} \in \R$ for $\pi \in \mP^{\sh}$ and $i \in \hI^{\sh}$ 
the counterparts of $H_i^{\pi'}$ and $m_i^{\pi'}$ respectively.
For $\nu \in \hP^{\sh}$, let $\pi_\nu \in \mPint^{\sh}$ denote the straight line path: $\pi_\nu(t) = t\nu$,
and $\mB^{\sh}_0(\nu) \subseteq \mPint^{\sh}$ the connected component containing $\pi_\nu$, which is a $U_q(\hfg^{\sh})$-crystal.
Since $\fg^{\sh}$ is simply laced, Proposition \ref{Prop: simply_laced} and Remark \ref{Rem: remark_of_character}
imply the following lemma:

\begin{Lem}\label{Lem: simply-laced}
  Let $\nu \in \ol{P}_+$.
  Then we have
  \[ \sum_{\begin{smallmatrix} \pi \in \mB_0^{\sh}(\nu)\\ \iota(\pi) \in \ol{P} \end{smallmatrix}} 
     e\big(\wt(\pi)\big) = \ch{\fh^{\sh}_d} \D^{\sh}(1, \nu)[0].
  \] 
\end{Lem}

\subsection{Identity on the weight sum of $\mB(\gl)_{\cl}$}

This subsection is devoted to the proof of the following proposition:

\begin{Prop}\label{Prop: critical_prop}
  Let $\gl \in P_+$, and set $\gl' = \gl -i_{\sh}(\ol{\gl})$. 
  Then we have
  \[ \sum_{\begin{smallmatrix} \eta \in \mB(\gl)_{\cl} \\ 
     \wt_{\hP}(\eta) \in \gl -Q_+^{\sh}+\Z \gd\end{smallmatrix}} e\big(\wt_{\hP}(\eta)\big)
     = e(\gl')i_{\sh}\ch{\fh_d^{\sh}}\D^{\sh}(1,\ol{\gl})[0].
  \]
\end{Prop}

\begin{Rem}\normalfont
  In the final part of this article, we shall prove that
  \[ \ch{\fh_d} W(\gl) = \sum_{\eta \in \mB(\gl)_{\cl}} e\big(\wt_{\hP}(\eta)\big),
  \]
  which is compatible with the above proposition, Lemma \ref{Lem: subisom} and \ref{Lem: character}.
\end{Rem}

First we prepare a technical lemma:

\begin{Lem}\label{Lem: existance}
  There exist $\tau \in \hW$ and $j \in I^{\sh}$ satisfying the following two conditions: \\
  {\normalfont(i)} $\tau(\ga_j) = \ga_0^{\sh}$, \\
  {\normalfont(ii)} $\tau$ has an expression $\tau = s_{i_1} \dots s_{i_M}$ satisfying for all $1 \le L \le M$ that
       \begin{equation} \label{eq: conditionii}
          s_{i_1}\dots s_{i_{L-1}}(\ga_{i_{L}}) \notin \{ \ga + k \gd \mid \ga \in \lpishr, k \in \Z \}.
       \end{equation}
\end{Lem}

\begin{proof}
  For each type of $\fg$, we give $\tau \in \hW$ with its expression and $j \in I^{\sh}$  as follows, 
  where we use the numbering of elements of $\hI$ in \cite[\S 4]{MR1104219}: \\[0.1cm]
  $\circ$ Type $B_\ell$:
  Let $\tau = s_{\ell-1} s_{\ell-2} \cdots s_2 s_0 s_1 s_2 \cdots s_{\ell-1}$ and $j = \ell$. \\
  \ \ \  \ \ \  \ \ \ \ \ \ \ \ \ \ 
  In this case, $ \ga_0^{\sh}=\ga_0+ \ga_1 + 2\ga_2 + 2\ga_3 + \dots + 2\ga_{\ell-1} + \ga_\ell$.\\[0.1cm]
  $\circ$ Type $C_\ell$:
  Let $\tau = s_\ell s_{\ell-1} \cdots s_1 s_0$ and $j = 1$. \\
  \ \ \ \ \ \ \ \ \ \ \ \ \ \ \ \ In this case, $ \ga_0^{\sh} = \ga_0 + \ga_1 + \dots +\ga_{\ell}$. \\[0.1cm]
  $\circ$ Type $F_4$: 
  Let $\tau = s_2 s_3 s_1 s_2 s_3 s_4 s_0 s_1 s_2$ and $j =3$. \\
  \ \ \ \ \ \ \ \ \ \ \ \ \ \ \ \ In this case, $ \ga_0^{\sh} =\ga_0+ 2\ga_1 + 3 \ga_2 + 3\ga_3 + \ga_4$. \\[0.1cm]
  $\circ$ Type $G_2$:
  Let $\tau =s_1 s_2 s_0 s_1$ and $j = 2$.\\
  \ \ \ \ \ \ \ \ \ \ \ \ \ \ \ \ In this case, $ \ga_0^{\sh}= \ga_0 + 2\ga_1 + 2\ga_2$. \\[0.1cm]
  Though it is a little troublesome work, 
  we can check directly that these elements actually satisfy the conditions (for informations of root
  systems, see \cite[Ch.\ VI.\ \S 4]{MR1890629} for example).
  Note that if $\ga_{i_L}$ is a long root, 
  the condition (\ref{eq: conditionii}) is trivial since the right hand side consists of short roots.
  Using this fact, we can reduce a bit the amount of calculations.  
\end{proof}

Now let us begin the proof of the proposition.
Denote by $\pi_{\gl'}$ the straight line path: $\pi_{\gl'}(t) = t\gl'$.
For $\pi \in \mB_0^{\sh}(\overline{\gl})$, define maps $i_{\sh}(\pi)\colon [0,1] \to \R \otimes_{\Z} \hP$ by 
$i_{\sh}(\pi)(t)  = i_{\sh}\big(\pi(t)\big)$, and $\varphi(\pi): [0,1] \to \R \otimes_\Z \hP$ by 
\[ \varphi(\pi) = i_{\sh}(\pi) + \pi_{\gl'}.
\]
The following lemma is essential for the proof of the above proposition.

\begin{Lem}\label{Lem: critical_lemma1}
  We have $\varphi(\pi) \in \mB_0(\gl)$ for all $\pi \in \mB_0^{\sh}(\overline{\gl})$.
\end{Lem}

It is easily seen that $\varphi(\pi_{\overline{\gl}}) = \pi_{\gl} \in \mB_0(\gl)$.
Hence by the definition of $\mB_0^{\sh}(\overline{\gl})$, in order to prove this lemma it suffices to show the following:

\begin{Lem}\label{Lem: critical_lemma2}
  Assume that $\pi \in \mB_0^{\sh}(\overline{\gl})$ satisfies $\varphi(\pi) \in \mB_0(\gl)$.
  Then for each $i \in \hI^{\sh}$, the following statements hold.\\
  {\normalfont(i)} If $\te_i^{\sh} \pi \neq \0$, then $\varphi(\te_i^{\sh} \pi) \in \mB_0(\gl)$. \\ 
  {\normalfont(ii)} If $\tf_i^{\sh} \pi \neq \0$, then $\varphi(\tf_i^{\sh} \pi) \in \mB_0(\gl)$.
\end{Lem}

\begin{proof} 
  Since the proof of (ii) is similar, we shall only prove (i).\\[-0.2cm] \\
  Claim 1.\ \ \  For $w \in \hW^{\sh}$ and $\nu \in \hP^{\sh}$, we have $i_{\sh}(w\nu) = wi_{\sh}(\nu)$.\\[-0.2cm]
    
    It suffices to show the claim for $w = s_i$ ($i \in I^\sh$) and $w = s_0^{\sh}$. 
    Let $i \in I^{\sh}$.
    Since $\ga_i^{\vee} \in \hfh^{\sh}$, we have 
    $\langle \nu, \ga_i^{\vee} \rangle = \langle i_{\sh}(\nu), \ga_i^{\vee} \rangle$,
    which together with $i_{\sh}(\overline{\ga}_i) = \ga_i$ implies $i_{\sh}(s_i\nu) = s_i i_{\sh}(\nu)$.
    The claim for $w = s_0^{\sh}$ is proved similarly.\\
    \\  
  Claim 2.\ \ \ We have $i_{\sh}(\te_i^{\sh}\pi) = \te_ii_{\sh}(\pi)$ for $i \in I^{\sh}$ 
    (though $i_{\sh}(\pi)$ may not belong to $\mPint$, $\te_ii_{\sh}(\pi)$ is defined in the same way as in 
    Subsection \ref{Def_of_path}). \\[-0.2cm]
  
    It follows that
    \[ H_i^{i_{\sh}(\pi)}(t) = \big\langle i_{\sh}\big(\pi(t)\big), \ga_i^{\vee} \big\rangle
       = \langle \pi(t), \ga_i^{\vee} \rangle =H_i^{\sh,\pi}(t)\ \ \ \text{for all} \ t \in [0,1] .
    \]
    Then the claim follows from the definition of the root operators and Claim 1.\\
    \\
  Claim 3.\ \ \ If $i \in I^{\sh}$, then (i) follows. \\[-0.2cm]
  
    Since $\langle \gl', \ga_i^{\vee} \rangle = 0$ follows, we have $H_i^{i_{\sh}(\pi)}(t) = H_i^{i_{\sh}(\pi) + \pi_{\gl'}}$.
    From this and the definition of $\te_i$, we can check that 
    \[ \te_i i_{\sh}(\pi) + \pi_{\gl'} = \te_i \big(i_{\sh}(\pi) + \pi_{\gl'}\big) = \te_i\varphi(\pi).
    \]
    Now the claim follows since we have from Claim 2 that
    \[ \varphi(\te_i^{\sh}\pi) = i_{\sh}(\te_i^{\sh}\pi) + \pi_{\gl'} 
       = \te_i i_{\sh}(\pi) + \pi_{\gl'} =  \te_i \varphi(\pi) \in \mB_0(\gl).
    \]
  
  It remains to show the case $i = 0$.
  Let $(\nu_1,\dots,\nu_N; \ul{\gs})$ be an expression of $\pi$, where $\nu_i \in \hW^{\sh} \ol{\gl}$.
  For each $1 \le k \le N$, there exist $w_k \in W^{\sh}$ and $p_k \in \Z$ 
  such that $\nu_k = w_k \overline{\gl} + p_k \overline{\gd}$ by (\ref{eq: id_of_Weyl_gp}).\\
  \\
  Claim 4.\ \ \ We have $\varphi(\pi) = (w_1 \gl + p_1\gd, \dots, w_N \gl + p_N\gd;\ul{\gs})$.\\[-0.2cm]
  
  For each $1\le k \le N$, we have from Claim 1 that 
  \[ i_{\sh}(w_k\ol{\gl} + p_k\ol{\gd}) +\gl' = w_ki_{\sh}(\ol{\gl}) + p_k \gd + \gl' 
     = w_k \big(i_{\sh}(\ol{\gl})+\gl'\big) + p_k \gd =w_k\gl + p_k \gd,
  \]
  which implies the claim.\\
  
  We set $\mu_k = w_k \gl + p_k \gd$ for $1 \le k \le N$.
  Put 
  \begin{align*}
    t_1 &= \min\{t\in[0,1]\mid H_0^{\sh, \pi}(t) = m_0^{\sh, \pi}\}, \\
    t_0 &= \max\{t\in[0,t_1]\mid H_0^{\sh,\pi}(t) = m_0^{\sh,\pi} + 1\}.
  \end{align*}
  By replacing $\ul{\gs}$ if necessary, we assume that $t_0, t_1 \in \{\gs_0, \dots, \gs_N\}$.
  Let $q_0,q_1 \in \{0, \dots, N \}$ be the integers such that $\gs_{q_0} = t_0$ and $\gs_{q_1} = t_1$.
  Then by the definition of $\te_0^{\sh}$ and the expression, we have
  \[ \te_0^{\sh} \pi = (\nu_1, \dots, \nu_{q_0}, s_0^{\sh}\nu_{q_0+1}, \dots, s_0^{\sh}\nu_{q_1}, 
     \nu_{q_1+1}, \dots,\nu_N; \ul{\gs}).
  \]
  Since $i_{\sh}(s_0^{\sh}\nu_k)+\gl' = s_0^{\sh}\big(i_{\sh}(\nu_k) + \gl'\big) = s_0^{\sh}\mu_k$, we also have
  \begin{equation} \label{eq:j}
     \varphi(\te_0^{\sh} \pi) = (\mu_1, \dots, \mu_{q_0}, s_0^{\sh}\mu_{q_0+1}, \dots, s_0^{\sh}\mu_{q_1}, \mu_{q_1+1}, 
     \dots, \mu_N; \ul{\gs}).
  \end{equation}
  Assume $\tau \in \hW$, its expression $\tau = s_{i_1} \cdots s_{i_M}$ and $j \in I^{\sh}$ satisfy 
  the conditions in Lemma \ref{Lem: existance}.
  In the sequel, we denote by $w \pi'$ for $w \in \hW$ and $\pi' \in \mP$ the path defined by $w \pi'(t) 
  =w \big(\pi'(t)\big)$. \\
  \\
  Claim 5.\ \ \ We have 
  \[ S_{\tau^{-1}}\varphi(\pi)=\tau^{-1} \varphi(\pi) = ( \tau^{-1}\mu_1, \dots, \tau^{-1}\mu_N; \ul{\gs})
  \]
  (the action $S_{\tau^{-1}}$ is defined in Theorem \ref{Thm: action_of_Weyl}). \\[-0.2cm]

    Set $\tau_L = s_{i_1} \cdots s_{i_L}$ for $1 \le L \le M$ and $\tau_0 = \id$.
    We shall show by induction on $L$ that 
    \[ S_{\tau_L^{-1}} \varphi(\pi) =\tau^{-1}_L \varphi(\pi).
    \]
    The case $L = 0$ is trivial.
    Assume $L \ge 1$ and $S_{\tau_{L-1}^{-1}} \varphi(\pi) = \tau_{L-1}^{-1} \varphi(\pi)$.
    By the condition (ii), $\tau_{L-1}\ga_{i_L} = \gb + \ell \gd$ follows for some 
    $\gb \in \gD \setminus \lpishr$ and $\ell \in \Z$,
    and we have for each $1 \le k \le N$ that 
    \begin{equation}\label{eq:adding}
      \langle \tau_{L-1}^{-1} \mu_k, \ga_{i_L}^{\vee} \rangle = \langle \gl, w_k^{-1}\tau_{L-1}\ga_{i_L}^{\vee}\rangle 
      = \langle \gl, (w_k^{-1} \gb)^{\vee} \rangle.
    \end{equation}
    By Lemma \ref{Lem:_preserve_sh}, if $\gb \in \gD_+ \setminus \lpishr_+$ we have $w_k^{-1} \gb \in \gD_+$ 
    for all $1 \le k \le N$, which implies the right hand side of (\ref{eq:adding}) is nonnegative for all $k$. 
    On the other hand if $\gb \in \gD_- \setminus \gD^{\sh}_-$,
    we have $w_k^{-1} \gb \in \gD_-$ and hence it is nonpositive for all $k$.
    Therefore the function $H_{i_L}^{\tau_{L-1}^{-1}\varphi(\pi)}$ is non-decreasing or non-increasing, 
    and then Lemma \ref{Lem: action_of_Weyl} completes the proof of the assertion. The claim is proved.\\[-0.2cm]
  
  Now the assertion (i) follows from the following claim:\\
  \\
  Claim 6.\ \ \ We have $\varphi(\te_0^{\sh}\pi)=S_\tau \te_j S_{\tau^{-1}} \varphi(\pi) \in \mB_0(\gl)$.\\[-0.2cm]
  
    Since $\varphi(\pi) = i_{\sh}(\pi) + \pi_{\gl'}$ and $\langle \gl', (\ga_0^{\sh})^{\vee} \rangle = 0$,
    we have from the condition (i) that
    \begin{align*}
      H_j^{\tau^{-1}\varphi(\pi)}(t) &= \langle \tau^{-1}\varphi(\pi)(t), \ga_j^{\vee} \rangle 
      = \langle \varphi(\pi)(t),  (\ga_0^{\sh})^{\vee} \rangle \\
      &= \langle \pi(t), (\ga_0^{\sh})^\vee \rangle = H_0^{\sh, \pi}(t) \ \ \ \ \ \ \text{for all} \ t \in [0,1],
    \end{align*}
    which implies $m_j^{\tau^{-1}\varphi(\pi)} = m_0^{\sh, \pi}$.
    Then it follows that
    \begin{align*}
      &\min\{t\in[0,1]\mid H^{\tau^{-1}\varphi(\pi)}_j = m_j^{\tau^{-1}\varphi(\pi)}\} = t_1 = \gs_{q_1} \ \text{and}\\
      &\max\{t\in[0,t_1]\mid H^{\tau^{-1}\varphi(\pi)}_j= m_j^{\tau^{-1}\varphi(\pi)} + 1\} = t_0 = \gs_{q_0}.
    \end{align*}
    Hence by Claim 5 and the definition of the root operator, we have
    \begin{align*}
      \te_j &S_{\tau^{-1}}\varphi(\pi) = \te_j(\tau^{-1}\mu_1,\dots,\tau^{-1}\mu_N; \ul{\gs}) \\
            &= (\tau^{-1} \mu_1,\dots, \tau^{-1}\mu_{q_0}, s_j\tau^{-1} \mu_{q_0+1},\dots,s_j\tau^{-1} \mu_{q_1}, \tau^{-1}
                                   \mu_{q_1+1}, \dots, \tau^{-1}\mu_N; \ul{\gs}) \\
            &= (\tau^{-1} \mu_1,\dots, \tau^{-1}\mu_{q_0}, \tau^{-1}s_0^{\sh}\mu_{q_0+1},\dots,\tau^{-1}s_0^{\sh} \mu_{q_1},
                                   \tau^{-1}\mu_{q_1+1}, \dots, \tau^{-1}\mu_N; \ul{\gs}),
    \end{align*}
    where the last equality follows from the condition (i). 
    Then similarly as the proof of Claim 5, it is proved that
    \[ S_{\tau} \te_j S_{\tau^{-1}}\varphi(\pi) =  (\mu_1,\dots, \mu_{q_0}, s_0^{\sh} \mu_{q_0+1},\dots,s_0^{\sh} 
       \mu_{q_1}, \mu_{q_1+1}, \dots, \mu_N; \ul{\gs}),
    \]
    which is equal to $\varphi(\te_0^{\sh}\pi)$ by (\ref{eq:j}).
\end{proof}

By Lemma \ref{Lem: simply-laced}, we have
\[ \sum_{\begin{smallmatrix} \pi \in \mB_0^{\sh}(\overline{\gl}) \\ \iota(\pi) \in \ol{P}\end{smallmatrix}} 
   e\big(\wt(\pi)\big) = \ch{\fh_d^{\sh}} \D^{\sh}(1,\overline{\gl})[0].
\]
Since $\wt\big(\varphi(\pi)\big) = i_{\sh}(\pi)(1) + \gl' = i_{\sh}\big(\wt(\pi)\big) + \gl'$,
this equality implies that
\[ \sum_{\begin{smallmatrix}\pi \in \mB_0^{\sh}(\overline{\gl}) \\ \iota(\pi) \in \ol{P}\end{smallmatrix}} 
    e\big(\wt (\varphi(\pi))\big) 
   = e(\gl')i_{\sh}\ch{\fhd^{\sh}}\D^{\sh}(1,\ol{\gl})[0].
\] 
On the other hand,  we have from Remark \ref{Rem: remark_of_character} that
\[ \sum_{\eta \in \mB(\gl)_{\cl}} e\big(\wt_{\hP}(\eta)\big) 
   = \sum_{\begin{smallmatrix}\pi' \in \mB_0(\gl) \\ \iota(\pi') \in P\end{smallmatrix}} e\big(\wt(\pi')\big).
\]
Hence we see that the following lemma, together with the above equalities, gives the proof of Proposition 
\ref{Prop: critical_prop}:

\begin{Lem}
  The map $\varphi$ induces a bijection 
  \[ \{ \pi \in \mB_0^{\sh}(\overline{\gl}) \mid \iota(\pi) \in \ol{P} \} \stackrel{\sim}{\to} 
     \{ \pi' \in \mB_0(\gl) \mid \iota(\pi') \in P, \ \wt(\pi') \in 
        \gl - Q^{\sh}_+ + \Z \gd \}.
  \]
\end{Lem}

\begin{proof}
  It is seen from Lemma \ref{Lem: critical_lemma1} that the image of the left hand side is contained in the right hand side,
  and the injectivity is obvious.
  Hence it suffices to show that these sets have the same number of elements.
  We have from Remark \ref{Rem: remark_of_character} that
  $\{ \pi' \in \mB_0(\gl)\mid \iota(\pi') \in P\} = \{ \pi_\eta \mid \eta \in \mB(\gl)_{\cl} \}$, 
  which implies 
  \begin{align*}
    \#\{ \pi' \in \mB_0(\gl) \mid \iota(\pi') &\in P, \ \wt(\pi') \in \gl - Q^{\sh}_+ + \Z \gd \}  \\
     &= \#\{ \eta \in \mB(\gl)_{\cl}\mid \wt(\eta) \in \cl(\gl-Q^{\sh}_+)\}.
  \end{align*}
  Write $\gl = \sum_{i \in I} \gl_i \varpi_{i}$.
  By Corollary \ref{Cor: h-weight}, we have
  \[ \sum_{\begin{smallmatrix}\eta \in \mB(\gl)_{\cl}\\ \wt(\eta) \in \cl(\gl-Q_+^{\sh})\end{smallmatrix}} 
     e\big(p \circ \wt(\eta)\big)
     = P_{\gl - Q_+^{\sh}} \prod_{i \in I} \ch{\fh} \D(1, \varpi_i)^{\gl_i}
  \]
  where $p\colon\hP_{\cl} \to P$ is the canonical projection, 
  and then
  \[ \sum_{\begin{smallmatrix}\eta \in \mB(\gl)_{\cl} \\ \wt(\eta) \in \cl(\gl-Q_+^{\sh})\end{smallmatrix}} 
     e\big(p \circ \wt(\eta)\big) = \prod_{i \in I} \left(P_{\varpi_i - Q_+^{\sh}}\ch{\fh}\D(1,\varpi_i)\right)^{\gl_i}
  \]
  follows since $\wt_{\fh}\D(1,\varpi_i) \subseteq \varpi_i - Q_+$.
  We have from Lemma \ref{Lem: short_Demazure} that
  \[ P_{\varpi_i - Q^{\sh}_+} \ch{\fh} \D(1,\varpi_i) = e\big(\varpi_i - i_{\sh}(\ol{\varpi}_i)\big)
     i_{\sh}\ch{\fh^{\sh}}\D^{\sh}(r,\ol{\varpi}_i),
  \]
  and from Lemma \ref{Lem: lemma_of_A} that 
  \[ \dim \D^{\sh}(r,\ol{\varpi}_i)  = \begin{cases} \dim V_{\fg^{\sh}}(\ol{\varpi}_i) & \text{if} \ i \in I^{\sh}, \\
                                                       1                               & \text{if} \ i \notin I^{\sh}.
                                                                             \end{cases}
  \]
  From these equations, we can see that
  \[ \#\{ \pi' \in \mB_0(\gl) \mid \iota(\pi') \in P,\ \wt(\pi') \in \gl - Q^{\sh}_+ + \Z \gd \}= \prod_{i \in I^{\sh}}  
     \dim V_{\fg^{\sh}}(\ol{\varpi}_i)^{\gl_i}.
  \]
  On the other hand, it follows that
  \begin{align*} \#\{ \pi \in \mB_0^{\sh}(\overline{\gl}) &\mid \iota(\pi) \in \ol{P} \} = \dim \D^{\sh}(1,\ol{\gl}) \\
     &= \prod_{i \in I^{\sh}}\dim \D^{\sh}(1, \ol{\varpi}_i)^{\gl_i}= \prod_{i \in I^{\sh}}  
     \dim V_{\fg^{\sh}}(\ol{\varpi}_i)^{\gl_i},
  \end{align*}
  where the first equality follows from Lemma \ref{Lem: simply-laced},
  the second one follows from \cite[Theorem 1]{MR2235341}, and the last one follows from Lemma \ref{Lem: lemma_of_A}.
  Hence it is proved that the two sets have the same number of elements. The lemma is proved.
\end{proof}

\subsection{Lower bound for the weight sum of $\mB(\gl)_{\cl}$}

Let $\gl \in P_+$.
Recall that, by Corollary \ref{Cor: Joseph's_result}, $\D^{\sh}(1,\overline{\gl})[0]$ has a $\Cgd^{\sh}$-module filtration
$0 = D_0 \subseteq D_1 \subseteq \dots \subseteq D_k = \D^{\sh}(1,\overline{\gl})[0]$ such that
\[ D_i / D_{i-1} \cong \D^{\sh}(r, \nu_i)[m_i] \ \text{for some} \ \nu_i \in \ol{P}_+ \ \text{and} \ m_i \in \Z_{\ge 0}.
\]
On the other hand by Corollary \ref{Cor: character_forumula}, there exist a sequence $\mu_1,\dots,\mu_\ell$ of elements of $P_+$ 
and a sequence $n_1,\dots,n_\ell$ of integers such that 
\begin{equation}\label{eq:equal}
  \sum_{\eta \in \mB(\gl)_{\cl}} e\big(\wt_{\hP}(\eta)\big) = \sum_{1 \le j \le \ell} \ch{\fh_d}\D(1,\mu_j)[n_j].
\end{equation}
Now we can state the main result in this section:

\begin{Prop}\label{Prop: Main_Theorem2}
  Let $\gl' = \gl -i_{\sh}(\ol{\gl})$.
  Then there exists a subset $S$ of $\{1,\ldots,\ell\}$ such that
  \[ \{ \mu_{j} + n_j\gd \mid j \in S \} = \big\{ \big(i_{\sh}(\nu_i) + \gl'\big) + m_i \gd \mid 1 \le i \le k\big\}
  \]
  as multisets.
\end{Prop}
 
\begin{proof} 
  Define a subset $S$ of $\{1,\ldots,\ell\}$ by 
  \[ S = \{1 \le j \le \ell \mid \mu_j \in \gl - Q_+^{\sh}\}.
  \]
  We shall prove that this $S$ is the required subset.
  Let $S^c = \{1,\ldots,\ell\} \setminus S$.
  Since $\wt_{\hP}\, B(\gl)_\cl\subseteq \gl - Q_+ + \Z\gd$ by Lemma \ref{Lem: contain_wt} (ii), 
  we have from (\ref{eq:equal}) that
  \[ \mu_j \in (\gl - Q_+) \setminus (\gl - Q_+^{\sh}) \ \ \ \text{for all} \ j \in S^c.
  \]
  From this, it is easily seen that
  \[ P_{\gl - Q_+^{\sh} + \Z\gd} \ch{\fhd} \D(1,\mu_j)[n_j] = 0 \ \ \ \text{if} \ j \in S^c.
  \]
  On the other hand if $j \in S$, we easily see from Lemma \ref{Lem: short_Demazure} that
  \begin{align*}
    P_{\gl - Q_+^{\sh}+\Z\gd} \ch{\fhd} \D(1,\mu_j)[n_j] 
    &= e\big(\mu_j -i_{\sh}(\ol{\mu}_j)\big)i_{\sh}\ch{\fhd^{\sh}}\D^{\sh}(r,\ol{\mu}_j)[n_j]\\
                                                   &= e(\gl')i_{\sh}\ch{\fhd^{\sh}}\D^{\sh}(r,\ol{\mu}_j)[n_j].
  \end{align*}
  Hence by applying $P_{\gl - Q_+^{\sh} + \Z\gd}$ to (\ref{eq:equal}) and using Proposition \ref{Prop: critical_prop},
  we obtain
  \[ e(\gl') i_{\sh} \ch{\fhd^{\sh}} \D^{\sh}(1, \ol{\gl})[0] 
   = e(\gl') \sum_{j \in S}i_{\sh}\ch{\fhd^{\sh}}\D^{\sh}(r,\ol{\mu}_j)[n_j].
  \]
  Note that we have
  \[ \ch{\fhd^{\sh}} \D^{\sh}(1, \ol{\gl})[0] = \sum_{1\le i \le k}\ch{\fhd^{\sh}}\D^{\sh}(r,\nu_i)[m_i].
  \]
  Hence by the linearly independence of the characters of lever $r$ Demazure modules, 
  it follows that $\{ \ol{\mu}_j + n_j \ol{\gd}\mid j \in S\} = \{ \nu_i + m_i \ol{\gd}\mid 1\le i\le k \}$ as multisets.
  Since each $\mu_j$ belongs to $\gl - Q_+^{\sh}$, the proposition is proved.
\end{proof}

\begin{Cor}\label{Cor: main_result2}
  Let $\gl \in P_+$, and set $\gl' = \gl - i_{\sh}(\ol{\gl})$. 
  Then we have
  \begin{align*}
    \sum_{\eta \in \mB(\gl)_{\cl}}& e\big(\wt_{\hP}(\eta)\big) \\ 
                                  & \ge \sum_{\nu \in \ol{P}_+, m \in \Z_{\ge 0}} \big(\D^{\sh}(1, \ol{\gl})[0] 
                                                    : \D^{\sh}(r, \nu)[m]\big)\ch{\fhd} \D\big(1,i_{\sh}(\nu) + \gl'\big)[m].
  \end{align*}
\end{Cor}

\section{Main theorems and corollaries}\label{Section: Main_Theorem}

In order to prove our main theorems, we need the following lemma:

\begin{Lem} \label{Lem:dim}
  For $\gl \in P_+$, we have
  \[ \dim W(\gl) \ge \# \mB(\gl)_{\cl}.
  \]
\end{Lem}

\begin{proof}
  Write $\gl = \sum_{i \in I} \gl_i \varpi_i$. 
  The inequality follows since
  \[ \dim W(\gl) \ge \prod_{i \in I} \dim \D(1,\varpi_i)^{\gl_i} = \prod_{i \in I} \# \mB(\varpi_i)_{\cl}^{\gl_i} 
     = \#\mB(\gl)_{\cl},
  \]
  where the first inequality follows from Lemma \ref{Lem: fusion_product} (ii), the second equality
  from Proposition \ref{Prop: fundamental}, and the third one from Theorem \ref{Thm: isom_of_finite_crystals} (i).
\end{proof}

Now we state the main theorems in this article.

\begin{Thm}\label{Thm: Main_Theorem}
  For $\gl \in P_+$, we have 
  \[ \ch{\fh_d}W(\gl) = \sum_{\eta \in \mB(\gl)_{\cl}} e\big(\wt_{\hP}(\eta)\big).
  \]
  Moreover, the both sides of this equality are equal to $\ch{\fhd}\D (1,\gl)[0]$ if $\fg$ is simply laced,
  and are equal to 
  \[ \sum_{\nu \in \ol{P}_+, m \in \Z_{\ge 0}} \big(\D^{\sh}(1, \ol{\gl})[0] : \D^{\sh}(r,\nu)[m]\big)
     \ch{\fhd}\D\big(1,i_{\sh}(\nu)+\gl'\big)[m]
  \]
  if $\fg$ is non-simply laced, where we set $\gl' = \gl - i_{\sh}(\ol{\gl})$.
\end{Thm}

\begin{proof}
  If $\fg$ is simply laced, the assertion follows from Theorem \ref{Thm: FL's_Thm} and Proposition \ref{Prop: simply_laced}.
  If $\fg$ is non-simply laced, it follows from Corollary \ref{Cor: Main_corollary1},
  \ref{Cor: main_result2} and Lemma \ref{Lem:dim}.
\end{proof}

Let $\gl \in P_+$. When $\fg$ is non-simply laced,
the Demazure module $\D^{\sh}(1, \ol{\gl})[0]$ has a $\Cgd^{\sh}$-module filtration
$0= D_0 \subseteq D_1 \subseteq \dots \subseteq D_k = \D^{\sh}(1, \ol{\gl})[0]$ such that
\[ D_i / D_{i-1} \cong \D^{\sh}(r, \nu_i)[m_i] \ \text{for some} \ \nu_i \in \ol{P}_+\ \text{and} \ m_i \in \Z_{\ge 0}
\] 
by Corollary \ref{Cor: Joseph's_result}. 
Then for each $1 \le i \le k$, we set $\mu_i = i_{\sh}(\nu_i) + \big(\gl - i_{\sh}(\ol{\gl})\big)$.
When $\fg$ is simply laced, we set $k = 1$, $\mu_1 = \gl$ and $m_1 = 0$.
Now we obtain the following result on the structure of the Weyl module $W(\gl)$:

\begin{Thm}\label{Thm: module_main_theorem}
  The Weyl module $W(\gl)$ has a $\Cgd$-module filtration 
  $0 = W_0 \subseteq W_1 \subseteq \cdots \subseteq W_k = W(\gl)$ such that 
  each subquotient $W_i /W_{i-1}$ is isomorphic to the Demazure module $\D(1,\mu_i)[m_i]$.
\end{Thm}

\begin{proof}
  The assertion for simply laced $\fg$ is just Theorem \ref{Thm: FL's_Thm}.
  Proposition \ref{Prop: Main_Theorem1} together with Theorem \ref{Thm: Main_Theorem} 
  implies the assertion for non-simply laced $\fg$.
\end{proof}

Next we state the following theorem about the crystal $\mB(\gl)_{\cl}$:

\begin{Thm}\label{Thm: crystal_main_theorem}
  There exists an isomorphism
  \[ \gk\colon\mcB(\gL_0) \otimes \mB(\gl)_{\cl} \stackrel{\sim}{\to}
     \bigoplus_{\eta_0 \in \mB(\gl)_{\cl}^{\gL_0}} \mcB\big(\gL_0 + \pi_{\eta_0}(1)\big)
  \]
  of $U_q'(\hfg)$-crystals whose restriction to $b_{\gL_0} \otimes \mB(\gl)_{\cl}$
  preserves the $\hP$-weights, and we have
  \[ \gk\big(b_{\gL_0} \otimes \mB(\gl)_{\cl}\big) = \coprod_{1\le i \le k} \mcB(1, \mu_i)[m_i]. 
  \]
\end{Thm}

\begin{proof}
  This follows from Proposition \ref{Prop: disjoint} and Theorem \ref{Thm: Main_Theorem}.  
\end{proof}

Next we introduce some corollaries of our main theorems.
The following corollary was previously shown for $\mathfrak{sl}_n$ in \cite{MR2271991},
and for simply laced $\fg$ in \cite{MR2323538}:

\begin{Cor}\label{Cor: Main_Corollary1}
  Let $\gl \in P_+$. \\
  {\normalfont(i)} If $\gl= \sum_{i \in I}\gl_i \varpi_i$, then
    \[ \dim W(\gl) = \prod_{i \in I} \dim W(\varpi_i)^{\gl_i}.
    \]
  {\normalfont(ii)} Let $\gl_1, \dots, \gl_\ell \in P_+$ be elements satisfying $\gl = \gl_1 + \dots + \gl_\ell$.
  Then for arbitrary pairwise distinct complex numbers $c_1,\dots,c_\ell$, we have 
  \[ W(\gl) \cong W(\gl_1)_{c_1} *\dots * W(\gl_\ell)_{c_\ell}
  \]
  as $\Cgd$-modules.
\end{Cor}
 
\begin{proof}
  By Theorem \ref{Thm: Main_Theorem} and Theorem \ref{Thm: isom_of_finite_crystals} (i), we have
  \[ \dim W(\gl) = \# \mB(\gl)_{\cl} = \prod_{i \in I} \# \mB(\varpi_i)_{\cl}^{\gl_i} 
     = \prod_{i \in I} \dim W(\varpi_i)^{\gl_i}.
  \]
  Hence the assertion (i) follows.
  By the same way as \cite[Lemma 5]{MR2323538}, we can show that
  there exists a surjective homomorphism from $W(\gl)$ to $W(\gl_1)_{c_1} *\dots * W(\gl_\ell)_{c_\ell}$.
  Since this is an isomorphism by (i), the assertion (ii) follows.
\end{proof}

Before stating the next corollary, we prepare some notation.
For a $\Cgd$-module $M$ and $c \in \C$, we set $M_c = \{ v \in M \mid d.v = cv \}$, which is obviously a $\fg$-module.
For $\mu \in P_+$, we set
\[ [M:V_{\fg}(\mu)]_q = \sum_{c \in \C}\,[M_c: V_{\fg}(\mu)]q^c,
\]
where $[M_c:V_{\fg}(\mu)]$ denotes the multiplicity of $V_{\fg}(\mu)$ in $M_c$.
Let $\mathbf{i} = (i_1,\dots,i_\ell)$ be a sequence of elements of $I$, and put
\[ W_{\mathbf{i}} = W_q(\varpi_{i_1}) \otimes \dots \otimes W_q(\varpi_{i_\ell}),
   \ \ \ \mcB_{\mathbf{i}} = \mcB\big(W_q(\varpi_{i_1})\big) \otimes \dots \otimes 
   \mcB\big(W_q(\varpi_{i_\ell})\big).
\]
As stated in \cite[\S 5,1]{MR2474320}, $W_q(\varpi_i)$ is the Kirillov-Reshetikhin module
$W_s^{(i)}$ (see \cite{MR1745263,MR1903978}) with $s = 1$.
Hence for $\mu \in P_+$, we can define the fermionic form $M(W_{\mathbf{i}}, \mu, q)$ and 
the classically restricted one-dimensional sum
$X(\mcB_{\mathbf{i}}, \mu,q)$ as in \cite{MR1745263} and \cite{MR1903978}.

For $i \in I$, let $\mathrm{KR}(\varpi_i)$ denote the Kirillov-Reshetikhin module
for $\Cgd$ associated with $\varpi_i$, which was defined in \cite{MR2238884} in terms of generators and relations.
By comparing the defining relations, we can easily see that 
\begin{equation}\label{eq:isomorphism}
  W(\varpi_i) \cong \mathrm{KR}(\varpi_i).
\end{equation}
Note that we have for each $i \in I$ that
\[ \dim \mathrm{KR}(\varpi_i) = \dim W(\varpi_i) = \# \mB(\varpi_i)_\cl = \dim W_q(\varpi_i)
\]
by Theorem \ref{Thm: Main_Theorem} and \ref{Thm: isom_of_finite_crystals} (ii).
Then this equality implies, as shown in \cite{MR2428305}, that
\[ M(W_{\mathbf{i}}, \mu, q) = [ \mathrm{KR}(\varpi_{i_1})_{c_1} * \dots * 
\mathrm{KR}(\varpi_{i_\ell})_{c_\ell}: 
 V_{\fg}(\mu)]_{q^{-1}}
\]
for arbitrary pairwise distinct complex numbers $c_1,\dots,c_\ell$.
Moreover by setting $\gl = \sum_{1 \le j \le \ell} \varpi_{i_j}$, we have from (\ref{eq:isomorphism}) and Corollary 
\ref{Cor: Main_Corollary1} (ii) that 
\[ W(\gl) \cong \mathrm{KR}(\varpi_{i_1})_{c_1} * \dots * \mathrm{KR}(\varpi_{i_\ell})_{c_\ell}.
\]
Hence we obtain the following equality:
\begin{equation} \label{eq:M=ch}
  M(W_{\mathbf{i}}, \mu, q) = [W(\gl) : V_{\fg}(\mu)]_{q^{-1}}.
\end{equation}
In \cite[Corollary 5.1.1]{MR2474320}, the authors showed that
\begin{equation} \label{eq: final}
     \sum_{\begin{smallmatrix}\eta \in \mB(\gl)_{\cl} \\ \te_j \eta = \0 \ (j \in I) \\ 
     \eta(1) = \cl(\mu) \end{smallmatrix}} q^{\Deg(\eta)}
     =q^{-D_{\mathbf{i}}^{\mathrm{ext}}}X(\mcB_{\mathbf{i}}, \mu,q),
\end{equation}
where $D_{\mathbf{i}}^{\mathrm{ext}}$ is a certain constant defined in 
\cite[Subsection 4.1]{MR2474320}.
Combining these results with our theorems, we obtain the following corollary,
which implies a special case of the so-called $X=M$ conjecture (see \cite{MR1745263,MR1903978}):

\begin{Cor}\label{Cor: final_Cor}
  Let $\gl \in P_+$, $\mathbf{i} = (i_1,\dots,i_\ell)$ be a sequence of elements of $I$ 
  such that $\gl = \sum_{1 \le j\le \ell} \varpi_{i_j}$ and $\mu \in P_+$.
  Then we have
  \begin{align*} M(W_{\mathbf{i}},\mu,q) &= q^{-D_{\mathbf{i}}^{\mathrm{ext}}}X(\mcB_{\mathbf{i}}, \mu,q) \\
                 &= [W(\gl):V_{\fg}(\mu)]_{q^{-1}}.
  \end{align*}
\end{Cor}

\begin{proof}
  From (\ref{eq:M=ch}) and (\ref{eq: final}), it suffices to show the equality
  \begin{equation} \label{eq:final_id}
    [W(\gl): V_{\fg}(\mu)]_{q^{-1}} =      
     \sum_{\begin{smallmatrix}\eta \in \mB(\gl)_{\cl} \\ \te_j \eta = \0 \ (j \in I) \\ \eta(1) = \cl(\mu) 
     \end{smallmatrix}} q^{\Deg(\eta)}.
  \end{equation}
  We obviously have 
  \[ \ch{\fhd}W(\gl) = \sum_{\nu \in P_+} [W(\gl):V_{\fg}(\nu)]_{q}\ch{\fh}V_{\fg}(\nu),
  \]
  where we identify $e(\gd) = q$.
  On the other hand, since $\mB(\gl)_{\cl}$ is isomorphic to the crystal basis of a finite-dimensional 
  $U_q'(\hfg)$-module (in particular $U_q(\fg)$-module)
  and $\te_{j}, \tf_{j}$ for $j \in I$ preserve the degree function, we can see that
  \[ \sum_{\eta \in \mB(\gl)_{\cl}} e\big(\wt_{\hP}(\eta)\big) = 
     \sum_{\nu \in P_+}
     \sum_{\begin{smallmatrix}\eta \in \mB(\gl)_{\cl} \\ \te_j \eta 
     = \0 \ (j \in I) \\ \eta(1) = \cl(\nu) \end{smallmatrix}} q^{-\Deg(\eta)} \ch{\fh}V_{\fg}(\nu),
  \]
  where we identify $e(\gd) = q$.
  Hence by Theorem \ref{Thm: Main_Theorem} and the linearly independence of the characters of 
  finite-dimensional irreducible $\fg$-modules, (\ref{eq:final_id}) is proved.
\end{proof}
 
\ \\
\textbf{Acknowledgements:} The author is very grateful to his supervisor Hisayosi Matumoto
for his constant encouragement and patient guidance.
He is also grateful to R.\ Kodera, S.\ Naito, Y.\ Saito and V.\ Chari for their helpful comments on this work.
He is supported by the Japan Society for the Promotion of Science Research
Fellowships for Young Scientists.



\def\cprime{$'$}

\end{document}